\documentclass[11pt,draft,reqno]{amsart}
\usepackage{amsmath,amsfonts,amssymb}
\usepackage[cp1251]{inputenc}
\usepackage[english]{babel}
\usepackage{mathrsfs}

\def\le{\leqslant}
\def\ge{\geqslant}

\theoremstyle{theorem}
\newtheorem{theorem}{Theorem}[section]

\newtheorem{lemma}[theorem]{Lemma}
\newtheorem{condition}[theorem]{Condition}
\newtheorem{remark}[theorem]{Remark}
\newtheorem{corollary}[theorem]{Corollary}

\numberwithin{equation}{section}

\theoremstyle{plain}
\newtoks\thehProclaim
\newtheorem*{Proclaim}{\the\thehProclaim}

\usepackage[height = 23.5cm, a4paper, hmargin={2.5cm , 2.5cm}]{geometry}

\begin{document}

\title[Homogenization of nonstationary periodic Maxwell system]
{Homogenization of nonstationary \\ periodic Maxwell system\\ in the case of constant permeability}

\author{M.~A.~Dorodnyi, T.~A.~Suslina}

\address{St.~Petersburg State University
\\
Universitetskaya nab. 7/9
\\
St.~Petersburg, 199034, Russia}

\email{mdorodni@yandex.ru}

\email{t.suslina@spbu.ru}




\thanks{Supported by Russian Science Foundation (project 17-11-01069).}

\begin{abstract}
In $L_2({\mathbb R}^3;{\mathbb C}^3)$, we consider a selfadjoint operator ${\mathcal L}_\varepsilon$,
$\varepsilon >0$, given by the differential expression 
$\mu_0^{-1/2}\operatorname{curl} \eta(\mathbf{x}/\varepsilon)^{-1} \operatorname{curl} \mu_0^{-1/2}
- \mu_0^{1/2}\nabla \nu(\mathbf{x}/\varepsilon) \operatorname{div} \mu_0^{1/2}$, where 
$\mu_0$ is a constant positive matrix,
a matrix-valued function $\eta(\mathbf{x})$  and a real-valued function $\nu(\mathbf{x})$ are periodic with respect to some lattice, positive definite and bounded.
We study the behavior of the operator-valued functions  $\cos (\tau {\mathcal L}_\varepsilon^{1/2})$ and
${\mathcal L}_\varepsilon^{-1/2} \sin (\tau {\mathcal L}_\varepsilon^{1/2})$
for $\tau \in {\mathbb R}$ and small  $\varepsilon$. It is shown that these operators converge to the corresponding operator-valued functions of the operator ${\mathcal L}^0$ in the norm of operators acting from the Sobolev space $H^s$ (with a suitable $s$) to $L_2$. Here ${\mathcal L}^0$ is the effective operator with constant coefficients. 
Also, an approximation with corrector in the  $(H^s \to H^1)$-norm for the operator ${\mathcal L}_\varepsilon^{-1/2} \sin (\tau {\mathcal L}_\varepsilon^{1/2})$ is obtained.  We prove error estimates and study the sharpness of the results regarding the type of the operator norm and regarding the dependence of the estimates on $\tau$. The results are applied to homogenization of the Cauchy problem for the nonstationary Maxwell system in the case where the magnetic permeability is equal to $\mu_0$, and the dielectric permittivity is given by the matrix  $\eta(\mathbf{x}/\varepsilon)$.
\end{abstract}

\maketitle

\section*{Introduction}
\setcounter{section}{0}
\setcounter{equation}{0}

\subsection{Operator error estimates}
The paper concerns homogenization theory of periodic differential operators (DOs). First of all, 
we mention the books  \cite{BeLP, BaPa, ZhKO}.

In a series of papers \cite{BSu1,BSu2,BSu3} by Birman and Suslina, an operator-theoretic  (spectral) approach to homogenization problems was developed. In $L_2({\mathbb R}^d; {\mathbb C}^n)$, a wide class of 
matrix strongly elliptic second order DOs ${\mathcal A}_\varepsilon$ was studied. 
The operator ${\mathcal A}_\varepsilon$ is given by  
\begin{equation}
\label{A_eps}
{\mathcal A}_\varepsilon = b(\mathbf{D})^* g(\mathbf{x}/\varepsilon) b(\mathbf{D}),  \quad \varepsilon >0,
\end{equation}
 where $g(\mathbf{x})$ is a bounded and positive definite $(m\times m)$-matrix-valued function
 periodic with respect to some lattice \hbox{$\Gamma \subset {\mathbb R}^d$}, and $b(\mathbf{D}) = \sum_{l=1}^d b_l D_l$ is a first order DO. Here $b_l$ are constant $(m \times n)$-matrices. It is assumed that $m \ge n $ and the symbol  $b(\boldsymbol{\xi})$ has maximal rank.

In \cite{BSu1}, it was shown that the resolvent  $({\mathcal A}_\varepsilon +I)^{-1}$ converges in the  $(L_2 \to L_2)$-operator norm to the resolvent of the effective operator ${\mathcal A}^0$, and 
\begin{equation}
\label{est_A_eps}
 \bigl\| ({\mathcal A}_\varepsilon +I)^{-1} - ({\mathcal A}^0+I)^{-1} \bigr\|_{L_2(\mathbb{R}^d)\to L_2(\mathbb{R}^d)} \le C \varepsilon.
\end{equation}
The effective operator is given by ${\mathcal A}^0= b(\mathbf{D})^* g^0  b(\mathbf{D})$, where $g^0$ is a constant positive matrix called the \textit{effective} matrix. In \cite{Su1}, a similar result was obtained for the parabolic semigroup:
\begin{equation}
\label{parab_est_A_eps}
 \bigl\| e^{- \tau {\mathcal A}_\varepsilon} - e^{-\tau {\mathcal A}^0} \bigr\|_{L_2(\mathbb{R}^d)\to L_2(\mathbb{R}^d)} \le C(\tau) \varepsilon,\quad \tau >0.
\end{equation}
Estimates  \eqref{est_A_eps} and \eqref{parab_est_A_eps} are order-sharp. Such inequalities are called  \textit{operator error estimates} in homogenization theory.

A different approach to operator error estimates (the shift method) was developed by Zhikov and Pastukhova. In \cite{Zh2,ZhPas1,ZhPas2}, estimates of the form \eqref{est_A_eps}, \eqref{parab_est_A_eps} were obtained for the operators of acoustics and elasticity. Further results were discussed in a survey \cite{ZhPas3}.

The operator error estimates for the nonstationary Schr{\"o}dinger-type equations and hyperbolic equations were studied in  \cite{BSu4} and in the recent works  \cite{Su3, Su4, M1, M2, DSu1, DSu2, D, DSu4}. 
In operator terms, the behavior of the operator-valued functions  
 $e^{-i \tau {\mathcal A}_\varepsilon}$,  
$\cos (\tau {\mathcal A}_\varepsilon^{1/2})$, 
${\mathcal A}_\varepsilon^{-1/2} \sin (\tau {\mathcal A}_\varepsilon^{1/2})$, $\tau \in \mathbb{R}$, was investigated. It turned out that the nature of the results differs from the case of elliptic and parabolic equations: the type of the operator norm must be changed. 

Let us dwell on the hyperbolic case. In \cite{BSu4}, the following sharp order estimate was proved:
\begin{equation}
\label{est_cos_A_eps}
\bigl\| \cos (\tau {\mathcal A}_\varepsilon^{1/2})  - \cos (\tau ({\mathcal A}^0)^{1/2}) \bigr\|_{H^2(\mathbb{R}^d)\to L_2(\mathbb{R}^d)} \le C(1+  |\tau|) \varepsilon.
\end{equation}
A similar result for the operator ${\mathcal A}_\varepsilon^{-1/2} \sin (\tau {\mathcal A}_\varepsilon^{1/2})$
together with approximation in the energy norm was obtained in \cite{M1, M2}:
\begin{align}
\label{est_sin_A_eps}
\bigl\|  {\mathcal A}_\varepsilon^{-1/2} \sin (\tau {\mathcal A}_\varepsilon^{1/2})  - 
({\mathcal A}^0)^{-1/2} \sin (\tau ({\mathcal A}^0)^{1/2}) \bigr\|_{H^1(\mathbb{R}^d)\to L_2(\mathbb{R}^d)} \le 
C(1  + |\tau|) \varepsilon,
\\
\label{est_sin_A_eps2}
 \bigl\|  {\mathcal A}_\varepsilon^{-1/2} \sin (\tau {\mathcal A}_\varepsilon^{1/2})  - 
({\mathcal A}^0)^{-1/2} \sin (\tau ({\mathcal A}^0)^{1/2}) - \varepsilon K(\varepsilon;\tau) \bigr\|_{H^2(\mathbb{R}^d)\to H^1(\mathbb{R}^d)} \le C(1  + |\tau|) \varepsilon.
\end{align}
Here $K(\varepsilon;\tau)$ is the corresponding corrector.

In \cite{DSu1, DSu2, DSu4}, it was shown that in the general case the results 
\eqref{est_cos_A_eps}--\eqref{est_sin_A_eps2} are sharp both regarding the type of the operator norm and regarding the dependence of estimates on $\tau$ (it is impossible to replace $(1+|\tau|)$ on the right by  
$(1+|\tau|)^\alpha$ with $\alpha<1$). On the other hand, under some additional assumptions the results admit the following improvement:
\begin{gather}
\label{usilenie_est_cos_A_eps}
 \bigl\| \cos (\tau {\mathcal A}_\varepsilon^{1/2})  - \cos (\tau ({\mathcal A}^0)^{1/2}) \bigr\|_{H^{3/2}(\mathbb{R}^d)\to L_2(\mathbb{R}^d)} \le C(1+ |\tau|)^{1/2} \varepsilon,
\\
\label{usilenie_est_sin_A_eps}
 \bigl\| {\mathcal A}_\varepsilon^{-1/2} \sin (\tau {\mathcal A}_\varepsilon^{1/2})  - 
({\mathcal A}^0)^{-1/2} \sin (\tau ({\mathcal A}^0)^{1/2}) \bigr\|_{H^{1/2}(\mathbb{R}^d)\to L_2(\mathbb{R}^d)} 
\le C(1+ |\tau|)^{1/2} \varepsilon,
\\
\label{usilenie_est_sin_A_eps2}
 \bigl\| {\mathcal A}_\varepsilon^{-1/2} \sin (\tau {\mathcal A}_\varepsilon^{1/2})  - 
({\mathcal A}^0)^{-1/2} \sin (\tau ({\mathcal A}^0)^{1/2}) - \varepsilon K(\varepsilon;\tau) 
 \bigr\|_{H^{3/2}(\mathbb{R}^d)\to H^1(\mathbb{R}^d)} \le C(1+ |\tau|)^{1/2} \varepsilon.
\end{gather}
The additional assumptions are formulated in terms of the spectral characteristics of the operator 
${\mathcal A}= b(\mathbf{D})^* g(\mathbf{x}) b(\mathbf{D})$ at the bottom of the spectrum.
Similar results for the exponential $e^{-i \tau {\mathcal A}_\varepsilon}$ were previously
obtained in  \cite{Su3, Su4, D}.

\subsection{Main results}
In the present paper, we apply the results of  \cite{BSu4, M1,M2,DSu2, DSu4} to the \textit{model operator of electrodynamics} acting in $L_2({\mathbb R}^3;{\mathbb C}^3)$ and given by the expression 
\begin{equation}
\label{L_eps_intr}
{\mathcal L}_\varepsilon = \mu_0^{-1/2} \operatorname{curl} \eta(\mathbf{x}/\varepsilon)^{-1} \operatorname{curl} \mu_0^{-1/2} - \mu_0^{1/2} \nabla \nu(\mathbf{x}/\varepsilon) \operatorname{div} \mu_0^{1/2}, \quad \varepsilon >0.
\end{equation}
Here $\mu_0$ is a constant positive matrix,  $\eta(\mathbf{x})$ is a matrix-valued function, and $\nu(\mathbf{x})$ is a real-valued function. It is assumed that $\eta(\mathbf{x})$ and  $\nu(\mathbf{x})$ are periodic, bounded and positive definite.  The operator \eqref{L_eps_intr} is a particular case of the operator \eqref{A_eps} with $m=4$ and $n=3$. The specific feature is that  the operator ${\mathcal L}_\varepsilon$ is reduced by the orthogonal decomposition of $L_2({\mathbb R}^3;{\mathbb C}^3)$ into  the divergence-free and the gradient subspaces (the Weyl decomposition). We are mainly interested in the divergence-free part ${\mathcal L}_{J,\varepsilon}$ 
of the operator ${\mathcal L}_\varepsilon$. For ${\mathcal L}_{J,\varepsilon}$ we obtain estimates of the form  \eqref{est_cos_A_eps}--\eqref{est_sin_A_eps2}. We show that in the general case these results cannot be improved. On the other hand, under some additional assumptions we obtain estimates of the form  \eqref{usilenie_est_cos_A_eps}--\eqref{usilenie_est_sin_A_eps2}. Some examples of both situations are discussed.

The results are applied to homogenization of the Cauchy problem for the nonstationary Maxwell system in the case where the magnetic permeability is equal to  $\mu_0$ and the dielectric permittivity is given by the matrix  $\eta(\mathbf{x}/\varepsilon)$.

Some partial results in this direction were obtained in the previous paper \cite{DSu3} by the authors (in the case where $\mu_0 = {\mathbf 1}$).

The method is based on the scaling transformation, the Floquet--Bloch theory, and the analytic perturbation theory. An important role is played by the spectral characteristics of the operator ${\mathcal L}$ (given by  \eqref{L_eps_intr} with $\varepsilon=1$) at the bottom of the spectrum.
We also rely  on the papers \cite{Su2, BSu-FAA, Su-AA18} about homogenization of the stationary periodic Maxwell system. 

\subsection{Plan of the paper}
In \S 1, we introduce the operator $\mathcal L$ acting in $L_2({\mathbb R}^3; {\mathbb C}^3)$; describe 
its reduction by  the Weyl decomposition; describe the expansion of $\mathcal L$ in the direct integral of the operators ${\mathcal L}(\mathbf{k})$ acting in  $L_2({\Omega}; {\mathbb C}^3)$ (where $\Omega$ is the cell of the lattice $\Gamma$) and depending on the parameter ${\mathbf k} \in {\mathbb R}^3$ (the quasimomentum).  In \S 2, the effective characteristics of the operator $\mathcal L$ are introduced. In \S 3,  main results of the paper  on homogenization of the operators ${\mathcal L}_\varepsilon$ and ${\mathcal L}_{J,\varepsilon}$ are obtained.
In \S 4, we apply the results to homogenization of the solutions of the Cauchy problem for the nonstationary Maxwell system. 

\subsection{Notation}
Let $\mathfrak{H}$ and $\mathfrak{H}_*$ be complex separable Hilbert spaces. By $\Vert \cdot \Vert _{\mathfrak{H}}$ we denote the norm in $\mathfrak{H}$; the symbol $\Vert \cdot \Vert _{\mathfrak{H}\rightarrow \mathfrak{H}_*}$ denotes the norm of a linear continuous operator acting from $\mathfrak{H}$ to $\mathfrak{H}_*$.

The inner product and the norm in $\mathbb{C}^n$ are denoted by $\langle \cdot ,\cdot \rangle$ and $\vert \cdot \vert$, respectively, $\mathbf{1}_n = \mathbf{1}$ is the unit $(n\times n)$-matrix.
If $a$ is a matrix of size $n\times n$, then $\vert a\vert$ stands for the norm of $a$ viewed as 
an operator in $\mathbb{C}^n$.
We denote  $\mathbf{x}=(x_1,x_2,x_3)\in \mathbb{R}^3$, $iD_j=\partial /\partial x_j$, $j=1,2,3$, $\mathbf{D}=-i\nabla =(D_1,D_2,D_3)$.

The class $L_2$ of $\mathbb{C}^n$-valued functions in a domain $\mathcal{O}\subset \mathbb{R}^d$  is denoted by $L_2(\mathcal{O};\mathbb{C}^n)$. The Sobolev classes of  $\mathbb{C}^n$-valued functions in a domain  $\mathcal{O}$ are denoted by $H^s(\mathcal{O};\mathbb{C}^n)$. For $n=1$, we write simply $L_2(\mathcal{O})$, $H^s(\mathcal{O})$, but sometimes we use such simple notation also for the spaces of vector-valued or matrix-valued functions.

\subsection{Acknowledgement} M.~A.~Dorodnyi is a Young Russian Mathematics award winner and would like to thank its sponsors and jury.

\section{The model second order operator}
\label{Section Preliminaries}
\setcounter{section}{1}
\setcounter{equation}{0}

\subsection{Lattices. The Gelfand transformation}\label{Subsection Lattices}
Let $\Gamma$ be a lattice in $\mathbb{R}^3$ generated by the basis $\mathbf{a}_1, \mathbf{a}_2, \mathbf{a}_3$:
\begin{equation*}
\Gamma=\biggl\{ \mathbf{a}\in \mathbb{R}^3 \,:\  \mathbf{a}=\sum \limits _{j=1}^3 q_j \mathbf{a}_j,\  q _j\in\mathbb{Z}\biggr\}.
\end{equation*}
Let $\Omega$ be the elementary cell of the lattice $\Gamma$:
\begin{equation*}
\Omega = \biggl \{ \mathbf{x}\in\mathbb{R}^3 \ :\, \mathbf{x}=\sum \limits _{j=1}^3 \xi_j\mathbf{a}_j,\;
0 <\xi_j < 1
 \biggr \}.
\end{equation*}
The basis  $\mathbf{b}_1, \mathbf{b}_2,\mathbf{b}_3\in \mathbb{R}^3$ dual to  $\mathbf{a}_1, \mathbf{a}_2,\mathbf{a}_3$ is defined by the relations $\langle \mathbf{b}_j,\mathbf{a}_i\rangle =2\pi \delta _{ji}$. This basis generates the lattice $\widetilde{\Gamma}$ dual to ${\Gamma}$.  Let $\widetilde{\Omega}$ be the central Brillouin zone of the lattice $\widetilde{\Gamma}$:
\begin{equation*}
\widetilde{\Omega}=\bigl \{ \mathbf{k}\in\mathbb{R}^3:\ \vert \mathbf{k}\vert <\vert \mathbf{k}-\mathbf{b}\vert ,\ 0\neq \mathbf{b}\in \widetilde{\Gamma}\bigr\}.
\end{equation*}
Let $r_0$ be the radius of the ball inscribed in  $\mathrm{clos}\,\widetilde{\Omega}$, i.~e.,
$2r_0=\min_{0\ne {\mathbf b} \in \widetilde{\Gamma}} |{\mathbf b}|$.

For $\Gamma$-periodic measurable matrix-valued functions,  we use the following notation:
 $$
f^\varepsilon (\mathbf{x}):=f(\mathbf{x}/\varepsilon), \ \varepsilon >0;\quad
\overline{f}:=\vert \Omega\vert ^{-1}\int _\Omega f(\mathbf{x})\,d\mathbf{x},\quad
  \underline{f}:=\left(\vert \Omega\vert ^{-1}\int _\Omega f(\mathbf{x})^{-1}\,d\mathbf{x}\right)^{-1}.
$$
 In the definition of  $\overline{f}$ it is assumed that  $f\in L_{1,\mathrm{loc}}(\mathbb{R}^3)$, and in the definition of
 $\underline{f}$ it is assumed that  $f(\mathbf{x})$ is a square  nondegenerate matrix such that  $f^{-1}\in L_{1,\mathrm{loc}}(\mathbb{R}^3)$.

Let $\widetilde{H}^1(\Omega;\mathbb{C}^n)$ be the subspace of 
$H^1(\Omega;\mathbb{C}^n)$ consisting of functions whose  $\Gamma$-periodic extension to $\mathbb{R}^3$ belongs to $H^1_{\textnormal{loc}}(\mathbb{R}^3;\mathbb{C}^n)$.

Now, we introduce the  \textit{Gelfand transformation} $\mathcal{U}$. First,  $\mathcal{U}$ is defined on the Schwartz class by the following relation:
	\begin{equation*}
\begin{split}	
(\mathcal{U}\mathbf{f})(\mathbf{k}, \mathbf{x}) =
\widetilde{\mathbf{f}}(\mathbf{k}, \mathbf{x}) :=
|\widetilde{\Omega}|^{-1/2} \sum_{\mathbf{a} \in \Gamma} e^{-i \langle \mathbf{k}, \mathbf{x} + \mathbf{a} \rangle} \mathbf{f}(\mathbf{x} + \mathbf{a}),
\quad \mathbf{f} \in \mathcal{S}(\mathbb{R}^3; \mathbb{C}^3), \  \mathbf{x} \in \Omega, \  \mathbf{k} \in \widetilde{\Omega}.
\end{split}
	\end{equation*}
	Next, it is extended up to unitary transformation 
	\begin{equation*}
	\mathcal{U}:  L_2(\mathbb{R}^3; \mathbb{C}^3) \to \int_{\widetilde{\Omega}} \oplus L_2(\Omega; \mathbb{C}^3) \, d \mathbf{k} =: \mathcal{K}.
\end{equation*}
The relation ${\mathbf f} \in H^1(\mathbb{R}^3;\mathbb{C}^3)$ is equivalent to the fact that  $\widetilde{\mathbf{f}}(\mathbf{k}, \cdot) \in \widetilde{H}^1(\Omega;\mathbb{C}^3)$ for almost all \hbox{$\mathbf{k} \in \widetilde{\Omega}$}~and
$$
\int_{\widetilde{\Omega}} \int_\Omega \left( |(\mathbf{D} + \mathbf{k}) \widetilde{\mathbf{f}}(\mathbf{k}, \mathbf{x}) |^2 +
| \widetilde{\mathbf{f}}(\mathbf{k}, \mathbf{x}) |^2 \right) \, d\mathbf{x}\, d \mathbf{k} < \infty.
$$
Under the transform $\mathcal{U}$, the operator in $L_2(\mathbb{R}^3;\mathbb{C}^3)$ acting as  multiplication by a bounded periodic matrix-valued function 
 turns into multiplication by the same function on the  fibers of the direct integral $\mathcal K$.
Action of the first order DO $b(\mathbf{D})$ on
${\mathbf f} \in H^1(\mathbb{R}^3;\mathbb{C}^3)$ turns into  action of the operators 
 $b({\mathbf D}+{\mathbf k})$ on $\widetilde{\mathbf f}(\mathbf{k},\cdot) \in \widetilde{H}^1(\Omega;\mathbb{C}^3)$ on the fibers of the direct integral.

\subsection{The operator $\mathcal{L}$}\label{Subsection Operator L}
Let $\mu_0$ be a symmetric positive $(3 \times 3)$-matrix with real entries. 
Suppose that $\eta(\mathbf{x})$ is a symmetric $(3 \times 3)$-matrix-valued function in ${\mathbb R}^3$ with real entries
and  $\nu(\mathbf{x})$ is a real-valued function in ${\mathbb R}^3$. We assume that $\eta(\mathbf{x})$ and 
$\nu(\mathbf{x})$ are periodic with respect to the lattice $\Gamma$ and such that 
	\begin{align}
	\label{eta}
	\eta(\mathbf{x}) &> 0; \quad \eta, \eta^{-1} \in L_\infty;\\
	\label{nu}
	\nu(\mathbf{x}) &> 0; \quad \nu, \nu^{-1} \in L_\infty.
	\end{align}
In $L_2(\mathbb{R}^3; \mathbb{C}^3)$, we consider the operator $\mathcal{L}$ formally given by the differential expression 
	\begin{equation}
	\label{L}
	\mathcal{L} = \mu_0^{-1/2}\operatorname{curl} \eta(\mathbf{x})^{-1} \operatorname{curl} \mu_0^{-1/2} 
	- \mu_0^{1/2} \nabla \nu(\mathbf{x}) \operatorname{div} \mu_0^{1/2}.
	\end{equation}
	The operator~(\ref{L}) can be represented as  $\mathcal{L} = b(\mathbf{D})^* g(\mathbf{x}) b(\mathbf{D})$, where 
	\begin{equation*}
	b(\mathbf{D}) = \begin{pmatrix}
	- i \operatorname{curl} \mu_0^{-1/2} \\
	- i \operatorname{div} \mu_0^{1/2}
	\end{pmatrix}, \quad g(\mathbf{x}) = \begin{pmatrix}
	\eta(\mathbf{x})^{-1} & 0 \\
	0 & \nu(\mathbf{x})\end{pmatrix}.
	\end{equation*}
	The symbol $b(\boldsymbol{\xi})$ of the operator $b({\mathbf D})$ is given by 
	\begin{equation}
	\label{symbol}
	b(\boldsymbol{\xi}) = \begin{pmatrix} r(\boldsymbol{\xi}) \mu_0^{-1/2} \\ \boldsymbol{\xi}^t \mu_0^{1/2}\end{pmatrix},
	\quad 
	r(\boldsymbol{\xi}) =
	\begin{pmatrix}
	0 & -\xi_3 & \xi_2 \\
	\xi_3 & 0 & -\xi_1 \\
	-\xi_2 & \xi_1 & 0 
	\end{pmatrix}, \quad \boldsymbol{\xi}^t = \begin{pmatrix} \xi_1 & \xi_2 & \xi_3 \end{pmatrix}.
	\end{equation}
We have 
$$
\operatorname{rank} b(\boldsymbol{\xi}) =3, \quad 0 \ne \boldsymbol{\xi} \in {\mathbb R}^3.
$$
This condition is equivalent to the estimates
	\begin{equation}
	\label{DSu1}
\alpha_0 \mathbf{1}_3 \le b(\boldsymbol{\xi})^* b(\boldsymbol{\xi}) \le \alpha_1 \mathbf{1}_3,
\quad |\boldsymbol{\xi}| =1,
	\end{equation}
with some positive constants $\alpha_0, \alpha_1$. It is easily seen that \eqref{DSu1}
 holds with the constants
$$
\alpha_0 = \min \{ |\mu_0|^{-1}; |\mu_0^{-1}|^{-1}\}, \quad \alpha_1 = |\mu_0| + |\mu_0^{-1}|.
$$
From \eqref{eta} and \eqref{nu} it follows that  $g({\mathbf x})$ is positive definite and bounded. Obviously,
$$
\|g\|_{L_\infty} = \max \{ \|\eta^{-1}\|_{L_\infty}, \|\nu\|_{L_\infty}\}, \quad 
\|g^{-1} \|_{L_\infty} = \max \{ \|\eta\|_{L_\infty}, \|\nu^{-1}\|_{L_\infty}\}.
$$
The precise definition of the operator $\mathcal{L}$ is given in terms of the quadratic form 
	\begin{equation*}
	\begin{aligned}
	&\mathfrak{l}[\mathbf{u},\mathbf{u}] := 
	\int_{{\mathbb R}^3} \langle g({\mathbf x}) b({\mathbf D}) \mathbf{u}, b({\mathbf D}) \mathbf{u} \rangle \, d{\mathbf x}  \\
	&=\int_{\mathbb{R}^3} \left( \left\langle \eta(\mathbf{x})^{-1} \operatorname{curl} \mu_0^{-1/2} \mathbf{u}, \operatorname{curl} \mu_0^{-1/2} \mathbf{u} \right\rangle + \nu(\mathbf{x}) \bigl| \operatorname{div} \mu_0^{1/2} \mathbf{u} \bigr|^2\right) \, d\mathbf{x}, 
	\quad \mathbf{u} \in H^1(\mathbb{R}^3; \mathbb{C}^3).
	\end{aligned}
	\end{equation*}
	Under our assumptions, 
	\begin{equation}
\label{estimates}
\begin{aligned}
c_0 \| {\mathbf D} {\mathbf u} \|^2_{L_2({\mathbb R}^3)} \le	\mathfrak{l}[\mathbf{u},\mathbf{u}] \le c_1 \| {\mathbf D} {\mathbf u}\|^2_{L_2({\mathbb R}^3)},
\quad {\mathbf u} \in H^1({\mathbb R}^3;{\mathbb C}^3),
\\
c_0 = \alpha_0 \|g^{-1}\|^{-1}_{L_\infty},\quad c_1 = \alpha_1 \|g \|_{L_\infty}.
\end{aligned}
\end{equation}
Thus, the form $\mathfrak{l}[{\mathbf u},{\mathbf u}]$ is closed and nonnegative. 
By definition, $\mathcal L$ is a selfadjoint operator in $L_2({\mathbb R}^3;{\mathbb C}^3)$ generated by this form. 

So, the operator $\mathcal L$ is a particular case of the operator $\mathcal A$ (see Introduction),
 and we can apply general resuts for the class of operators  $\mathcal A$.

\subsection{The Weyl decomposition. Reduction of the operator $\mathcal{L}$}\label{Weyl}
In $L_2({\mathbb R}^3;{\mathbb C}^3)$, we introduce the \textquotedblleft gradient\textquotedblright \  subspace 
	\begin{equation*}
	G(\mu_0) := \left\lbrace \mathbf{u} = \mu_0^{1/2} \nabla \phi \colon\  \phi \in H_{\mathrm{loc}}^1 (\mathbb{R}^3), \ \nabla \phi \in L_2(\mathbb{R}^3; \mathbb{C}^3)   \right\rbrace.
	\end{equation*}
The	\textquotedblleft divergence-free\textquotedblright \ subspace $J(\mu_0)$ is defined as the orthogonal complement  to  $G(\mu_0)$.
So, we have the following Weyl decomposition 
	\begin{equation}
	\label{Weyl_decomp}
	L_2(\mathbb{R}^3; \mathbb{C}^3) = J(\mu_0) \oplus G(\mu_0).
	\end{equation}
The subspace $J(\mu_0)$ consists of the functions $\mathbf{u} \in L_2(\mathbb{R}^3; \mathbb{C}^3) $ satisfying \hbox{$\operatorname{div} \mu_0^{1/2}\mathbf{u} = 0$} (in the sense of distributions). 
By $\mathcal{P}(\mu_0)$ we denote the orthogonal projection onto $J(\mu_0)$.

\begin{remark}
\label{PJ}
It is easily seen that  {\rm (}see, e.~g., \cite[Chapter 7, Section 2.4]{BSu1}{\rm )}
for $s>0$ the operator $\mathcal{P}(\mu_0)$ restricted to $H^s(\mathbb{R}^3;\mathbb{C}^3)$
is the orthogonal projection of the space $H^s(\mathbb{R}^3;\mathbb{C}^3)$ onto the subspace  $J^s(\mu_0) := 
J(\mu_0) \cap H^s(\mathbb{R}^3;\mathbb{C}^3)$. The operator $I -\mathcal{P}(\mu_0)$ restricted to  $H^s(\mathbb{R}^3;\mathbb{C}^3)$ is the orthogonal projection of  
$H^s(\mathbb{R}^3;\mathbb{C}^3)$ onto the subspace 
$G^s(\mu_0) := G(\mu_0) \cap H^s(\mathbb{R}^3;\mathbb{C}^3)$.
\end{remark}

The operator \eqref{L} is reduced by the decomposition~\eqref{Weyl_decomp}: $\mathcal{L}= \mathcal{L}_J \oplus \mathcal{L}_G$.  The part $\mathcal{L}_J$ acting in  $J(\mu_0)$ is formally given by the differential expression 
$\mu_0^{-1/2}\operatorname{curl} \eta(\mathbf{x})^{-1} \operatorname{curl} \mu_0^{-1/2}$, 
and the part $\mathcal{L}_G$ acting in  $G(\mu_0)$
is given by $- \mu_0^{1/2} \nabla \nu(\mathbf{x}) \operatorname{div} \mu_0^{1/2}$.

\subsection{The operators $\mathcal{L}(\mathbf{k})$}
In $L_2 (\Omega; \mathbb{C}^3)$, we consider the operator
 $\mathcal{L}(\mathbf{k})$ depending on the parameter $\mathbf{k}\in {\mathbb R}^3$ (called the quasimomentum) 
 and formally given by 
	\begin{equation*}
	\mathcal{L}(\mathbf{k}) =  
	\mu_0^{-1/2} \operatorname{curl}_\mathbf{k} \eta(\mathbf{x})^{-1} \operatorname{curl}_\mathbf{k} \mu_0^{-1/2}
	- \mu_0^{1/2} \nabla_\mathbf{k} \nu(\mathbf{x}) \operatorname{div}_\mathbf{k} \mu_0^{1/2}
	\end{equation*}
	with periodic boundary conditions. Here
	\begin{equation*}
	\nabla_\mathbf{k} \phi := \nabla \phi + i \mathbf{k} \phi, \quad \operatorname{div}_\mathbf{k} \mathbf{f} := \operatorname{div} \mathbf{f} + i \, \mathbf{k} \cdot \mathbf{f}, \quad \operatorname{curl}_\mathbf{k} \mathbf{f} := \operatorname{curl} \mathbf{f} + i \, \mathbf{k} \times \mathbf{f}
	\end{equation*}
($\mathbf{k} \cdot \mathbf{f}$ is the inner product and $\mathbf{k} \times \mathbf{f}$ is the vector product).
Strictly speaking, $\mathcal{L}(\mathbf{k})$ is a selfadjoint operator in  $L_2 (\Omega; \mathbb{C}^3)$ generated by the closed nonnegative quadratic form 
	\begin{equation*}
\begin{aligned}
	\mathfrak{l}(\mathbf{k})[\mathbf{u}, \mathbf{u}] =
\int_{\Omega}  \left\langle \eta(\mathbf{x})^{-1} \operatorname{curl}_\mathbf{k} \mu_0^{-1/2}\mathbf{u}, \operatorname{curl}_\mathbf{k} \mu_0^{-1/2}\mathbf{u} \right\rangle \, d {\mathbf x} 
\\
+ \int_\Omega \nu(\mathbf{x}) \bigl| \operatorname{div}_\mathbf{k} \mu_0^{1/2}\mathbf{u} \bigr|^2 \, d\mathbf{x},
\quad \mathbf{u} \in \widetilde{H}^1(\Omega; \mathbb{C}^3).
\end{aligned}	
\end{equation*}
Using the Fourier series expansion for a function $\mathbf{u}$, it is easily seen that 
	\begin{equation}
	\label{l(k)_form_estimate}
\begin{split}
	c_0  \|(\mathbf{D} + \mathbf{k}) \mathbf{u}\|_{L_2 (\Omega)}^2 \le \mathfrak{l}(\mathbf{k})[\mathbf{u}, \mathbf{u}] \le c_1  \|(\mathbf{D} + \mathbf{k}) \mathbf{u}\|_{L_2 (\Omega)}^2, \quad
\mathbf{u} \in \widetilde{H}^1(\Omega; \mathbb{C}^3),
\end{split}	
\end{equation}
where the constants $c_0, c_1$ are the same as in  \eqref{estimates}.

	By the lower estimate~(\ref{l(k)_form_estimate}), 
	\begin{equation}
	\label{c*}
	\mathcal{L}(\mathbf{k}) \ge c_0 |\mathbf{k}|^2 I, \quad \mathbf{k} \in \widetilde{\Omega}.
	\end{equation}

\subsection{Reduction of the operators ${\mathcal L}(\mathbf{k})$}
In $L_2 (\Omega; \mathbb{C}^3)$, we define  the 
 \textquotedblleft gradient\textquotedblright \ subspace  (depending on the parameter $\mathbf{k} \in \mathbb{R}^3$)
	\begin{equation*}
	G(\mathbf{k};\mu_0) := \{\mathbf{u} = \mu_0^{1/2} \nabla_\mathbf{k} \phi \colon\  \phi \in \widetilde{H}^1(\Omega) \}.
	\end{equation*}
The \textquotedblleft divergence-free\textquotedblright \ subspace $J(\mathbf{k}; \mu_0)$ is defined as the orthogonal complement to $G(\mathbf{k};\mu_0)$:
	\begin{equation}
	\label{H_Weyl_decomp}
	 L_2 (\Omega; \mathbb{C}^3) = J(\mathbf{k};\mu_0) \oplus G(\mathbf{k}; \mu_0).
	\end{equation}
	The subspace $J(\mathbf{k}; \mu_0)$ consists of the functions $\mathbf{u} \in L_2 (\Omega; \mathbb{C}^3)$ satisfying  $\operatorname{div}_\mathbf{k} \mu_0^{1/2}\check{\mathbf{u}} = 0$ (in the sense of distributions), where $\check{\mathbf{u}}$ is the $\Gamma$-periodic extension of a function  
	$\mathbf{u}$ to ${\mathbb R}^3$. Let $\mathcal{P}(\mathbf{k};\mu_0)$ be the orthogonal projection onto 
	$J(\mathbf{k}; \mu_0)$.

The operator $\mathcal{L}(\mathbf{k})$ is reduced by  decomposition~\eqref{H_Weyl_decomp}.
 The part  $\mathcal{L}_J(\mathbf{k})$ acting in  $J(\mathbf{k};\mu_0)$ is formally given by the expression
$\mu_0^{-1/2}\operatorname{curl}_{\mathbf{k}} \eta(\mathbf{x})^{-1} \operatorname{curl}_{\mathbf{k}} \mu_0^{-1/2}$ (with periodic boundary conditions), and the part  $\mathcal{L}_G(\mathbf{k})$ acting in
$G(\mathbf{k}; \mu_0)$ is given by 
$- \mu_0^{1/2}\nabla_{\mathbf{k}} \nu(\mathbf{x}) \operatorname{div}_{\mathbf{k}} \mu_0^{1/2}$.

\subsection{Direct integral expansion for the operator $\mathcal{L}$}
	Under the Gelfand transformation $\mathcal{U}$, the operator $\mathcal{L}$ expands in the direct integral of the operators  $\mathcal{L} (\mathbf{k})$:
	\begin{equation*}
	\mathcal{U} \mathcal{L}  \mathcal{U}^{-1} = \int_{\widetilde \Omega} \oplus \mathcal{L} (\mathbf{k}) \, d \mathbf{k}.
	\end{equation*}
This means the following. Let  $\mathbf{v} \in H^1(\mathbb{R}^3; \mathbb{C}^3)$. Then 
	\begin{align}
	\label{Gelfand_indetail_1}
	\widetilde{\mathbf{v}}(\mathbf{k}, \cdot) &\in \widetilde{H}^1(\Omega; \mathbb{C}^3) \quad \text{for almost all} \; \mathbf{k} \in \widetilde \Omega, \\
	\label{Gelfand_indetail_2}
	\mathfrak{l}[\mathbf{v}, \mathbf{v}] &= \int_{\widetilde{\Omega}} \mathfrak{l}(\mathbf{k}) [\widetilde{\mathbf{v}}(\mathbf{k}, \cdot), \widetilde{\mathbf{v}}(\mathbf{k}, \cdot)] \, d \mathbf{k} .
	\end{align}
	Conversely, if  $\widetilde{\mathbf{v}} \in \mathcal{K}$ satisfies~\eqref{Gelfand_indetail_1} and the
	 integral in~\eqref{Gelfand_indetail_2} is finite, then $\mathbf{v} \in H^1(\mathbb{R}^3; \mathbb{C}^3)$  
	 and~\eqref{Gelfand_indetail_2} holds.

Under the Gelfand transform, the orthogonal projection $\mathcal{P}(\mu_0)$ expands in the direct integral of the orthogonal projections $\mathcal{P}(\mathbf{k}; \mu_0)$; see \cite{Su2}. Therefore, the operator $\mathcal{L} \mathcal{P}(\mu_0) = \mathcal{L}_J \oplus \mathbf{0}_{G(\mu_0)}$ expands in the direct integral of the operators $\mathcal{L}(\mathbf{k}) \mathcal{P}(\mathbf{k};\mu_0) = \mathcal{L}_J (\mathbf{k}) \oplus \mathbf{0}_{G(\mathbf{k};\mu_0)}$:
	\begin{equation*}
	\mathcal{U} \mathcal{L} \mathcal{P}(\mu_0) \mathcal{U}^{-1} = \int_{\widetilde{\Omega}} \oplus \mathcal{L}(\mathbf{k}) \mathcal{P}(\mathbf{k};\mu_0) \, d \mathbf{k}.
	\end{equation*}

\section{Effective characteristics}

\subsection{The analytic branches of eigenvalues and eigenvectors}
According to \cite{BSu1}, we put 
$$
\mathbf{k} = t \boldsymbol{\theta},\quad t= |\mathbf{k}|,\quad \boldsymbol{\theta}\in {\mathbb S}^2,
$$
and denote $\mathcal{L} (\mathbf{k})= \mathcal{L}(t \boldsymbol{\theta})=: L(t ;\boldsymbol{\theta})$.
The operator family $L(t ;\boldsymbol{\theta})$ depends on the onedimensional parameter $t$ analytically 
and  has discrete spectrum (since $\mathcal{L} (\mathbf{k})$ is an elliptic operator in a bounded domain).
We can apply analytic perturbation theory (see \cite{K}).
For $t=0$ the \textquotedblleft unperturbed\textquotedblright \  operator $\mathcal{L}(0)$ has an  isolated threemultiple eigenvalue $\lambda_0=0$.
The corresponding eigenspace consists of constant vector-valued functions:
\begin{equation}
\label{frakN}
\mathfrak{N} := \operatorname{Ker} \mathcal{L} (0) =  \left\lbrace \mathbf{u} \in L_2(\Omega;\mathbb{C}^3)
\colon\  \mathbf{u} = \mathbf{c} \in \mathbb{C}^3 \right\rbrace.
\end{equation}
Let $P$ be the orthogonal projection of $L_2(\Omega;\mathbb{C}^3)$ onto the subspace $\mathfrak{N}$:
\begin{equation*}
P \mathbf{u} = |\Omega|^{-1} \int_\Omega \mathbf{u}(\mathbf{x}) \, d\mathbf{x}.
\end{equation*}

We put 
\begin{equation*}
\delta:= \frac{r_0^2}{4} \alpha_0 \| g^{-1}\|^{-1}_{L_\infty},
\quad
t^0 := \frac{r_0}{2} \alpha_0^{1/2} \alpha_1^{-1/2} \|g\|^{-1/2}_{L_\infty} \| g^{-1}\|^{-1/2}_{L_\infty}.
\end{equation*}
As was shown in \cite{BSu1}, for $t\le t^0$ the operator $L(t ;\boldsymbol{\theta})$ has exactly three eigenvalues 
 (counted with multiplicities) 
 $\lambda_l(t ;\boldsymbol{\theta})$, $l=1,2,3,$ belonging to the interval $[0,\delta]$,
while the interval $(\delta, 3\delta)$ is free of the spectrum. By 
${\mathfrak F}(\mathbf{k})={\mathfrak F}(t ;\boldsymbol{\theta})$ we denote the eigenspace of the operator 
$L(t ;\boldsymbol{\theta})$ for the interval $[0,\delta]$.

According to the analytic perturbation theory,  for $t \le t^0$ the eigenvalues 
$\lambda_l(t ;\boldsymbol{\theta})$, $l=1,2,3,$ can be enumerated in such a way  that they 
are real-analytic functions of $t$
(for each fixed $\boldsymbol{\theta} \in \mathbb{S}^2$) and the corresponding  eigenvectors 
$\boldsymbol{\varphi}_l(t;\boldsymbol{\theta})$, $l = 1,2,3$, orthonormal in $L_2(\Omega;\mathbb{C}^3)$ 
are real-analytic in  $t$. Thus,
\begin{equation*}
L(t;\boldsymbol{\theta}) \boldsymbol{\varphi}_l(t;\boldsymbol{\theta}) = \lambda_l(t; \boldsymbol{\theta}) \boldsymbol{\varphi}_l(t;\boldsymbol{\theta}), \quad l = 1,2,3, \quad 0 \le t \le t^0,
\end{equation*}
and the set $\boldsymbol{\varphi}_l(t;\boldsymbol{\theta})$, $l = 1,2,3$, forms an orthonormal basis in the subspace 
${\mathfrak F}(t ;\boldsymbol{\theta})$.
For sufficiently small  $0< t_*=t_*(\boldsymbol{\theta}) \le t^0$ and $t \le t_*(\boldsymbol{\theta})$, we have 
the following convergent power series expansions 
\begin{align}
\label{eigenvalues_series}
\lambda_l(t; \boldsymbol{\theta}) &=  \gamma_l(\boldsymbol{\theta})t^2 +  \mu_l(\boldsymbol{\theta})t^3 + \ldots, \qquad  l = 1,2,3, \\
\label{eigenvectors_series}
\boldsymbol{\varphi}_l(t;\boldsymbol{\theta}) &= \boldsymbol{\omega}_l(\boldsymbol{\theta}) + t \boldsymbol{\psi}_l(\boldsymbol{\theta}) + \ldots, \qquad l = 1,2,3.
\end{align}
The vectors $\boldsymbol{\omega}_l(\boldsymbol{\theta})$, $l = 1,2,3$, form an orthonormal basis 
in the subspace $\mathfrak{N}$.
By \eqref{c*},  \hbox{$\gamma_l(\boldsymbol{\theta}) \ge c_0 >0$}; in general, the coefficients $\mu_l(\boldsymbol{\theta}) \in \mathbb{R}$ may be nonzero.
The coefficients of the power series expansions \eqref{eigenvalues_series}, \eqref{eigenvectors_series} are called the  \textit{threshold characteristics} of the operator $\mathcal{L}$ at the bottom of the spectrum.

\subsection{The spectral germ. The effective matrix\label{sec_effective}}
The key role is played by the \textit{spectral germ} $S (\boldsymbol{\theta})$ of the operator  
$L(t; \boldsymbol{\theta})$; see \cite{BSu1}.
Let us give the spectral definition of the germ: $S (\boldsymbol{\theta})$  \textit{is a selfadjoint operator 
in the space $\mathfrak{N}$ such that the numbers $\gamma_l(\boldsymbol{\theta})$ and the elements 
$\boldsymbol{\omega}_l(\boldsymbol{\theta})$ are its eigenvalues and eigenvectors}:
\begin{equation*}
S(\boldsymbol{\theta})\boldsymbol{\omega}_l(\boldsymbol{\theta})=\gamma_l(\boldsymbol{\theta})
\boldsymbol{\omega}_l(\boldsymbol{\theta}),\quad l=1,2,3.
\end{equation*}
In  \cite{BSu1},  the following invariant representation for the germ was obtained:
\begin{equation}
\label{germ}
S (\boldsymbol{\theta}) = b(\boldsymbol{\theta})^* g^0 b(\boldsymbol{\theta}), \quad  \boldsymbol{\theta} \in \mathbb{S}^{2},
\end{equation}
where $b(\boldsymbol{\theta})$~is the symbol of the operator $b(\mathbf{D})$, and $g^0$~is the so called effective matrix. The constant positive  ($4 \times 4$)-matrix $g^0$ is defined as follows. Let $\Lambda \in \widetilde{H}^1 (\Omega)$~be the ($3 \times 4$)-matrix-valued function which is the $\Gamma$-periodic solution of the problem 
\begin{equation}
\label{equation_for_Lambda}
b(\mathbf{D})^* g(\mathbf{x}) (b(\mathbf{D}) \Lambda (\mathbf{x}) + \mathbf{1}_4) = 0, \quad \int_{\Omega} \Lambda (\mathbf{x}) \, d \mathbf{x} = 0.
\end{equation}
The effective matrix $g^0$ is defined in terms of the matrix $\Lambda (\mathbf{x})$:
\begin{gather}
\label{g_tilde}
\widetilde{g} (\mathbf{x}) := g(\mathbf{x})( b(\mathbf{D}) \Lambda (\mathbf{x}) + \mathbf{1}_4), \\
\label{g0}
g^0 = | \Omega |^{-1} \int_{\Omega} \widetilde{g} (\mathbf{x}) \, d \mathbf{x}.
\end{gather}
It turns out that  the matrix $g^0$ is positive definite.

The  effective characteristics for the operator 
$L(t; \boldsymbol{\theta})$ were calculated in \cite{BSu-FAA} and \cite{Su-AA18}.

First, we introduce the effective matrix $\eta^0$~for the scalar elliptic operator $- \operatorname{div} \eta(\mathbf{x}) \nabla = \mathbf{D}^* \eta(\mathbf{x}) \mathbf{D}$.
Recall the definition of  $\eta^0$. Let $\mathbf{e}_1$, $\mathbf{e}_2$, $\mathbf{e}_3$ be the standard othonormal basis in $\mathbb{R}^3$. Let $\Phi_j(\mathbf{x})$ be the $\Gamma$-periodic solution of the problem
\begin{equation}
\label{2.8a}
\operatorname{div} \eta(\mathbf{x}) (\nabla \Phi_j(\mathbf{x}) + \mathbf{e}_j)=0,
\quad \int_{\Omega} \Phi_j (\mathbf{x}) \, d \mathbf{x} = 0.
\end{equation}
Consider the matrix $\Sigma_{\circ}({\mathbf x})$ with the columns $\nabla \Phi_j({\mathbf x})$, $j=1,2,3$. We put  
$$
\widetilde{\eta} (\mathbf{x}):= \eta({\mathbf x}) (\Sigma_\circ({\mathbf x}) + {\mathbf 1}_3).
$$
 Then 
$$
\eta^0 = | \Omega |^{-1} \int_{\Omega} \widetilde{\eta} (\mathbf{x}) \, d \mathbf{x}.
$$

\begin{remark}
\label{eta0_properties}
Note that the matrix $\eta^0$ has the following properties\emph{:}

\noindent $1^\circ$. We have $\underline{\eta} \le \eta^0 \le \overline{\eta}$ {\rm (}these estimates are known as the Voigt--Reuss bracketing{\rm )}.
It follows that  $|\eta^0|\le \|\eta\|_{L_\infty}$, $|(\eta^0)^{-1}| \le \|\eta^{-1}\|_{L_\infty}$.

\noindent $2^\circ$. The identity $\eta^0 = \overline{\eta}$ is equivalent to the fact that the columns 
$\boldsymbol{\eta}_j ({\mathbf x})$, $j=1,2,3$, of the matrix $\eta({\mathbf x})$ are divergence-free{\rm :} $\operatorname{div}\, \boldsymbol{\eta}_j ({\mathbf x})=0$.
In this case, the solution of problem \eqref{2.8a} is trivial{\rm :} \hbox{$\Phi_j(\mathbf{x})=0$}, $j=1,2,3$.

\noindent $3^\circ$. The identity $\eta^0 = \underline{\eta}$ is equivalent to the 
 fact that the columns 
$\mathbf{l}_j ({\mathbf x})$, $j=1,2,3$, of the matrix 
 $\eta({\mathbf x})^{-1}$ can be represented as  
 $\mathbf{l}_j ({\mathbf x}) = \nabla \phi_j(\mathbf{x}) + \mathbf{l}_j^0$
for some  $\phi_j \in \widetilde{H}^1(\Omega)$ and $\mathbf{l}_j^0 \in \mathbb{R}^3$.
 In this case we have  $\widetilde{\eta}(\mathbf{x}) = \eta^0 = \underline{\eta}$.
\end{remark}

We put  ${\mathbf c}_j= (\eta^0)^{-1} {\mathbf e}_j$, $j=1,2,3$.
Let $\widetilde{\Phi}_j(\mathbf{x})$ be the $\Gamma$-periodic solution of the problem 
\begin{equation}
\label{2.8aaa}
\operatorname{div} \eta(\mathbf{x}) (\nabla \widetilde{\Phi}_j(\mathbf{x}) + \mathbf{c}_j)=0,
\quad \int_{\Omega} \widetilde{\Phi}_j (\mathbf{x}) \, d \mathbf{x} = 0.
\end{equation}
Let ${\mathbf p}_j \in \widetilde{H}^1(\Omega;{\mathbb C}^3)$ (where $j=1,2,3$) be the $\Gamma$-periodic solution of the problem
$$
\begin{aligned}
\operatorname{curl} (\mu_0^{-1} \operatorname{curl} {\mathbf p}_j({\mathbf x})) = {\eta}({\mathbf x}) (\nabla \widetilde{\Phi}_j({\mathbf x}) + {\mathbf c}_j)  -{\mathbf e}_j,
\\
 \operatorname{div} {\mathbf p}_j( {\mathbf x}) =0, \quad \int_\Omega {\mathbf p}_j({\mathbf x}) \,d {\mathbf x} =0.
\end{aligned}
$$
Let $\rho \in \widetilde{H}^1(\Omega)$ be the $\Gamma$-periodic solution of the problem 
$$
- \operatorname{div} (\mu_0 \nabla \rho({\mathbf x})) = 1 - \underline{\nu} \, \nu({\mathbf x})^{-1}, \quad \int_\Omega \rho({\mathbf x}) \,d {\mathbf x} =0. 
$$
Then the $(3 \times 4)$-matrix  $\Lambda({\mathbf x})$ takes the form
\begin{equation*}
\Lambda({\mathbf x}) = i \begin{pmatrix} 
\mu_0^{-1/2} \Psi({\mathbf x}) & \mu_0^{1/2} \nabla \rho({\mathbf x})
\end{pmatrix},
\end{equation*}
where $\Psi({\mathbf x})$ is the $(3 \times 3)$-matrix with the columns $\operatorname{curl} {\mathbf p}_j({\mathbf x})$, $j=1,2,3$. 
 
Next, the matrix $\widetilde{g} ({\mathbf x}) = g({\mathbf x}) (b({\mathbf D}) \Lambda({\mathbf x}) + {\mathbf 1}_4)$ is given by  
$$
\widetilde{g}({\mathbf x}) = \begin{pmatrix}  (\eta^0)^{-1} + \Sigma({\mathbf x}) & 0 \\ 0 & \underline{\nu} \end{pmatrix},
$$
where $\Sigma({\mathbf x})$ is the matrix with the columns $\nabla \widetilde{\Phi}_j({\mathbf x})$, $j=1,2,3$. Note that 
$\Sigma({\mathbf x}) = \Sigma_\circ({\mathbf x}) (\eta^0)^{-1}$.

According to \eqref{g0}, we obtain
\begin{equation}
\label{g00}
g^0 = \begin{pmatrix}
(\eta^0)^{-1} & 0 \\
0 & \underline{\nu}
\end{pmatrix}.
\end{equation}

By \eqref{germ} and \eqref{g00}, the germ $S(\boldsymbol{\theta})$ can be written as 
\begin{equation}
\label{S_decomp}
S(\boldsymbol{\theta}) = \mu_0^{-1/2} r(\boldsymbol{\theta})^t (\eta^0)^{-1} r(\boldsymbol{\theta}) \mu_0^{-1/2} + \underline{\nu} \, \mu_0 ^{1/2} \boldsymbol{\theta} \boldsymbol{\theta}^t \mu_0^{1/2},
\end{equation}
where the symbol $r(\boldsymbol{\theta})$ is defined by \eqref{symbol}.

\subsection{Decomposition of the spectral germ}
Consider the following orthogonal decomposition of the threedimensional space~(\ref{frakN}) depending on the parameter \hbox{$\boldsymbol{\theta} \in \mathbb{S}^2$}:
\begin{equation}
\label{N_Weyl_decomp}
\mathfrak{N} = J_{\boldsymbol{\theta}}^0 \oplus G_{\boldsymbol{\theta}}^0,
\end{equation}
where 
\begin{gather*}
J_{\boldsymbol{\theta}}^0 = \{ \mu_0^{1/2} \mathbf{c} \in \mathbb{C}^3 \colon \langle \mu_0 \mathbf{c},  \boldsymbol{\theta} \rangle =0 \}, \\
G_{\boldsymbol{\theta}}^0 = \{ \mathbf{c} = \alpha \mu_0^{1/2}\boldsymbol{\theta} \colon \alpha \in \mathbb{C} \}.
\end{gather*}

Obviously, the operator $S(\boldsymbol{\theta})$ is reduced by  decomposition~(\ref{N_Weyl_decomp}). 
 The part $S_J(\boldsymbol{\theta})$ of $S(\boldsymbol{\theta})$ in $J_{\boldsymbol{\theta}}^0$ corresponds to the first term in  (\ref{S_decomp}), and the part $S_G(\boldsymbol{\theta})$ of  $S(\boldsymbol{\theta})$ in  $G_{\boldsymbol{\theta}}^0$ corresponds to the second term. 
 The operator $S(\boldsymbol{\theta})$ has unique eigenvalue in the subspace $G_{\boldsymbol{\theta}}^0$:
 \begin{equation}
 \label{gamma3}
\gamma_3 (\boldsymbol{\theta}) = \underline{\nu} \langle \mu_0 \boldsymbol{\theta}, \boldsymbol{\theta} \rangle.
\end{equation} 
 The corresponding normed eigenvector is given by 
 \begin{equation}
 \label{omega3}
 \boldsymbol{\omega}_3(\boldsymbol{\theta}) = |\Omega|^{-1/2} \langle \mu_0 \boldsymbol{\theta}, \boldsymbol{\theta}\rangle^{-1/2}  \mu_0^{1/2}  \boldsymbol{\theta}.
 \end{equation}

In the subspace  $J_{\boldsymbol{\theta}}^0$~the germ has two eigenvalues $\gamma_1 (\boldsymbol{\theta})$ and $\gamma_2 (\boldsymbol{\theta})$ corresponding to the algebraic problem 
\begin{equation}
\label{solenoid_germ_spec_probl}
r(\boldsymbol{\theta})^t (\eta^0)^{-1} r(\boldsymbol{\theta}) \mathbf{c} = \gamma \mu_0 
\mathbf{c}, \quad \mu_0 \mathbf{c} \perp \boldsymbol{\theta}.
\end{equation}

We have the following  simple estimates 
\begin{equation}
\label{gammaj_est1}
\begin{aligned}
&\gamma_j(\boldsymbol{\theta}) \le |\mu_0^{-1}| | (\eta^0)^{-1}| \le |\mu_0^{-1}| \| \eta^{-1}\|_{L_\infty},
\quad \boldsymbol{\theta} \in \mathbb{S}^2, \quad j=1,2;
\\
&\gamma_3(\boldsymbol{\theta}) \ge \underline{\nu} |\mu_0^{-1}|^{-1},\quad \boldsymbol{\theta} \in \mathbb{S}^2.
\end{aligned}
\end{equation}

\begin{remark}
		As was mentioned in \cite[Remark 4.5]{Su2}, we can always choose the analytic branches of eigenvalues and eigenvectors of the operator $L(t; \boldsymbol{\theta})$, $t \in [0, t^0]$, in such a way that one of the eigenvectors \emph{(}let it be $\boldsymbol{\varphi}_3 (t; \boldsymbol{\theta})$\emph{)} belongs to the  \textquotedblleft gradient\textquotedblright \ subspace $G(t\boldsymbol{\theta}; \mu_0)$ for \hbox{$t \ne 0$}, and then  \emph{(}automatically\emph{)} the other two eigenvectors $\boldsymbol{\varphi}_1 (t; \boldsymbol{\theta})$, $\boldsymbol{\varphi}_2 (t; \boldsymbol{\theta})$ belong to the \textquotedblleft divergence-free\textquotedblright \ subspace $J(t \boldsymbol{\theta};\mu_0)$. The coefficient $\gamma_3(\boldsymbol{\theta})$ in 
		expansion~\emph{(\ref{eigenvalues_series})} for $\lambda_3 (t; \boldsymbol{\theta})$ is the eigenvalue of the part of the germ $S(\boldsymbol{\theta})$ in the subspace $G_{\boldsymbol{\theta}}^0$.
The \textquotedblleft embryo\textquotedblright \ $\boldsymbol{\omega}_3(\boldsymbol{\theta})$ in 
 expansion~\emph{(\ref{eigenvectors_series})} for $\boldsymbol{\varphi}_3 (t; \boldsymbol{\theta})$ is given by  
\eqref{omega3}.
The coefficients $\gamma_1(\boldsymbol{\theta})$ and $\gamma_2(\boldsymbol{\theta})$ in  
expansions~\emph{(\ref{eigenvalues_series})} for $\lambda_1 (t; \boldsymbol{\theta})$, $\lambda_2 (t; \boldsymbol{\theta})$ are eigenvalues of 
 $S_J(\boldsymbol{\theta})$  and correspond to the algebraic problem~\emph{(\ref{solenoid_germ_spec_probl})}. The \textquotedblleft embryos\textquotedblright \ $\boldsymbol{\omega}_1 (\boldsymbol{\theta})$ and 
$\boldsymbol{\omega}_2 (\boldsymbol{\theta})$ in  expansions~\emph{(\ref{eigenvectors_series})} for $\boldsymbol{\varphi}_1 (t; \boldsymbol{\theta})$ and  $\boldsymbol{\varphi}_2 (t; \boldsymbol{\theta})$ belong to  $J_{\boldsymbol{\theta}}^0$ and are the eigenvectors of  problem~\eqref{solenoid_germ_spec_probl}.
		If $\gamma_1(\boldsymbol{\theta}) \ne \gamma_2(\boldsymbol{\theta})$, then $\boldsymbol{\omega}_1 (\boldsymbol{\theta})$ and $\boldsymbol{\omega}_2 (\boldsymbol{\theta})$ are defined uniquely  \emph{(}up to phase factors\emph{)}.
		For $t=0$ all three eigenvectors belong to the  \textquotedblleft divergence-free\textquotedblright \ subspace $J(0;\mu_0)$: $\boldsymbol{\varphi}_l (0;\boldsymbol{\theta}) = \boldsymbol{\omega}_l(\boldsymbol{\theta}) \in \mathfrak{N}$, $l = 1,2,3$. Note also that, if  $\gamma_1(\boldsymbol{\theta}) = \gamma_2(\boldsymbol{\theta})$, then the knowledge of the germ $S(\boldsymbol{\theta})$ is not sufficient to determine the \textquotedblleft embryos\textquotedblright \ $\boldsymbol{\omega}_1 (\boldsymbol{\theta})$, $\boldsymbol{\omega}_2 (\boldsymbol{\theta})$.
\end{remark}

\subsection{The operator $N(\boldsymbol{\theta})$}
We also need the operator $N(\boldsymbol{\theta})$ acting in the space $\mathfrak{N}$ and defined in terms of the coefficients of the power series expansions \eqref{eigenvalues_series}, \eqref{eigenvectors_series} as follows:
\begin{align}
\nonumber
N(\boldsymbol{\theta}) &= N_0(\boldsymbol{\theta}) + N_*(\boldsymbol{\theta}),
\\
\label{N0}
N_0(\boldsymbol{\theta}) &= \sum_{l=1}^3 \mu_l(\boldsymbol{\theta}) (\cdot,\boldsymbol{\omega}_l(\boldsymbol{\theta}) )_{L_2(\Omega)}\boldsymbol{\omega}_l(\boldsymbol{\theta}),
\\
\nonumber
N_*(\boldsymbol{\theta})
&= \sum_{l=1}^3 \gamma_l(\boldsymbol{\theta}) \left((\cdot,P \boldsymbol{\psi}_l(\boldsymbol{\theta}) )_{L_2(\Omega)}
\boldsymbol{\omega}_l(\boldsymbol{\theta})  + (\cdot,\boldsymbol{\omega}_l(\boldsymbol{\theta}))_{L_2(\Omega)}   P \boldsymbol{\psi}_l(\boldsymbol{\theta})  \right).
\end{align}
For more details, see \cite{BSu2}.

\begin{remark}
\label{rem_N}
In the basis $\{\boldsymbol{\omega}_l(\boldsymbol{\theta})\}_{l=1}^3$, the operator $N_0(\boldsymbol{\theta})$ is diagonal, while the diagonal entries of   $N_*(\boldsymbol{\theta})$ are equal to zero. We have
\begin{align}
\label{2.13a}
(N(\boldsymbol{\theta})\boldsymbol{\omega}_l(\boldsymbol{\theta}),\boldsymbol{\omega}_l(\boldsymbol{\theta}))_{L_2(\Omega)}&=
(N_0(\boldsymbol{\theta})\boldsymbol{\omega}_l(\boldsymbol{\theta}),\boldsymbol{\omega}_l(\boldsymbol{\theta}))_{L_2(\Omega)}
=\mu_l(\boldsymbol{\theta}), \quad l=1,2,3,
\\
\nonumber 
\begin{split}
(N(\boldsymbol{\theta})\boldsymbol{\omega}_l(\boldsymbol{\theta}),\boldsymbol{\omega}_j(\boldsymbol{\theta}))_{L_2(\Omega)}&=
(N_*(\boldsymbol{\theta})\boldsymbol{\omega}_l(\boldsymbol{\theta}),\boldsymbol{\omega}_j(\boldsymbol{\theta}))_{L_2(\Omega)}
\\
&=(\gamma_l(\boldsymbol{\theta}) - \gamma_j(\boldsymbol{\theta})) (P \boldsymbol{\psi}_l(\boldsymbol{\theta}),
\boldsymbol{\omega}_j(\boldsymbol{\theta})), \quad l\ne j.
\end{split}
\end{align}
\end{remark}

In~\cite[\S4]{BSu2}, the following invariant representation for the operator $N (\boldsymbol{\theta})$ was obtained:
\begin{equation*}
N (\boldsymbol{\theta}) = b(\boldsymbol{\theta})^* M(\boldsymbol{\theta}) b(\boldsymbol{\theta}),
\end{equation*}
where $M (\boldsymbol{\theta})$~is the ($4 \times 4$)-matrix given by 
\begin{equation*}
M (\boldsymbol{\theta}) = | \Omega |^{-1} \int_{\Omega} (\Lambda (\mathbf{x})^* b(\boldsymbol{\theta})^* \widetilde{g}(\mathbf{x}) + \widetilde{g}(\mathbf{x})^* b(\boldsymbol{\theta}) \Lambda (\mathbf{x}) ) \, d \mathbf{x}.
\end{equation*}
Here $\Lambda (\mathbf{x})$~is the $\Gamma$-periodic solution of problem~(\ref{equation_for_Lambda}), and $\widetilde{g}(\mathbf{x})$~is given by~(\ref{g_tilde}).
For $L(t; \boldsymbol{\theta})$, the operator $N(\boldsymbol{\theta})$ was calculated  in~\cite[Section~14.3]{BSu2} (in the case where $\mu_0 = {\mathbf 1}$). Transferring the calculation to the case of a constant matrix $\mu_0$, 
it is easy to show that
\begin{equation}
\label{N_operator}
N(\boldsymbol{\theta})  = - i f(\boldsymbol{\theta}) \mu_0^{-1/2} r(\boldsymbol{\theta}) \mu_0^{-1/2},
\end{equation}
where the matrix $r(\boldsymbol{\theta})$ is defined by  \eqref{symbol}, and 
\begin{equation}
\label{Mjk}
\begin{split}
f(\boldsymbol{\theta}) :=& (\rho_{12}(\boldsymbol{\theta}) - \rho_{21}(\boldsymbol{\theta}))  \theta_3 +
(\rho_{31}(\boldsymbol{\theta}) - \rho_{13}(\boldsymbol{\theta})) \theta_2 +
(\rho_{23}(\boldsymbol{\theta}) - \rho_{32}(\boldsymbol{\theta})) \theta_1,
\\
\rho_{jk}(\boldsymbol{\theta}) :=& | \Omega |^{-1} \int_{\Omega} \widetilde{\Phi}_j (\mathbf{x}) \bigl\langle
 \eta (\mathbf{x}) (\nabla \widetilde{\Phi}_k (\mathbf{x}) + \mathbf{c}_k), \boldsymbol{\theta} \bigr\rangle \, d\mathbf{x}.
\end{split}
\end{equation}
Obviously, the operator $N(\boldsymbol{\theta})$ is reduced by  decomposition~(\ref{N_Weyl_decomp}). 
The part of $N(\boldsymbol{\theta})$ in the subspace $G_{\boldsymbol{\theta}}^0$ is equal to zero.

\begin{remark}
\label{rem_NN}
Since $\boldsymbol{\omega}_3 (\boldsymbol{\theta}) = \alpha \mu_0^{1/2} \boldsymbol{\theta}$ \emph{(}see  \eqref{omega3}\emph{)}, then \eqref{N_operator} and the obvious identity $r(\boldsymbol{\theta}) \boldsymbol{\theta} =0$ imply  that 
		\begin{equation*}
  (N (\boldsymbol{\theta}) \boldsymbol{\omega}_3(\boldsymbol{\theta}), \boldsymbol{\omega}_j (\boldsymbol{\theta})) =
  (N (\boldsymbol{\theta}) \boldsymbol{\omega}_j (\boldsymbol{\theta}), \boldsymbol{\omega}_3 (\boldsymbol{\theta})) = 0,
\quad j=1,2,3,\quad \boldsymbol{\theta} \in \mathbb{S}^2.
		\end{equation*}
It follows {\rm (}see \eqref{2.13a}{\rm )} that the coefficient $\mu_3(\boldsymbol{\theta})$ in  expansion  \eqref{eigenvalues_series} of the eigenvalue 
$\lambda_3(t;\boldsymbol{\theta})$ corresponding to the \textquotedblleft gradient\textquotedblright \ subspace 
$G(t\boldsymbol{\theta};\mu_0)$, is equal to zero{\rm :}
		\begin{equation*}
		\mu_3(\boldsymbol{\theta})= 0,
\quad \boldsymbol{\theta} \in \mathbb{S}^2.
		\end{equation*}
\end{remark}

\begin{remark}
\label{N=0}
$1^\circ$.
Suppose that $\eta^0 = \overline{\eta}$ \emph{(}see Remark \emph{\ref{eta0_properties}$(2^\circ)$)}. 
Then the columns of the matrix
$\eta(\mathbf{x})$ are divergence-free, whence the periodic solutions $\widetilde{\Phi}_j({\mathbf x})$ $(j=1,2,3)$ of  problems  \eqref{2.8aaa} are equal to zero. In this case,  $N(\boldsymbol{\theta})=0$ for any $\boldsymbol{\theta} \in \mathbb{S}^2$.
In particular, this is the case if the matrix  $\eta(\mathbf{x})$ is constant.

\noindent
$2^\circ$. Suppose that $\eta^0 = \underline{\eta}$ \emph{(}see Remark \emph{\ref{eta0_properties}$(3^\circ)$)}. 
 Then  the vector-functions
$\eta (\mathbf{x}) (\nabla \widetilde{\Phi}_k (\mathbf{x}) + \mathbf{c}_k)$ $(k=1,2,3)$ are constant. Hence, by 
\eqref{N_operator}, \eqref{Mjk}, we have $N(\boldsymbol{\theta})=0$ for any $\boldsymbol{\theta} \in \mathbb{S}^2$.
\end{remark}

\begin{remark}
\label{rem2_5}
$1^\circ$. According to  \cite[Proposition 4.2]{BSu2}, if $b(\boldsymbol{\theta})$ and $g(\mathbf{x})$ are matrices with real entries 
{\rm (}which is satisfied for the operator $\mathcal L${\rm )} and the vectors $\boldsymbol{\omega}_l (\boldsymbol{\theta})$, $l=1,2,3$, can be chosen real 
 {\rm(}for fixed  $\boldsymbol{\theta} \in \mathbb{S}^2${\rm)}, then $N_0(\boldsymbol{\theta})=0$.
These conditions are ensured provided that $\gamma_1(\boldsymbol{\theta}) \ne \gamma_2(\boldsymbol{\theta})$,
because the vector  $\boldsymbol{\omega}_3(\boldsymbol{\theta})$ is real \emph{(}see \eqref{omega3}\emph{)}, and, in the case under consideration, the eigenvectors of problem \eqref{solenoid_germ_spec_probl} 
 are determined  uniquely  {\rm (}up to phase factors{\rm )} and can be chosen real. For such $\boldsymbol{\theta}$ we have $N(\boldsymbol{\theta}) = N_*(\boldsymbol{\theta})$ and $\mu_l(\boldsymbol{\theta})=0$, $l=1,2,3$.

\noindent
$2^\circ$. If $\gamma_1(\boldsymbol{\theta}_0) = \gamma_2(\boldsymbol{\theta}_0)$ for some $\boldsymbol{\theta}_0 \in \mathbb{S}^2$, then, by Remarks {\rm \ref{rem_N}} and 
{\rm \ref{rem_NN}}, we have $N_*(\boldsymbol{\theta}_0)=0$ and $N(\boldsymbol{\theta}_0)= N_0(\boldsymbol{\theta}_0)$.
Herewith,  $\mu_{1}(\boldsymbol{\theta}_0)$ and $\mu_{2}(\boldsymbol{\theta}_0)$ are the eigenvalues of the operator \eqref{N_operator} in the subspace $J^0_{\boldsymbol{\theta}_0}$, they are given by 
$$
\mu_{1,2}(\boldsymbol{\theta}_0) = \pm f(\boldsymbol{\theta}_0) 
\frac{\langle \mu_0 \boldsymbol{\theta}_0, \boldsymbol{\theta}_0 \rangle^{1/2}}{( \operatorname{det}\mu_0)^{1/2}}.
$$
 If $f(\boldsymbol{\theta}_0) \ne 0$  \emph{(}and then also  $\mu_{1,2}(\boldsymbol{\theta}_0) \ne 0$\emph{)},
then the vectors $\boldsymbol{\omega}_{1,2}(\boldsymbol{\theta}_0)$ are determined uniquely  {\rm (}up to phase factors{\rm )} and coincide with the eigenvectors of the matrix $\mu_0^{-1/2} r(\boldsymbol{\theta}_0) \mu_0^{-1/2}$ corresponding to the eigenvalues  $\pm i \frac{\langle \mu_0 \boldsymbol{\theta}_0, \boldsymbol{\theta}_0 \rangle^{1/2}}{( \operatorname{det}\mu_0)^{1/2}}$.
\end{remark}

\subsection{The effective operator}
	We put
	\begin{equation}
	\label{effective_oper_symb}
	S (\mathbf{k}) := t^2 S (\boldsymbol{\theta}) = b(\mathbf{k})^* g^0 b(\mathbf{k}), \quad \mathbf{k} \in \mathbb{R}^{3}.
	\end{equation}
	Expression~(\ref{effective_oper_symb}) is the symbol of the DO
	\begin{equation}
	\label{L0}
	\mathcal{L}^0 = b(\mathbf{D})^* g^0 b(\mathbf{D}) = \mu_0^{-1/2} \operatorname{curl} (\eta^0)^{-1} \operatorname{curl} \mu_0^{-1/2} -  \mu_0^{1/2}\nabla  \underline{\nu} \operatorname{div} \mu_0^{1/2},
	\end{equation}
	acting in  $L_2(\mathbb{R}^3;\mathbb{C}^3)$ on the domain  
	$H^2(\mathbb{R}^3;\mathbb{C}^3)$
and called the  \emph{effective operator} for the operator $\mathcal{L}$.

\section{Homogenization of the operator $\mathcal{L}_\varepsilon$}
\label{L_eps_approx_section}

\subsection{The operator $\mathcal{L}_\varepsilon$}
\emph{Our main object}~is the operator $\mathcal{L}_\varepsilon$ acting in $L_2(\mathbb{R}^3;\mathbb{C}^3)$ and formally given by 
\begin{equation}
\label{L_eps}
\mathcal{L}_\varepsilon =  \mu_0^{-1/2} \operatorname{curl} (\eta^\varepsilon(\mathbf{x}))^{-1} 
\operatorname{curl} \mu_0^{-1/2} - \mu_0^{1/2} \nabla \nu^\varepsilon(\mathbf{x})  \operatorname{div} \mu_0^{1/2} = b(\mathbf{D})^* g^{\varepsilon}(\mathbf{x}) b(\mathbf{D}).
\end{equation}
The precise definition is given in terms of the corresponding quadratic form (cf.~Subsection~\ref{Subsection Operator L}).
The coefficients of the operator~(\ref{L_eps}) oscillate rapidly as  $\varepsilon \to 0$.
We obtain approximations of the operators $\cos(\tau \mathcal{L}_\varepsilon^{1/2})$ and $\mathcal{L}_\varepsilon^{-1/2} \sin(\tau \mathcal{L}_\varepsilon^{1/2})$ for small $\varepsilon$.

As well as  $\mathcal L$, the operator \eqref{L_eps} is reduced by the Weyl decomposition
\eqref{Weyl_decomp}. Its parts in the divergence-free and the gradient subspaces are denoted by
${\mathcal L}_{J,\varepsilon}$ and ${\mathcal L}_{G,\varepsilon}$, respectively.

Using that $\mathcal{L}_\varepsilon$ and $\mathcal{L}^0$ are simultaneously reduced by the Weyl decomposition  \eqref{Weyl_decomp} and taking Remark \ref{PJ} into account, we obtain the following simple statement.

\begin{lemma}\label{lemma}
Suppose that $\mathcal{L}_\varepsilon$~is the operator~\emph{(\ref{L_eps})}, and $\mathcal{L}^0$~is the effective operator~\emph{(\ref{L0})}.
Let $\mathcal{L}_{J,\varepsilon}$, $\mathcal{L}_{G,\varepsilon}$ be the parts of  $\mathcal{L}_{\varepsilon}$ in the subspaces $J(\mu_0)$ and $G(\mu_0)$, respectively. Let $\mathcal{L}_{J}^0$, $\mathcal{L}_{G}^0$ be the parts of the operator $\mathcal{L}^0$ in the subspaces  $J(\mu_0)$ and 
$G(\mu_0)$, respectively. 

\noindent $1^\circ$. The estimate 
\begin{equation*}
		\bigl\| \cos( \tau \mathcal{L}_{\varepsilon}^{1/2})  - \cos( \tau (\mathcal{L}^0)^{1/2})  \bigr\|_{H^s (\mathbb{R}^3) \to L_2 (\mathbb{R}^3)} \le \mathcal{C}(\tau) \varepsilon^{\sigma}
\end{equation*}
with some $s\ge 0$ and $\sigma \ge 0$ is equivalent to the pair of inequalities 
\begin{align*}
	\bigl\| \cos( \tau \mathcal{L}_{J,\varepsilon}^{1/2})  - \cos( \tau (\mathcal{L}_J^0)^{1/2}) \bigr\|_{J^s \to J} \le \mathcal{C}(\tau) \varepsilon^{\sigma},
\\
	\bigl\| \cos( \tau \mathcal{L}_{G,\varepsilon}^{1/2})  - \cos( \tau (\mathcal{L}_G^0)^{1/2}) \bigr\|_{G^s \to G} \le \mathcal{C}(\tau) \varepsilon^{\sigma}.
\end{align*}
Here for brevity we denote $J:=J(\mu_0)$, $G:=G(\mu_0)$, $J^s := J^s(\mu_0)$, $G^s := G^s(\mu_0)$.

\noindent
$2^\circ$. The estimate 
\begin{equation*}
	\bigl\| \mathcal{L}_{\varepsilon}^{-1/2} \sin( \tau \mathcal{L}_{\varepsilon}^{1/2})  - (\mathcal{L}^0)^{-1/2} \sin( \tau (\mathcal{L}^0)^{1/2}) \bigr\|_{H^s (\mathbb{R}^3) \to L_2 (\mathbb{R}^3)} \le {\mathcal{C}}(\tau) \varepsilon^{\sigma}
\end{equation*}
with some $s\ge 0$ and $\sigma \ge 0$ is equivalent to the pair of inequalities
\begin{align*}
	\bigl\| \mathcal{L}_{J,\varepsilon}^{-1/2} \sin( \tau \mathcal{L}_{J,\varepsilon}^{1/2})  - (\mathcal{L}_J^0)^{-1/2}\sin( \tau (\mathcal{L}_J^0)^{1/2})  \bigr\|_{J^s \to J} \le
{\mathcal{C}}(\tau) \varepsilon^{\sigma},
\\
	\bigl\| \mathcal{L}_{G,\varepsilon}^{-1/2} \sin( \tau \mathcal{L}_{G,\varepsilon}^{1/2})  -
(\mathcal{L}_G^0)^{-1/2} \sin( \tau (\mathcal{L}_G^0)^{1/2})  \bigr\|_{G^s \to G} \le
{\mathcal{C}}(\tau) \varepsilon^{\sigma}.
\end{align*}
$3^\circ$. The estimate 
\begin{equation*}
	\bigl\| \mathcal{L}_{\varepsilon}^{-1/2} \sin( \tau \mathcal{L}_{\varepsilon}^{1/2})D_j  - (\mathcal{L}^0)^{-1/2} \sin( \tau (\mathcal{L}^0)^{1/2}) D_j \bigr\|_{H^s (\mathbb{R}^3) \to L_2 (\mathbb{R}^3)} \le {\mathcal{C}}(\tau) \varepsilon^{\sigma}
\end{equation*}
with some $s\ge 0$ and $\sigma \ge 0$ is equivalent to the pair of inequalities
\begin{align*}
	\bigl\| \mathcal{L}_{J,\varepsilon}^{-1/2} \sin( \tau \mathcal{L}_{J,\varepsilon}^{1/2}) D_j  - (\mathcal{L}_J^0)^{-1/2}\sin( \tau (\mathcal{L}_J^0)^{1/2})  D_j \bigr\|_{J^s \to J} \le
{\mathcal{C}}(\tau) \varepsilon^{\sigma},
\\
	\bigl\| \mathcal{L}_{G,\varepsilon}^{-1/2} \sin( \tau \mathcal{L}_{G,\varepsilon}^{1/2}) D_j  -
(\mathcal{L}_G^0)^{-1/2} \sin( \tau (\mathcal{L}_G^0)^{1/2}) D_j  \bigr\|_{G^s \to G} \le
{\mathcal{C}}(\tau) \varepsilon^{\sigma}.
\end{align*}
Here $j=1,2,3$.
\end{lemma}

\subsection{Approximation for the operator-valued  functions of $\mathcal{L}_\varepsilon$ in the principal order}

For convenience of further references,  the following set of parameters is called the \textit{problem data}:
\begin{equation}
\label{problem_data}
|\mu_0|,\ |\mu_0^{-1}|,\ 
\| \eta \|_{L_\infty}, \ \| \eta^{-1} \|_{L_\infty},\ \| \nu \|_{L_\infty},\ \| \nu^{-1} \|_{L_\infty}; \  \text{the parameters of the lattice}\ \Gamma.
\end{equation}
 
The following theorem is a consequence of the general results for the class of operators 
$\mathcal{A}_\varepsilon$.

\begin{theorem}
		\label{cos_thrm1}
 	Let $\mathcal{L}_\varepsilon$~be the operator~\emph{(\ref{L_eps})}, and let $\mathcal{L}^0$~be the effective operator~\emph{(\ref{L0})}. Then for $\tau \in \mathbb{R}$ and $\varepsilon >0$ we have
 	 	\begin{gather}
	 	\label{cos_est}
	 	\bigl\| \cos( \tau \mathcal{L}_{\varepsilon}^{1/2})  - \cos( \tau (\mathcal{L}^0)^{1/2}) \bigr\|_{H^2 (\mathbb{R}^3) \to L_2 (\mathbb{R}^3)} \le {C}_1 (1+|\tau|) \varepsilon,
	 	\\
	 	\label{sin_est}
	 	\bigl\|\mathcal{L}_\varepsilon^{-1/2} \sin( \tau \mathcal{L}_\varepsilon^{1/2})  - (\mathcal{L}^0)^{-1/2} \sin( \tau (\mathcal{L}^0)^{1/2})
\bigr\|_{H^1 (\mathbb{R}^3) \to L_2 (\mathbb{R}^3)} \le {C}_2 (1+ |\tau| )  \varepsilon.
	 	\end{gather}
The constants ${C}_1$ and ${C}_2$ depend only on the problem data \eqref{problem_data}.
	  \end{theorem}	

 Estimate \eqref{cos_est} was obtained in  \cite[Theorem 13.1]{BSu4}, and estimate \eqref{sin_est} was proved in  \cite[Theorem 9.1]{M2} (see also \cite{M1}).

By interpolation,  Theorem \ref{cos_thrm1} implies the following result (see \cite[Theorem 13.2]{BSu4} and \cite[Corollary 15.3]{DSu4}).

\begin{theorem}
		\label{cos_thrm2}
 	Let $\mathcal{L}_\varepsilon$~be the operator~\emph{(\ref{L_eps})}, and let $\mathcal{L}^0$~be the effective operator~\emph{(\ref{L0})}. Then for $0 \le s \le 2$, $\tau \in \mathbb{R}$, and $\varepsilon >0$ we have 
	 	\begin{equation*}
	 	\bigl\| \cos( \tau \mathcal{L}_{\varepsilon}^{1/2})  - \cos( \tau (\mathcal{L}^0)^{1/2})  \bigr\|_{H^s (\mathbb{R}^3) \to L_2 (\mathbb{R}^3)} \le \mathcal{C}_1(s) (1+|\tau|)^{s/2} \varepsilon^{s/2},
	 	\end{equation*}
	 	\begin{equation*}
	 	\begin{split}
	 	\bigl\|\mathcal{L}_\varepsilon^{-1/2} \sin( \tau \mathcal{L}_\varepsilon^{1/2}) D_j  - (\mathcal{L}^0)^{-1/2} \sin( \tau (\mathcal{L}^0)^{1/2}) D_j  \bigr\|_{H^s (\mathbb{R}^3) \to L_2 (\mathbb{R}^3)}
 \le \mathcal{C}_2 (s) (1+ |\tau|)^{s/2}  \varepsilon^{s/2},
	 	\end{split}
	 	\end{equation*}
 $j=1,2,3$. The constants $\mathcal{C}_1(s)$ and $\mathcal{C}_2(s)$ depend on the problem data \eqref{problem_data} and on $s$.
	  \end{theorem}

As was shown in \cite{DSu1,DSu2,DSu4}, under some additional assumptions, the results of Theorems \ref{cos_thrm1} and \ref{cos_thrm2} can be improved. 

\begin{condition}
\label{cond1}
Let~$N(\boldsymbol{\theta})$ be the operator defined by~\emph{(\ref{N_operator})}, \eqref{Mjk}. Suppose that  \hbox{$N(\boldsymbol{\theta}) = 0$} for any $\boldsymbol{\theta} \in \mathbb{S}^{2}$ which is equivalent to the assumption that $f(\boldsymbol{\theta}) \equiv 0$.
\end{condition}

Theorem 15.2  from \cite{DSu4}  directly implies the following result.

	\begin{theorem}
		\label{cos_thrm3a}
		Let $\mathcal{L}_\varepsilon$~be the operator~\emph{(\ref{L_eps})}, and let  $\mathcal{L}^0$~be the effective operator~\emph{(\ref{L0})}. Suppose that Condition \emph{\ref{cond1}} is satisfied.
		Then for $\tau \in \mathbb{R}$ and $\varepsilon >0$ we have 
	 	\begin{equation}
	 	\label{cos_est3}
	 \bigl\| \cos( \tau \mathcal{L}_{\varepsilon}^{1/2})  - \cos( \tau (\mathcal{L}^0)^{1/2})  \bigr\|_{H^{3/2} (\mathbb{R}^3) \to L_2 (\mathbb{R}^3)} \le {C}_3 (1+|\tau|)^{1/2} \varepsilon,
	 	\end{equation}
	 	\begin{equation}
	 	\label{sin_est3}
\begin{aligned}
\bigl\|\mathcal{L}_\varepsilon^{-1/2} \sin( \tau \mathcal{L}_\varepsilon^{1/2})  - (\mathcal{L}^0)^{-1/2} \sin( \tau (\mathcal{L}^0)^{1/2})
\bigr\|_{H^{1/2} (\mathbb{R}^3) \to L_2 (\mathbb{R}^3)}  \le {C}_4 (1+ |\tau| )^{1/2}  \varepsilon.
	\end{aligned}
	\end{equation}
The constants ${C}_3$ and ${C}_4$ depend only on the problem data \eqref{problem_data}. 
		\end{theorem}

By interpolation, Theorem \ref{cos_thrm3a} implies the following result (see \cite[Corollary 15.4]{DSu4}).

\begin{theorem}
		\label{cos_thrm4}
 	Suppose that the assumptions of Theorem \emph{\ref{cos_thrm3a}} are satisfied.
 	Then for \hbox{$0 \le s \le 3/2$}, $\tau \in \mathbb{R}$, and $\varepsilon >0$ we have 
	 	\begin{equation}
	 	\label{cos_est4}
	 \bigl\| \cos( \tau \mathcal{L}_{\varepsilon}^{1/2})  - \cos( \tau (\mathcal{L}^0)^{1/2})  \bigr\|_{H^s (\mathbb{R}^3) \to L_2 (\mathbb{R}^3)} \le \mathcal{C}_3(s) (1+|\tau|)^{s/3} \varepsilon^{2s/3},
	 	\end{equation}
	 	\begin{equation}
	 	\label{sin_est4}
	 	\begin{split}
	 \bigl\|\mathcal{L}_\varepsilon^{-1/2} \sin( \tau \mathcal{L}_\varepsilon^{1/2}) D_j  - (\mathcal{L}^0)^{-1/2} \sin( \tau (\mathcal{L}^0)^{1/2}) D_j \bigr\|_{H^s (\mathbb{R}^3) \to L_2 (\mathbb{R}^3)}
 \le \mathcal{C}_4 (s) (1+ |\tau|)^{s/3}  \varepsilon^{2s/3},
	 	\end{split}
	 	\end{equation}
$j=1,2,3$. The constants $\mathcal{C}_3(s)$ and  $\mathcal{C}_4(s)$ depend on the problem data  \eqref{problem_data} and on $s$. 	
  \end{theorem}	

Note that the operators $\mathcal{L}_{J,\varepsilon}$ and $\mathcal{L}^0_J$ depend on the coefficient $\eta(\mathbf{x})$, but not on  $\nu({\mathbf x})$.
Conversely, $\mathcal{L}_{G,\varepsilon}$ and $\mathcal{L}^0_G$ depend on the coefficient $\nu(\mathbf{x})$, but not on  $\eta({\mathbf x})$.
Consider the operator $\check{\mathcal L}_\varepsilon$ with the initial coefficients $\nu(\mathbf{x})$, $\mu_0$
and the constant coefficient
$\check{\eta}$ (for simplicity, let $\check{\eta}= {\mathbf 1}_3$). By Remark \ref{N=0}($1^\circ$), such an operator satisfies Condition \ref{cond1}. Then, by Theorems \ref{cos_thrm3a} and \ref{cos_thrm4}, the operator
$\check{\mathcal L}_\varepsilon$ satisfies
estimates of the form  \eqref{cos_est3}--\eqref{sin_est4}. Using Lemma \ref{lemma}, we arrive at the following statement.

	\begin{corollary}
		\label{corollary}
		Let $\mathcal{L}_\varepsilon$~be the operator~\emph{(\ref{L_eps})}, and let $\mathcal{L}^0$~be the effective operator~\emph{(\ref{L0})}.
Let $\mathcal{L}_{G,\varepsilon}$ and $\mathcal{L}_{G}^0$ be the parts of the operators  $\mathcal{L}_{\varepsilon}$ and $\mathcal{L}^0$ in the subspace $G(\mu_0)$.
Then for  $\tau \in {\mathbb R}$ and $\varepsilon >0$ we have
		\begin{gather*}
	\bigl\| \cos( \tau \mathcal{L}_{G,\varepsilon}^{1/2})  - \cos( \tau (\mathcal{L}_G^0)^{1/2})
\bigr\|_{G^{3/2} \to G} \le \check{C}_3 (1+| \tau|)^{1/2} \varepsilon,
\\
	\bigl\|\mathcal{L}_{G,\varepsilon}^{-1/2} \sin( \tau \mathcal{L}_{G,\varepsilon}^{1/2}) -
(\mathcal{L}_G^0)^{-1/2} \sin( \tau (\mathcal{L}_G^0)^{1/2})  \bigr\|_{G^{1/2}  \to G} \le \check{C}_4 (1+ |\tau|)^{1/2}  \varepsilon.
		\end{gather*}
For $0\le s \le 3/2$, $\tau \in {\mathbb R}$, and $\varepsilon >0$ we have 
\begin{gather*}	
		\bigl\| \cos( \tau \mathcal{L}_{G,\varepsilon}^{1/2})  - \cos( \tau (\mathcal{L}_G^0)^{1/2})
\bigr\|_{G^s \to G} \le \check{\mathcal{C}}_3 (s) (1+ |\tau| )^{s/3} \varepsilon^{2s/3},
		\\
		\bigl\|\mathcal{L}_{G,\varepsilon}^{-1/2} \sin( \tau \mathcal{L}_{G,\varepsilon}^{1/2}) D_j -
(\mathcal{L}_G^0)^{-1/2} \sin( \tau (\mathcal{L}_G^0)^{1/2})  D_j \bigr\|_{G^s  \to G} 
\le  \check{\mathcal{C}}_4 (s)(1+| \tau|)^{s/3}  \varepsilon^{2s/3}, \quad j=1,2,3.
		\end{gather*}
	The constants $\check{C}_3$ and $\check{C}_4$ are controlled in terms of $|\mu_0|$, $|\mu_0^{-1}|$, $\|\nu\|_{L_\infty}$, $\|\nu^{-1}\|_{L_\infty}$, and the parameters of the lattice $\Gamma$.
The constants $\check{\mathcal C}_3(s)$ and $\check{\mathcal C}_4(s)$ depend on the same parameters and on $s$.
	\end{corollary}

Now we abandon the condition $N(\boldsymbol{\theta})\equiv 0$, and instead assume that 
$N_0(\boldsymbol{\theta}) \equiv 0$. However, in this case we have to impose an additional condition  about the spectrum of the germ $S(\boldsymbol{\theta})$.

\begin{condition}
\label{cond2}
		$1^\circ$. The operator $N_0(\boldsymbol{\theta})$ is equal to zero{\rm :} $N_0(\boldsymbol{\theta}) = 0$ for any $\boldsymbol{\theta} \in \mathbb{S}^2$. This is equivalent to  
		$\mu_1(\boldsymbol{\theta}) = \mu_2(\boldsymbol{\theta}) = 0$ for any $\boldsymbol{\theta} \in \mathbb{S}^2$.
$2^\circ$. The branches of the eigenvalues $\gamma_1 (\boldsymbol{\theta})$ and $\gamma_2 (\boldsymbol{\theta})$ of  problem \eqref{solenoid_germ_spec_probl} either do not intersect or coincide identically.
	\end{condition}
	
Note that the intersection of the branch  
$\gamma_3 (\boldsymbol{\theta}) = \underline{\nu} \langle \mu_0 \boldsymbol{\theta}, \boldsymbol{\theta} \rangle$
(see \eqref{gamma3}) with the branches 
$\gamma_1 (\boldsymbol{\theta})$ and $\gamma_2 (\boldsymbol{\theta})$ is allowed.	
Under Condition \ref{cond2}, in the case where  $\gamma_1 (\boldsymbol{\theta})$ and  $\gamma_2 (\boldsymbol{\theta})$ do not intersect, we denote 
$$
c^\circ:= \min_{\boldsymbol{\theta}\in \mathbb{S}^2} |\gamma_1 (\boldsymbol{\theta}) - \gamma_2 (\boldsymbol{\theta})|.
$$
By Remark \ref{rem2_5}, if 
$\gamma_1 (\boldsymbol{\theta})$ and  $\gamma_2 (\boldsymbol{\theta})$ do not intersect, then $N_0(\boldsymbol{\theta}) \equiv 0$ and Condition \ref{cond2} is valid automatically.

The following result  is deduced from \cite[Theorem 15.2]{DSu4}.

	\begin{theorem}
		\label{cos_thrm3aa}
		Let $\mathcal{L}_\varepsilon$~be the operator~\emph{(\ref{L_eps})} and let $\mathcal{L}^0$~be the effective operator~\emph{(\ref{L0})}. Suppose that Condition \emph{\ref{cond2}} is satisfied.
		Then for $\tau \in \mathbb{R}$ and  $\varepsilon >0$ we have 
	 	\begin{equation}
	 	\label{cos_est6}
	 \bigl\| \cos( \tau \mathcal{L}_{\varepsilon}^{1/2})  - \cos( \tau (\mathcal{L}^0)^{1/2})  \bigr\|_{H^{3/2} (\mathbb{R}^3) \to L_2 (\mathbb{R}^3)} \le {C}_5 (1+|\tau|)^{1/2} \varepsilon,
	 	\end{equation}
	 	\begin{equation}
	 	\label{sin_est6}
\begin{aligned}
\bigl\|\mathcal{L}_\varepsilon^{-1/2} \sin( \tau \mathcal{L}_\varepsilon^{1/2})  - (\mathcal{L}^0)^{-1/2} \sin( \tau (\mathcal{L}^0)^{1/2})
\bigr\|_{H^{1/2} (\mathbb{R}^3) \to L_2 (\mathbb{R}^3)} 
\le {C}_6 (1+ |\tau| )^{1/2}  \varepsilon.
	\end{aligned}
	\end{equation}
The constants ${C}_5$ and ${C}_6$ depend on the problem data  \eqref{problem_data}
and also on the parameter $c^\circ$.
		\end{theorem}

\begin{proof}
By Lemma \ref{lemma}, the required estimates  \eqref{cos_est6} and \eqref{sin_est6} are equivalent to similar estimates for the divergence-free and the gradient parts of the operator
${\mathcal L}_\varepsilon$. According to Corollary \ref{corollary}, these estimates are valid for the gradient part. 
So, the problem is reduced to the proof of the following estimates:
		\begin{gather}
\label{312}
		\bigl\| \cos( \tau \mathcal{L}_{J,\varepsilon}^{1/2})  - \cos( \tau (\mathcal{L}^0_J)^{1/2}) \bigr\|_{J^{3/2} \to J} \le {C}_5  (1+|\tau|)^{1/2} \varepsilon,
		\\
\label{313}
	\bigl\|\mathcal{L}_{J,\varepsilon}^{-1/2} \sin( \tau \mathcal{L}_{J,\varepsilon}^{1/2})  - (\mathcal{L}_J^0)^{-1/2} \sin( \tau (\mathcal{L}_J^0)^{1/2}) \bigr\|_{J^{1/2} \to J} \le {C}_6 (1+  |\tau| )^{1/2}  \varepsilon.
		\end{gather}

Consider the operator $\widehat{\mathcal L}_\varepsilon$ with the initial coefficients $\mu_0$, $\eta(\mathbf{x})$ and the constant coefficient $\widehat{\nu} = 2 |\mu_0^{-1}|^2 \|\eta^{-1}\|_{L_\infty}$. By \eqref{gammaj_est1}, such a choice  of the coefficient  $\widehat{\nu}$ ensures that the branch  
$\widehat{\gamma}_3(\boldsymbol{\theta})= \widehat{\nu} \langle \mu_0 \boldsymbol{\theta},\boldsymbol{\theta}\rangle$ does not intersect with $\gamma_1(\boldsymbol{\theta})$ and $\gamma_2(\boldsymbol{\theta})$.
Together with Condition \ref{cond2}, this ensures that Condition  9.7 from \cite{DSu4} is satisfied
(this condition means that  $N_0(\boldsymbol{\theta}) \equiv 0$ and the multiplicity of the spectrum of the germ $S(\boldsymbol{\theta})$ does not depend on $\boldsymbol{\theta}$). Then, by Theorem 15.2 from  \cite{DSu4}, 
estimates of the form \eqref{cos_est6}, \eqref{sin_est6}
are valid for the operator $\widehat{\mathcal L}_\varepsilon$.
Applying Lemma \ref{lemma} and taking into account that the divergence-free parts of the operators ${\mathcal L}_\varepsilon$ and $\widehat{\mathcal L}_\varepsilon$ coincide, we arrive at the required estimates \eqref{312}, \eqref{313}.
\end{proof}

By interpolation, we obtain the following result (it is deduced from \cite[Corollary 15.4]{DSu4} by analogy with the proof of Theorem \ref{cos_thrm3aa}).

\begin{theorem}
		\label{cos_thrm6}
 	Suppose that the assumptions of Theorem \emph{\ref{cos_thrm3aa}} are satisfied.
 	Then for \hbox{$0 \le s \le 3/2$}, $\tau \in \mathbb{R}$, and $\varepsilon >0$ we have 
 	 	\begin{equation*}
	 \bigl\| \cos( \tau \mathcal{L}_{\varepsilon}^{1/2})  - \cos( \tau (\mathcal{L}^0)^{1/2}) \bigr\|_{H^s (\mathbb{R}^3) \to L_2 (\mathbb{R}^3)} \le \mathcal{C}_5(s) (1+|\tau|)^{s/3} \varepsilon^{2s/3},
	 	\end{equation*}
	 	\begin{equation*}
	 	\begin{split}
	 \bigl\|\mathcal{L}_\varepsilon^{-1/2} \sin( \tau \mathcal{L}_\varepsilon^{1/2}) D_j  - (\mathcal{L}^0)^{-1/2} \sin( \tau (\mathcal{L}^0)^{1/2}) D_j \bigr\|_{H^s (\mathbb{R}^3) \to L_2 (\mathbb{R}^3)}
 \le \mathcal{C}_6 (s) (1+ |\tau|)^{s/3}  \varepsilon^{2s/3},
	 	\end{split}
	 	\end{equation*}
 $j=1,2,3$. The constants $\mathcal{C}_5(s)$ and $\mathcal{C}_6(s)$ depend on the problem data  \eqref{problem_data},  on $s$, and on the parameter~$c^\circ$. 
	  \end{theorem}

\subsection{Approximation for the operator $\mathcal{L}_{\varepsilon}^{-1/2} \sin( \tau \mathcal{L}_{\varepsilon}^{1/2})$  in the energy norm}

Approximation for the operator-valued function $\mathcal{L}_{\varepsilon}^{-1/2} \sin( \tau \mathcal{L}_{\varepsilon}^{1/2})$ in the  ``energy'' norm  (i.~e., the norm of operators acting from $H^s$ to $H^1$) follows from the results of  \cite{M2}, where the general class of the operators ${\mathcal A}_\varepsilon$ was considered.
In this approximation, a corrector is taken into account. In the general case, the corrector involves an auxiliary smoothing operator. However, under the additional assumption that the solution 
$\Lambda({\mathbf x})$ of problem \eqref{equation_for_Lambda} is a multiplier from $H^{2}$ to $H^1$,
we can get rid of the smoothing operator. In dimension $d\le 4$, this condition holds automatically.  
We are also interested in approximation of the so called  ``flux'', i.~e., of the operator 
$g^\varepsilon b({\mathbf D}) \mathcal{L}_\varepsilon^{-1/2} \sin( \tau \mathcal{L}_\varepsilon^{1/2})$ in the $(H^s \to L_2)$-norm.
From \cite[Theorem 9.8]{M2} we deduce the following result.

  \begin{theorem}	
  \label{th1_corrector}
		Let $\mathcal{L}_\varepsilon$~be the operator~\emph{(\ref{L_eps})}, and let $\mathcal{L}^0$~be the effective operator~\emph{(\ref{L0})}. 
		Then for $\tau \in \mathbb{R}$ and $0< \varepsilon \le 1$ we have 
\begin{equation*}
\begin{aligned}
\bigl\|\mathcal{L}_\varepsilon^{-1/2} \sin( \tau \mathcal{L}_\varepsilon^{1/2})  - 
\bigl(I + \varepsilon  \Lambda^\varepsilon  b({\mathbf D}) \bigr) (\mathcal{L}^0)^{-1/2} \sin( \tau (\mathcal{L}^0)^{1/2})
\bigr\|_{H^{2} (\mathbb{R}^3) \to H^1 (\mathbb{R}^3)} 
 \le {C}_7 (1+ |\tau| )  \varepsilon,
	\end{aligned}
\end{equation*}
\begin{equation*}
\begin{aligned}
 \bigl\| g^\varepsilon  b({\mathbf D}) \mathcal{L}_\varepsilon^{-1/2} \sin( \tau \mathcal{L}_\varepsilon^{1/2})  - 
\widetilde{g}^\varepsilon  b({\mathbf D}) (\mathcal{L}^0)^{-1/2} \sin( \tau (\mathcal{L}^0)^{1/2})
 \bigr\|_{H^{2} (\mathbb{R}^3) \to L_2 (\mathbb{R}^3)} 
\le {C}_8 (1+ |\tau| )  \varepsilon.
	\end{aligned}
\end{equation*}	
The constants ${C}_7$ and ${C}_8$ depend only on the problem data \eqref{problem_data}.
  \end{theorem}

We have 
$$
\Lambda^\varepsilon b({\mathbf D}) = \mu_0^{-1/2} \Psi^\varepsilon \operatorname{curl} \mu_0^{-1/2} + \mu_0^{1/2} 
(\nabla \rho)^\varepsilon \operatorname{div} \mu_0^{1/2}.
$$
Obviously, the first term is equal to zero on $G(\mu_0)$, and the second is equal to zero on $J(\mu_0)$.
Next,
$$
g^\varepsilon  b({\mathbf D}) = -i \begin{pmatrix} (\eta^\varepsilon)^{-1} \operatorname{curl} \mu_0^{-1/2} \\ {\nu}^\varepsilon \operatorname{div} \mu_0^{1/2} \end{pmatrix},
\quad
\widetilde{g}^\varepsilon b({\mathbf D}) = -i \begin{pmatrix} \bigl((\eta^0)^{-1} + \Sigma^\varepsilon \bigr) \operatorname{curl} \mu_0^{-1/2} \\ \underline{\nu}\, \operatorname{div} \mu_0^{1/2} \end{pmatrix}.
$$
Using these relations, it is easy to check the following analog of Lemma~\ref{lemma}.

\begin{lemma}
\label{lemma2}
$1^\circ$. The estimate
\begin{equation*}
	\bigl\| \mathcal{L}_{\varepsilon}^{-1/2} \sin( \tau \mathcal{L}_{\varepsilon}^{1/2})  - 
	\bigl(I + \varepsilon  \Lambda^\varepsilon b({\mathbf D}) \bigr) (\mathcal{L}^0)^{-1/2} \sin( \tau (\mathcal{L}^0)^{1/2})  \bigr\|_{H^s (\mathbb{R}^3) \to H^1 (\mathbb{R}^3)} \le \mathcal{C}(\tau) \varepsilon^{\sigma}
\end{equation*}
with some $s\ge 0$ and $\sigma \ge 0$ is equivalent to the pair of inequalities
\begin{align*}
	\bigl\| \mathcal{L}_{J,\varepsilon}^{-1/2} \sin( \tau \mathcal{L}_{J,\varepsilon}^{1/2})  - 
		\bigl(I + \varepsilon \mu_0^{-1/2} \Psi^\varepsilon  \operatorname{curl} \mu_0^{-1/2} \bigr) (\mathcal{L}_J^0)^{-1/2} \sin( \tau (\mathcal{L}_J^0)^{1/2})  \bigr\|_{J^s \to H^1} \le \mathcal{C}(\tau) \varepsilon^{\sigma},
\\
	\bigl\| \mathcal{L}_{G,\varepsilon}^{-1/2} \sin( \tau \mathcal{L}_{G,\varepsilon}^{1/2})  -
	\bigl(I +\varepsilon \mu_0^{1/2} (\nabla \rho)^\varepsilon  \operatorname{div} \mu_0^{1/2} \bigr) (\mathcal{L}_G^0)^{-1/2} \sin( \tau (\mathcal{L}_G^0)^{1/2})  \bigr\|_{G^s \to H^1} \le \mathcal{C}(\tau) \varepsilon^{\sigma}.
\end{align*}
$2^\circ$. The estimate 
\begin{equation*}
	\bigl\| g^\varepsilon  b({\mathbf D} )\mathcal{L}_{\varepsilon}^{-1/2} \sin( \tau \mathcal{L}_{\varepsilon}^{1/2})  - 
		\widetilde{g}^\varepsilon  b({\mathbf D}) (\mathcal{L}^0)^{-1/2} \sin( \tau (\mathcal{L}^0)^{1/2})  \bigr\|_{H^s (\mathbb{R}^3) \to L_2 (\mathbb{R}^3)} \le \mathcal{C}(\tau) \varepsilon^{\sigma}
\end{equation*}
with some $s\ge 0$ and $\sigma \ge 0$ is equivalent to the pair of inequalities
\begin{align*}
		\begin{split}
		&\bigl\| (\eta^\varepsilon)^{-1} \operatorname{curl}  \mu_0^{-1/2} \mathcal{L}_{J,\varepsilon}^{-1/2} \sin( \tau \mathcal{L}_{J,\varepsilon}^{1/2})  - 
		\bigl((\eta^0)^{-1} + \Sigma^\varepsilon \bigr) \operatorname{curl} \mu_0^{-1/2} (\mathcal{L}_J^0)^{-1/2} \sin( \tau (\mathcal{L}_J^0)^{1/2})  \bigr\|_{J^s \to L_2}
		 \\
		 &\qquad \le \mathcal{C}(\tau) \varepsilon^{\sigma},
		\end{split}
\\
		&\bigl\| \nu^\varepsilon \operatorname{div} \mu_0^{1/2} \mathcal{L}_{G,\varepsilon}^{-1/2} \sin( \tau \mathcal{L}_{G,\varepsilon}^{1/2})  -
	\underline{\nu} \, \operatorname{div} \mu_0^{1/2} (\mathcal{L}_G^0)^{-1/2} \sin( \tau (\mathcal{L}_G^0)^{1/2})  \bigr\|_{G^s \to L_2} \le \mathcal{C}(\tau) \varepsilon^{\sigma}.
\end{align*}
\end{lemma}

Next, under some additional assumptions (for instance, under Condition \ref{cond1}), the results of Theorem
\ref{th1_corrector} can be improved; see \cite{DSu4}. Now, in order to remove the smoothing operator in the corrector, it suffices to assume that the solution 
$\Lambda({\mathbf x})$ of problem \eqref{equation_for_Lambda} is a multiplier from $H^{3/2}$ to $H^1$.
In dimension $d\le 3$, this condition is valid automatically (see \cite[Proposition 14.25]{DSu4}). 
From \cite[Theorem 15.36]{DSu4} we obtain the following result.
  
 \begin{theorem}	
  \label{th2_corrector}
		Let $\mathcal{L}_\varepsilon$~be the operator~\emph{(\ref{L_eps})}, and let $\mathcal{L}^0$~be the effective operator~\emph{(\ref{L0})}. Suppose that Condition \emph{\ref{cond1}} is satisfied. 
		Then for $\tau \in \mathbb{R}$ and $0< \varepsilon \le 1$ we have 
\begin{equation*}
\begin{aligned}
\bigl\|\mathcal{L}_\varepsilon^{-1/2} \sin( \tau \mathcal{L}_\varepsilon^{1/2})  - 
\bigl( I + \varepsilon \Lambda^\varepsilon  b({\mathbf D}) \bigr) (\mathcal{L}^0)^{-1/2} \sin( \tau (\mathcal{L}^0)^{1/2})
\bigr\|_{H^{3/2} (\mathbb{R}^3) \to H^1 (\mathbb{R}^3)} 
\le {C}_9 (1+ |\tau| )^{1/2}  \varepsilon,
	\end{aligned}
\end{equation*}
\begin{equation*}
\begin{aligned}
\bigl\| g^\varepsilon b({\mathbf D}) \mathcal{L}_\varepsilon^{-1/2} \sin( \tau \mathcal{L}_\varepsilon^{1/2})  - 
\widetilde{g}^\varepsilon  b({\mathbf D}) (\mathcal{L}^0)^{-1/2} \sin( \tau (\mathcal{L}^0)^{1/2})
\bigr\|_{H^{3/2} (\mathbb{R}^3) \to L_2 (\mathbb{R}^3)} 
 \le {C}_{10} (1+ |\tau| )^{1/2}  \varepsilon.
	\end{aligned}
\end{equation*}	
The constants ${C}_9$ and ${C}_{10}$ depend only on the problem data \eqref{problem_data}.
  \end{theorem}

By analogy with the proof of Corollary \ref{corollary}, from Theorem \ref{th2_corrector} and Lemma \ref{lemma2} we deduce the following corollary.

	\begin{corollary}
		\label{corollary25}
		Let $\mathcal{L}_\varepsilon$~be the operator~\eqref{L_eps}, and let $\mathcal{L}^0$~be the effective operator~\eqref{L0}.
Let $\mathcal{L}_{G,\varepsilon}$ and $\mathcal{L}_{G}^0$ be the parts of the operators $\mathcal{L}_{\varepsilon}$ and $\mathcal{L}^0$ in the subspace  $G(\mu_0)$.
Then for $\tau \in {\mathbb R}$ and $0< \varepsilon \le 1$ we have 
\begin{equation}
\label{sin_est_corrector3}
\begin{aligned}
&\bigl\|\mathcal{L}_{G,\varepsilon}^{-1/2} \sin( \tau \mathcal{L}_{G,\varepsilon}^{1/2})  - 
\bigl(I + \varepsilon \mu_0^{1/2} (\nabla \rho)^\varepsilon  \operatorname{div} \mu_0^{1/2} \bigr)
(\mathcal{L}_G^0)^{-1/2} \sin( \tau (\mathcal{L}_G^0)^{1/2})
\bigr\|_{G^{3/2}  \to H^1}  
\\
& \qquad \le \check{C}_9 (1+ |\tau| )^{1/2}  \varepsilon,
	\end{aligned}
\end{equation}
\begin{equation*}
\begin{aligned}
	\bigl\| \nu^\varepsilon  \operatorname{div} \mu_0^{1/2} \mathcal{L}_{G,\varepsilon}^{-1/2} \sin( \tau \mathcal{L}_{G,\varepsilon}^{1/2})  - \underline{\nu} \, \operatorname{div} \mu_0^{1/2} (\mathcal{L}_G^0)^{-1/2} \sin( \tau (\mathcal{L}_G^0)^{1/2})  \bigr\|_{G^{3/2} \to L_2} 
		\le \check{C}_{10} (1+ |\tau| )^{1/2}  \varepsilon.
	\end{aligned}
\end{equation*}	
The constants $\check{C}_9$ and $\check{C}_{10}$ depend only on $|\mu_0|$, $|\mu_0^{-1}|$,		
	 $\|\nu\|_{L_\infty}$, $\|\nu^{-1}\|_{L_\infty}$, and the parameters of the lattice $\Gamma$.
	\end{corollary}

Now, using \cite[Theorem 15.36]{DSu4} together with Lemma \ref{lemma2} and Corollary \ref{corollary25}, we obtain the following result; cf. the proof of Theorem \ref{cos_thrm3aa}.
  
 \begin{theorem}	
  \label{th3_corrector}
		Let $\mathcal{L}_\varepsilon$~be the operator~\emph{(\ref{L_eps})}, and let $\mathcal{L}^0$~be the effective operator~\emph{(\ref{L0})}. Suppose that Condition \emph{\ref{cond2}} is satisfied. 
		Then for $\tau \in \mathbb{R}$ and $0< \varepsilon \le 1$ we have 
\begin{equation*}
\begin{aligned}
\bigl\|\mathcal{L}_\varepsilon^{-1/2} \sin( \tau \mathcal{L}_\varepsilon^{1/2})  - 
\bigl( I + \varepsilon \Lambda^\varepsilon  b({\mathbf D}) \bigr) (\mathcal{L}^0)^{-1/2} \sin( \tau (\mathcal{L}^0)^{1/2})
\bigr\|_{H^{3/2} (\mathbb{R}^3) \to H^1 (\mathbb{R}^3)} 
\le {C}_{11} (1+ |\tau| )^{1/2}  \varepsilon,
	\end{aligned}
\end{equation*}
\begin{equation*}
\begin{aligned}
 \bigl\| g^\varepsilon  b({\mathbf D}) \mathcal{L}_\varepsilon^{-1/2} \sin( \tau \mathcal{L}_\varepsilon^{1/2})  - 
\widetilde{g}^\varepsilon  b({\mathbf D}) (\mathcal{L}^0)^{-1/2} \sin( \tau (\mathcal{L}^0)^{1/2})
 \bigr\|_{H^{3/2} (\mathbb{R}^3) \to L_2 (\mathbb{R}^3)} 
 \le {C}_{12} (1+ |\tau| )^{1/2}  \varepsilon.
	\end{aligned}
\end{equation*}	
The constants ${C}_{11}$ and ${C}_{12}$ depend only on the problem data \eqref{problem_data} and $c^\circ$.
  \end{theorem}

In the interpolational results about approximation of the operator 
$\mathcal{L}_\varepsilon^{-1/2} \sin( \tau \mathcal{L}_\varepsilon^{1/2})$
in the energy norm,  we use the smoothing operator $\Pi_\varepsilon$ acting in  
$L_2({\mathbb R}^3;{\mathbb C}^4)$ and given by 
$$
(\Pi_\varepsilon {\mathbf u})({\mathbf x}) = (2\pi)^{-3/2} \intop_{\widetilde{\Omega}/ \varepsilon } e^{i \langle {\mathbf x}, \boldsymbol{\xi}\rangle} 
\widehat{{\mathbf u}}(\boldsymbol{\xi}) \, d\boldsymbol{\xi}.
$$
 Here $\widehat{{\mathbf u}}(\boldsymbol{\xi})$ is the Fourier-image of a function ${\mathbf u}({\mathbf x})$.

 \begin{theorem}	
  \label{th4_corrector}
		Let $\mathcal{L}_\varepsilon$~be the operator~\emph{(\ref{L_eps})}, and let  $\mathcal{L}^0$~be the effective operator~\emph{(\ref{L0})}. 
	Then for $0\le s \le 2$, $\tau \in \mathbb{R}$, and $\varepsilon >0$ we have 
\begin{equation}
\label{sin_est_corrector7}
\begin{aligned}
&\bigl\| {\mathbf D} \left(\mathcal{L}_\varepsilon^{-1/2} \sin( \tau \mathcal{L}_\varepsilon^{1/2})  - 
\bigl( I + \varepsilon \Lambda^\varepsilon \Pi_\varepsilon b({\mathbf D}) \bigr) (\mathcal{L}^0)^{-1/2} \sin( \tau (\mathcal{L}^0)^{1/2}) \right)
\bigr\|_{H^{s} (\mathbb{R}^3) \to L_2 (\mathbb{R}^3)} 
\\
&\qquad \le \mathcal{C}_{7}(s) (1+ |\tau| )^{s/2}  \varepsilon^{s/2},
	\end{aligned}
\end{equation}
\begin{equation}
\label{sin_est_corrector8}
\begin{aligned}
&\bigl\| g^\varepsilon b({\mathbf D}) \mathcal{L}_\varepsilon^{-1/2} \sin( \tau \mathcal{L}_\varepsilon^{1/2})  - 
\bigl( g^0 + (\widetilde{g}^\varepsilon  - g^0) \Pi_\varepsilon \bigr) b({\mathbf D}) (\mathcal{L}^0)^{-1/2} \sin( \tau (\mathcal{L}^0)^{1/2}) \bigr\|_{H^{s} (\mathbb{R}^3) \to L_2 (\mathbb{R}^3)} 
\\
&\qquad \le \mathcal{C}_{8}(s) (1+ |\tau| )^{s/2}  \varepsilon^{s/2}.
	\end{aligned}
\end{equation}	
The constants $\mathcal{C}_{7}(s)$ and $\mathcal{C}_{8}(s)$ depend only on the problem data \eqref{problem_data} and on~$s$.
  \end{theorem}	

\begin{proof}
Corollary 15.9 from \cite{DSu4} directly implies estimate \eqref{sin_est_corrector7} together with the inequality 
\begin{equation}
\label{sin_est_corrector8a}
\begin{aligned}
& \bigl\| g^\varepsilon  b({\mathbf D}) \mathcal{L}_\varepsilon^{-1/2} \sin( \tau \mathcal{L}_\varepsilon^{1/2})  - 
\widetilde{g}^\varepsilon \Pi_\varepsilon  b({\mathbf D}) (\mathcal{L}^0)^{-1/2} \sin( \tau (\mathcal{L}^0)^{1/2})
\bigr\|_{H^{s} (\mathbb{R}^3) \to L_2 (\mathbb{R}^3)} 
\\
&\qquad \le \widetilde{\mathcal{C}}_{8}(s) (1+ |\tau| )^{s/2}  \varepsilon^{s/2}.
	\end{aligned}
\end{equation}	
Take into account that 
\begin{equation}
\label{sin_est_corrector8b}
\begin{split}
\bigl\|{g}^0 (I- \Pi_\varepsilon ) b({\mathbf D}) (\mathcal{L}^0)^{-1/2} \sin( \tau (\mathcal{L}^0)^{1/2})
\bigr\|_{H^{s} (\mathbb{R}^3) \to L_2 (\mathbb{R}^3)} 
\le \|g\|_{L_\infty}^{1/2}  \|I- \Pi_\varepsilon \|_{H^s({\mathbb R}^3) \to L_2({\mathbb R}^3)}.
\end{split}
\end{equation}	
We have 
$$
 \|(I- \Pi_\varepsilon ) {\mathbf u} \|^2_{L_2({\mathbb R}^3)} = \intop_{{\mathbb R}^3 \setminus (\widetilde{\Omega}/ \varepsilon )} 
 |\widehat{{\mathbf u}}(\boldsymbol{\xi})|^2
\, d\boldsymbol{\xi} \le r_0^{-2\sigma} \varepsilon^{2\sigma} \| {\mathbf u}\|^2_{H^\sigma({\mathbb R}^3)},
$$
whence 
\begin{equation}
\label{sin_est_corrector8c}
 \| I- \Pi_\varepsilon  \|_{H^s({\mathbb R}^3) \to L_2({\mathbb R}^3)} \le 
 r_0^{-\sigma} \varepsilon^\sigma, \quad 0\le \sigma \le s.
\end{equation}
Relations \eqref{sin_est_corrector8a}, \eqref{sin_est_corrector8b}, and \eqref{sin_est_corrector8c} (with $\sigma = s/2$)
imply \eqref{sin_est_corrector8}.
\end{proof}

We need the following analog of Lemma \ref{lemma2}.

\begin{lemma}
\label{lemma3}
$1^\circ$. The estimate 
\begin{equation*}
	\bigl\| {\mathbf D} \left( \mathcal{L}_{\varepsilon}^{-1/2} \sin( \tau \mathcal{L}_{\varepsilon}^{1/2})  - 
		\bigl(I + \varepsilon \Lambda^\varepsilon \Pi_\varepsilon b({\mathbf D}) \bigr) (\mathcal{L}^0)^{-1/2} \sin( \tau (\mathcal{L}^0)^{1/2}) \right) \bigr\|_{H^s (\mathbb{R}^3) \to L_2 (\mathbb{R}^3)} \le \mathcal{C}(\tau) \varepsilon^{\sigma}
\end{equation*}
with some  $s\ge 0$ and  $\sigma \ge 0$ is equivalent to the pair of inequalities 
\begin{align*}
	\bigl\| {\mathbf D} \left( \mathcal{L}_{J,\varepsilon}^{-1/2} \sin( \tau \mathcal{L}_{J,\varepsilon}^{1/2})  - 
		\bigl(I + \varepsilon \mu_0^{-1/2} \Psi^\varepsilon \Pi_\varepsilon \operatorname{curl} \mu_0^{-1/2} \bigr)
		 (\mathcal{L}_J^0)^{-1/2} \sin( \tau (\mathcal{L}_J^0)^{1/2}) \right) 
		\bigr\|_{J^s \to L_2} \le \mathcal{C}(\tau) \varepsilon^{\sigma},
\\
	\bigl\| {\mathbf D} \left(\mathcal{L}_{G,\varepsilon}^{-1/2} \sin( \tau \mathcal{L}_{G,\varepsilon}^{1/2})  -
	\bigl(I + \varepsilon \mu_0^{1/2} (\nabla \rho)^\varepsilon \Pi_\varepsilon  \operatorname{div} \mu_0^{1/2} \bigr)
	 (\mathcal{L}_G^0)^{-1/2} \sin( \tau (\mathcal{L}_G^0)^{1/2}) \right) \bigr\|_{G^s \to L_2} \le \mathcal{C}(\tau) \varepsilon^{\sigma}.
\end{align*}
$2^\circ$. The estimate 
\begin{equation*}
	\bigl\| g^\varepsilon b({\mathbf D} )\mathcal{L}_{\varepsilon}^{-1/2} \sin( \tau \mathcal{L}_{\varepsilon}^{1/2})  - 
	\bigl( g^0 + (\widetilde{g}^\varepsilon  - g^0) \Pi_\varepsilon \bigr) b({\mathbf D}) (\mathcal{L}^0)^{-1/2} \sin( \tau (\mathcal{L}^0)^{1/2})  \bigr\|_{H^s (\mathbb{R}^3) \to L_2 (\mathbb{R}^3)} \le \mathcal{C}(\tau) \varepsilon^{\sigma}
\end{equation*}
with some $s\ge 0$ and $\sigma \ge 0$ is equivalent to the pair of inequalities 
\begin{align*}
		&\bigl\| (\eta^\varepsilon)^{-1} \operatorname{curl} \mu_0^{-1/2} 
		\mathcal{L}_{J,\varepsilon}^{-1/2} \sin( \tau \mathcal{L}_{J,\varepsilon}^{1/2})  - 
		\bigl((\eta^0)^{-1} + \Sigma^\varepsilon   \Pi_\varepsilon\bigr) \operatorname{curl} \mu_0^{-1/2} (\mathcal{L}_J^0)^{-1/2} \sin( \tau (\mathcal{L}_J^0)^{1/2})  \bigr\|_{J^s \to L_2} 
		\\
		&\qquad \le \mathcal{C}(\tau) \varepsilon^{\sigma},
\\
		& \bigl\| \nu^\varepsilon \operatorname{div} \mu_0^{1/2} \mathcal{L}_{G,\varepsilon}^{-1/2} \sin( \tau \mathcal{L}_{G,\varepsilon}^{1/2})  - \underline{\nu} \, \operatorname{div} \mu_0^{1/2}
	(\mathcal{L}_G^0)^{-1/2} \sin( \tau (\mathcal{L}_G^0)^{1/2})  \bigr\|_{G^s \to L_2} \le \mathcal{C}(\tau) \varepsilon^{\sigma}.
\end{align*}
\end{lemma}

Using  \cite[Corollary 15.12]{DSu4} and taking into account  \eqref{sin_est_corrector8b} and \eqref{sin_est_corrector8c} 
(with $\sigma = 2 s/3$), we deduce the following result.

 \begin{theorem}	
  \label{th5_corrector}
		Let $\mathcal{L}_\varepsilon$~be the operator~\emph{(\ref{L_eps})}, and let $\mathcal{L}^0$~be the effective operator~\emph{(\ref{L0})}. Suppose that Condition \emph{\ref{cond1}} is satisfied.
		Then for $0\le s \le 3/2$, $\tau \in \mathbb{R}$, and $\varepsilon >0$ we have
		\begin{equation*}
\begin{aligned}
&\bigl\| {\mathbf D} \left(\mathcal{L}_\varepsilon^{-1/2} \sin( \tau \mathcal{L}_\varepsilon^{1/2})  - 
\bigl( I + \varepsilon  \Lambda^\varepsilon  \Pi_\varepsilon  b({\mathbf D}) \bigr) (\mathcal{L}^0)^{-1/2} \sin( \tau (\mathcal{L}^0)^{1/2}) \right)
\bigr\|_{H^{s} (\mathbb{R}^3) \to L_2 (\mathbb{R}^3)} 
\\
&\qquad \le \mathcal{C}_{9}(s) (1+ |\tau| )^{s/3}  \varepsilon^{2s/3},
	\end{aligned}
\end{equation*}
\begin{equation*}
\begin{aligned}
&\bigl\| g^\varepsilon b({\mathbf D}) \mathcal{L}_\varepsilon^{-1/2} \sin( \tau \mathcal{L}_\varepsilon^{1/2})  - 
\bigl( g^0 + (\widetilde{g}^\varepsilon - g^0) \Pi_\varepsilon \bigr) b({\mathbf D}) (\mathcal{L}^0)^{-1/2} \sin( \tau (\mathcal{L}^0)^{1/2})
\bigr\|_{H^{s} (\mathbb{R}^3) \to L_2 (\mathbb{R}^3)} 
\\
&\qquad \le \mathcal{C}_{10}(s) (1+ |\tau| )^{s/3}  \varepsilon^{2s/3}.
	\end{aligned}
\end{equation*}	
The constants $\mathcal{C}_{9}(s)$ and $\mathcal{C}_{10}(s)$ depend only on the problem data \eqref{problem_data} and on~$s$.
  \end{theorem}

Similarly to the proof of Corollary \ref{corollary}, from Theorem \ref{th5_corrector} and Lemma \ref{lemma3} we deduce the following corollary.

	\begin{corollary}
		\label{corollary3}
		Let $\mathcal{L}_\varepsilon$~be the operator~\eqref{L_eps}, and let $\mathcal{L}^0$~be the effective operator~\eqref{L0}.
Let $\mathcal{L}_{G,\varepsilon}$ and $\mathcal{L}_{G}^0$ be the parts of the operators $\mathcal{L}_{\varepsilon}$ and $\mathcal{L}^0$ in the subspace  $G(\mu_0)$, respectively.
Then for  $0 \le s \le 3/2$, $\tau \in {\mathbb R}$, and $\varepsilon >0$ we have 
\begin{equation*}
\begin{aligned}
&\bigl\| {\mathbf D} \left(\mathcal{L}_{G,\varepsilon}^{-1/2} \sin( \tau \mathcal{L}_{G,\varepsilon}^{1/2})  -
	\bigl(I + \varepsilon  \mu_0^{1/2} (\nabla \rho)^\varepsilon  \Pi_\varepsilon  \operatorname{div} \mu_0^{1/2} \bigr)
	 (\mathcal{L}_G^0)^{-1/2} \sin( \tau (\mathcal{L}_G^0)^{1/2}) \right)\bigr\|_{G^{s}  \to L_2}  
\\
&\qquad \le \check{\mathcal C}_9(s) (1+ |\tau| )^{s/3}  \varepsilon^{2s/3},
	\end{aligned}
\end{equation*}
\begin{equation*}
\begin{aligned}
&\bigl\| \nu^\varepsilon  \operatorname{div} \mu_0^{1/2} \mathcal{L}_{G,\varepsilon}^{-1/2} \sin( \tau \mathcal{L}_{G,\varepsilon}^{1/2})  - \underline{\nu} \, \operatorname{div} \mu_0^{1/2}
	(\mathcal{L}_G^0)^{-1/2} \sin( \tau (\mathcal{L}_G^0)^{1/2})  \bigr\|_{G^s \to L_2}
\\
&\qquad  \le \check{\mathcal C}_{10}(s) (1+ |\tau| )^{s/3}  \varepsilon^{2s/3}.
	\end{aligned}
\end{equation*}	
The constants $\check{\mathcal C}_9(s)$ and $\check{\mathcal C}_{10}(s)$ depend only on  $|\mu_0|$, $|\mu_0^{-1}|$,	$\|\nu\|_{L_\infty}$, $\|\nu^{-1}\|_{L_\infty}$, the parameters of the lattice $\Gamma$, and on $s$.
	\end{corollary}

Combining  \cite[Corollary 15.12]{DSu4}, Lemma \ref{lemma3}, and Corollary \ref{corollary3}, we deduce the following result; cf. the proof of Theorem \ref{cos_thrm3aa}.

 \begin{theorem}	
  \label{th6_corrector}
		Let $\mathcal{L}_\varepsilon$~be the operator~\emph{(\ref{L_eps})}, and let $\mathcal{L}^0$~be the effective operator~\emph{(\ref{L0})}. Suppose that Condition \emph{\ref{cond2}} is satisfied.
		Then for  $0\le s \le 3/2$, $\tau \in \mathbb{R}$, and $\varepsilon >0$ we have 
\begin{equation*}
\begin{aligned}
& \bigl\| {\mathbf D} \left(\mathcal{L}_\varepsilon^{-1/2} \sin( \tau \mathcal{L}_\varepsilon^{1/2})  - 
(I + \varepsilon \Lambda^\varepsilon  \Pi_\varepsilon ) b({\mathbf D}) (\mathcal{L}^0)^{-1/2} \sin( \tau (\mathcal{L}^0)^{1/2}) \right)
\bigr\|_{H^{s} (\mathbb{R}^3) \to L_2 (\mathbb{R}^3)} 
\\
&\qquad \le \mathcal{C}_{11}(s) (1+ |\tau| )^{s/3}  \varepsilon^{2s/3},
	\end{aligned}
\end{equation*}
\begin{equation*}
\begin{aligned}
& \bigl\| g^\varepsilon  b({\mathbf D}) \mathcal{L}_\varepsilon^{-1/2} \sin( \tau \mathcal{L}_\varepsilon^{1/2})  - \bigl( g^0 +
(\widetilde{g}^\varepsilon  - g^0) \Pi_\varepsilon \bigr) b({\mathbf D}) (\mathcal{L}^0)^{-1/2} \sin( \tau (\mathcal{L}^0)^{1/2})
\bigr\|_{H^{s} (\mathbb{R}^3) \to L_2 (\mathbb{R}^3)} 
\\
&\qquad \le \mathcal{C}_{12}(s) (1+ |\tau| )^{s/3}  \varepsilon^{2s/3}.
	\end{aligned}
\end{equation*}	
The constants  $\mathcal{C}_{11}(s)$ and $\mathcal{C}_{12}(s)$ depend only on the problem data \eqref{problem_data}, on~$s$, and on $c^\circ$.
  \end{theorem}

\subsection{Approximation for the operator-valued functions of $\mathcal{L}_{J,\varepsilon}$}
Using Lemma \ref{lemma} and applying Theorems \ref{cos_thrm1}, \ref{cos_thrm2}, \ref{cos_thrm3a}, \ref{cos_thrm4},
\ref{cos_thrm3aa}, \ref{cos_thrm6} to the operator
$\widehat{\mathcal L}_\varepsilon$ with the initial coefficients $\mu_0, \eta(\mathbf{x})$ and the constant coefficient
$\widehat{\nu} = 2 |\mu_0^{-1}|^2 \|\eta^{-1}\|_{L_\infty}$, we obtain the following (combined) result.

\begin{theorem}
		\label{cos_thrm1_J}
Let $\mathcal{L}_{J,\varepsilon}$ be the part of the operator~\eqref{L_eps} in the subspace $J(\mu_0)$, and let $\mathcal{L}_{J}^0$ be the part of the effective operator~\eqref{L0} in the subspace $J(\mu_0)$.

\noindent$1^\circ$.
For $\tau \in \mathbb{R}$ and $\varepsilon >0$ we have 
	 	\begin{gather}
	 	\label{cos_thrm1_H^2_L2_est}
	 \bigl\| \cos( \tau \mathcal{L}_{J,\varepsilon}^{1/2})  - \cos( \tau (\mathcal{L}^0_J)^{1/2})  \bigr\|_{J^2 \to J} \le
\widehat{C}_1 (1+ |\tau|) \varepsilon,
	 	\\
	 	\label{sin_thrm1_H^1_L2_est}
	 \bigl\|\mathcal{L}_{J,\varepsilon}^{-1/2} \sin( \tau \mathcal{L}_{J,\varepsilon}^{1/2})
 - (\mathcal{L}^0_J)^{-1/2} \sin( \tau (\mathcal{L}^0_J)^{1/2})
\bigr\|_{J^1  \to J} \le \widehat{C}_2 (1+ |\tau|)  \varepsilon.
	 	\end{gather}
For $0 \le s \le 2$, $\tau \in \mathbb{R}$, and $\varepsilon >0$ we have 
	 	\begin{gather}
	 	\label{cos_thrm1_H^s_L2_est}
	 \bigl\| \cos( \tau \mathcal{L}_{J,\varepsilon}^{1/2})  - \cos( \tau (\mathcal{L}^0_J)^{1/2})  \bigr\|_{J^s \to J} \le
\widehat{\mathcal{C}}_1 (s) (1+|\tau|)^{s/2} \varepsilon^{s/2},
	 	\\
	 	\label{sin_thrm1_H^s_L2_est}
\begin{split}
	 \bigl\|\mathcal{L}_{J,\varepsilon}^{-1/2} \sin( \tau \mathcal{L}_{J,\varepsilon}^{1/2}) D_j
 - (\mathcal{L}^0_J)^{-1/2} \sin( \tau (\mathcal{L}^0_J)^{1/2}) D_j
\bigr\|_{J^s  \to J}
 \le \widehat{\mathcal{C}}_2 (s) (1+ |\tau|)^{s/2}  \varepsilon^{s/2},\quad j=1,2,3.
	 	\end{split}
	 	\end{gather}
The constants $\widehat{C}_1$ and $\widehat{C}_2$ are controlled in terms of the norms $|\mu_0|$, $|\mu_0^{-1}|$, 
$\|\eta\|_{L_\infty}$, $\|\eta^{-1}\|_{L_\infty}$, and the parameters of the lattice  $\Gamma$. The constants  
$\widehat{\mathcal{C}}_1(s)$ and $\widehat{\mathcal{C}}_2 (s)$ depend on the same parameters and on~$s$.

\noindent$2^\circ$. Suppose that Condition \emph{\ref{cond1}} or Condition \emph{\ref{cond2}} is satisfied.
Then for $\tau \in \mathbb{R}$ and $\varepsilon >0$ we have 
	 	\begin{gather}
	 	\label{cos_thrm1_H^3/2_L2_est}
	 \bigl\| \cos( \tau \mathcal{L}_{J,\varepsilon}^{1/2})  - \cos( \tau (\mathcal{L}^0_J)^{1/2}) \bigr\|_{J^{3/2} \to J} \le
\widehat{C}_3 (1+ |\tau|)^{1/2} \varepsilon,
	 	\\
	 	\label{sin_thrm1_H^1/2_L2_est}
	 \bigl\|\mathcal{L}_{J,\varepsilon}^{-1/2} \sin( \tau \mathcal{L}_{J,\varepsilon}^{1/2})
 - (\mathcal{L}^0_J)^{-1/2} \sin( \tau (\mathcal{L}^0_J)^{1/2})
\bigr\|_{J^{1/2}  \to J} \le \widehat{C}_4 (1+ |\tau|)^{1/2}  \varepsilon.
	 	\end{gather}
For $0 \le s \le 3/2$, $\tau \in \mathbb{R}$, and $\varepsilon >0$ we have 
	 	\begin{gather}
	 	\label{cos111}
	 \bigl\| \cos( \tau \mathcal{L}_{J,\varepsilon}^{1/2})  - \cos( \tau (\mathcal{L}^0_J)^{1/2})  \bigr\|_{J^s \to J} \le
\widehat{\mathcal{C}}_3 (s) (1+|\tau|)^{s/3} \varepsilon^{2s/3},
	 	\\
	 	\label{sin111}
\begin{split}
	 \bigl\|\mathcal{L}_{J,\varepsilon}^{-1/2} \sin( \tau \mathcal{L}_{J,\varepsilon}^{1/2}) D_j
 - (\mathcal{L}^0_J)^{-1/2} \sin( \tau (\mathcal{L}^0_J)^{1/2}) D_j
\bigr\|_{J^s  \to J} 
\le \widehat{\mathcal{C}}_4 (s) (1+|\tau|)^{s/3}  \varepsilon^{2s/3}, \quad j=1,2,3.
	 	\end{split}
	 	\end{gather}
Under Condition \emph{\ref{cond1}} the constants $\widehat{C}_3$ and $\widehat{C}_4$ are controlled in terms of the norms $|\mu_0|$, $|\mu_0^{-1}|$,  $\|\eta\|_{L_\infty}$, $\|\eta^{-1}\|_{L_\infty}$, and the parameters of the lattice $\Gamma$\emph{;} the constants  $\widehat{\mathcal{C}}_3 (s)$ and  
$\widehat{\mathcal{C}}_4 (s)$ depend on the same parameters and on $s$.
Under Condition \emph{\ref{cond2}} the constants depend also on   $c^\circ$.
  \end{theorem}

Similarly, using Lemmas \ref{lemma2}, \ref{lemma3} and applying Theorems  \ref{th1_corrector}, \ref{th2_corrector}, \ref{th3_corrector}, \ref{th4_corrector}, \ref{th5_corrector},  and \ref{th6_corrector} to the operator
$\widehat{\mathcal L}_\varepsilon$ with the initial coefficients  $\mu_0, \eta(\mathbf{x})$ and the constant coefficient 
$\widehat{\nu} = 2 |\mu_0^{-1}|^2 \|\eta^{-1}\|_{L_\infty}$, we obtain the following  (combined) result.

\begin{theorem}
		\label{cos_thrm2_J}
Let $\mathcal{L}_{J,\varepsilon}$ be the part of the operator~\eqref{L_eps} in the subspace 
$J(\mu_0)$ and let  $\mathcal{L}_{J}^0$ be the part of the effective operator~\eqref{L0} in the subspace $J(\mu_0)$.

\noindent$1^\circ$.
For $\tau \in \mathbb{R}$ and $0< \varepsilon \le 1$ we have 
	 	\begin{gather}
	 	\label{sin_thrm1_corr1}
	 	\begin{split}
	 	&\bigl\|\mathcal{L}_{J,\varepsilon}^{-1/2} \sin( \tau \mathcal{L}_{J,\varepsilon}^{1/2})
 - \bigl(I + \varepsilon \mu_0^{-1/2} \Psi^\varepsilon \operatorname{curl} \mu_0^{-1/2} \bigr)(\mathcal{L}^0_J)^{-1/2} \sin( \tau (\mathcal{L}^0_J)^{1/2}) \bigr\|_{J^2  \to H^1} 
\\
&\qquad \le \widehat{C}_7 (1+ |\tau|)  \varepsilon,
\end{split}
\\
	 	\label{sin_thrm1_corr2}
	 	\begin{split}
&\bigl\| (\eta^\varepsilon)^{-1} \operatorname{curl} \mu_0^{-1/2} \mathcal{L}_{J,\varepsilon}^{-1/2} \sin( \tau \mathcal{L}_{J,\varepsilon}^{1/2})
 - \bigl( (\eta^0)^{-1} +  \Sigma^\varepsilon \bigr) \operatorname{curl} \mu_0^{-1/2} (\mathcal{L}^0_J)^{-1/2} \sin( \tau (\mathcal{L}^0_J)^{1/2})
\bigr\|_{J^2  \to L_2} \\
&\qquad \le \widehat{C}_8 (1+ |\tau|)  \varepsilon.
	 	\end{split}
	 	\end{gather}
For $0 \le s \le 2$, $\tau \in \mathbb{R}$, and $\varepsilon >0$ we have 
	 	\begin{gather}
	 	\label{sin_thrm1_corr3}
 	\begin{split}
	 	&\bigl\| {\mathbf D} \left(\mathcal{L}_{J,\varepsilon}^{-1/2} \sin( \tau \mathcal{L}_{J,\varepsilon}^{1/2})
 - \bigl(I + \varepsilon  \mu_0^{-1/2} \Psi^\varepsilon\Pi_\varepsilon \operatorname{curl} \mu_0^{-1/2}\bigr) (\mathcal{L}^0_J)^{-1/2} \sin( \tau (\mathcal{L}^0_J)^{1/2})\right)
\bigr\|_{J^s  \to L_2} 
\\
&\qquad \le \widehat{\mathcal C}_7(s) (1+ |\tau|)^{s/2}  \varepsilon^{s/2},
\end{split}
\\
\label{sin_thrm1_corr4}
	 	\begin{split}
&\bigl\| (\eta^\varepsilon)^{-1} \operatorname{curl} \mu_0^{-1/2} \mathcal{L}_{J,\varepsilon}^{-1/2} \sin( \tau \mathcal{L}_{J,\varepsilon}^{1/2})
 - \bigl( (\eta^0)^{-1} + \Sigma^\varepsilon \Pi_\varepsilon \bigr)  \operatorname{curl} \mu_0^{-1/2} (\mathcal{L}^0_J)^{-1/2} \sin( \tau (\mathcal{L}^0_J)^{1/2})
\bigr\|_{J^s  \to L_2} \\
&\qquad \le \widehat{\mathcal C}_8(s) (1+ |\tau|)^{s/2}  \varepsilon^{s/2}.
	 	\end{split}
	 	\end{gather}
The constants $\widehat{C}_7$ and $\widehat{C}_8$ are controlled in terms of the norms $|\mu_0|$, $|\mu_0^{-1}|$, 
$\|\eta\|_{L_\infty}$, $\|\eta^{-1}\|_{L_\infty}$, and the parameters of the lattice $\Gamma$. The constants  
$\widehat{\mathcal{C}}_7(s)$ and $\widehat{\mathcal{C}}_8 (s)$ depend on the same parameters and on~$s$.

\noindent$2^\circ$. Suppose that Condition \emph{\ref{cond1}} or Condition  \emph{\ref{cond2}} is satisfied.
Then for $\tau \in \mathbb{R}$ and $0< \varepsilon \le 1$ we have 
	 	\begin{gather}
	 	\label{sin_thrm1_corr5}
	 	\begin{split}
	 	&\bigl\|\mathcal{L}_{J,\varepsilon}^{-1/2} \sin( \tau \mathcal{L}_{J,\varepsilon}^{1/2})
 - \bigl(I + \varepsilon \mu_0^{-1/2} \Psi^\varepsilon \operatorname{curl} \mu_0^{-1/2}\bigr) (\mathcal{L}^0_J)^{-1/2} \sin( \tau (\mathcal{L}^0_J)^{1/2}) \bigr\|_{J^{3/2}  \to H^1} 
\\
&\qquad \le \widehat{C}_9 (1+ |\tau|)^{1/2}  \varepsilon,
\end{split}
\\
	 	\label{sin_thrm1_corr6}
	 	\begin{split}
&\bigl\| (\eta^\varepsilon )^{-1} \operatorname{curl} \mu_0^{-1/2} \mathcal{L}_{J,\varepsilon}^{-1/2} \sin( \tau \mathcal{L}_{J,\varepsilon}^{1/2})
 - \bigl( (\eta^0)^{-1} +  \Sigma^\varepsilon \bigr) \operatorname{curl} \mu_0^{-1/2} (\mathcal{L}^0_J)^{-1/2} \sin( \tau (\mathcal{L}^0_J)^{1/2})
\bigr\|_{J^{3/2}  \to L_2} \\
&\qquad \le \widehat{C}_{10} (1+ |\tau|)^{1/2}  \varepsilon.
	 	\end{split}
	 	\end{gather}
For $0 \le s \le 3/2$, $\tau \in \mathbb{R}$, and $\varepsilon >0$ we have 
	 	\begin{gather}
	 	\label{sin_thrm1_corr7}
 	\begin{split}
	 	&\bigl\| {\mathbf D} \left(\mathcal{L}_{J,\varepsilon}^{-1/2} \sin( \tau \mathcal{L}_{J,\varepsilon}^{1/2})
 - \bigl(I + \varepsilon \mu_0^{-1/2} \Psi^\varepsilon \Pi_\varepsilon  \operatorname{curl} \mu_0^{-1/2}\bigr)(\mathcal{L}^0_J)^{-1/2} \sin( \tau (\mathcal{L}^0_J)^{1/2})\right)
\bigr\|_{J^s  \to L_2} 
\\
&\qquad \le \widehat{\mathcal C}_9(s) (1+ |\tau|)^{s/3}  \varepsilon^{2s/3},
\end{split}
\\
\label{sin_thrm1_corr8}
	 	\begin{split}
&\bigl\| (\eta^\varepsilon)^{-1} \operatorname{curl} \mu_0^{-1/2} \mathcal{L}_{J,\varepsilon}^{-1/2} \sin( \tau \mathcal{L}_{J,\varepsilon}^{1/2})
 - \bigl( (\eta^0)^{-1}+ \Sigma^\varepsilon  \Pi_\varepsilon  \bigr)  \operatorname{curl} \mu_0^{-1/2} (\mathcal{L}^0_J)^{-1/2} \sin( \tau (\mathcal{L}^0_J)^{1/2})
\bigr\|_{J^s  \to L_2} \\
&\qquad \le \widehat{\mathcal C}_{10}(s) (1+ |\tau|)^{s/3}  \varepsilon^{2 s/3}.
	 	\end{split}
	 	\end{gather}	 	
Under Condition  \emph{\ref{cond1}} the constants   $\widehat{C}_9$ and $\widehat{C}_{10}$ are controlled in terms of the norms $|\mu_0|$, $|\mu_0^{-1}|$,  $\|\eta\|_{L_\infty}$, $\|\eta^{-1}\|_{L_\infty}$, and the parameters of the lattice  $\Gamma$\emph{;} the constants $\widehat{\mathcal{C}}_9 (s)$ and 
$\widehat{\mathcal{C}}_{10} (s)$ depend on the same parameters and on  $s$.
Under Conditiion \emph{\ref{cond2}} these constants depend also on~$c^\circ$.
  \end{theorem}

\begin{remark}
Tracking  the dependence of the estimates on $\tau$, we can obtain  qualified estimates for small  $\varepsilon$ and large  $|\tau|$, which is of independent interest.

\noindent{\emph{1)}} Under the assumptions of Theorem \emph{\ref{cos_thrm1_J}}$(1^\circ)$ or Theorem \emph{\ref{cos_thrm2_J}}$(1^\circ)$, we can take $\tau = O(\varepsilon^{-\alpha})$,
$0< \alpha < 1$. Then the norms in  \eqref{cos_thrm1_H^2_L2_est},  \eqref{sin_thrm1_H^1_L2_est},
\eqref{sin_thrm1_corr1}, \eqref{sin_thrm1_corr2} are estimated by  $O(\varepsilon^{1-\alpha})$, and the norms in  \eqref{cos_thrm1_H^s_L2_est}, \eqref{sin_thrm1_H^s_L2_est}, \eqref{sin_thrm1_corr3}, \eqref{sin_thrm1_corr4} are of order $O(\varepsilon^{s(1-\alpha)/2})$.

\noindent{\emph{2)}} Under the assumptions of Theorem \emph{\ref{cos_thrm1_J}}$(2^\circ)$ or Theorem \emph{\ref{cos_thrm2_J}}$(2^\circ)$, we can take $\tau = O(\varepsilon^{-\alpha})$,
$0< \alpha < 2$. Then the norms in  \eqref{cos_thrm1_H^3/2_L2_est},  \eqref{sin_thrm1_H^1/2_L2_est},
\eqref{sin_thrm1_corr5}, \eqref{sin_thrm1_corr6}  are estimated by  $O(\varepsilon^{1-\alpha/2})$, and the norms in  \eqref{cos111}, \eqref{sin111}, \eqref{sin_thrm1_corr7}, \eqref{sin_thrm1_corr8} are of order $O(\varepsilon^{s(2-\alpha)/3})$.
\end{remark}

\subsection{The sharpness of the results}
Applying Theorem 13.6 from \cite{DSu2} and Theorem 15.15 from \cite{DSu4}, we arrive at the following statement confirming that, in the general case, Theorems \ref{cos_thrm1} and \ref{th1_corrector} are sharp regarding the type of the operator norm.

\begin{theorem}
	\label{s<2_thrm}
	Let  $N_0 (\boldsymbol{\theta})$ be the operator defined by~\emph{(\ref{N0})}. Suppose that
	$N_0 (\boldsymbol{\theta}_0) \ne 0$ at least for one point $\boldsymbol{\theta}_0 \in \mathbb{S}^{2}$. Then the following is true. 
	
\noindent $1^\circ$. Let $0 \ne \tau \in \mathbb{R}$ and $0 \le s < 2$. Then there does not exist a constant  
 $\mathcal{C} (\tau) > 0$ such that the estimate 
	\begin{equation}
	\label{s<2_est_imp}
	\bigl\| \cos(\tau \mathcal{L}_\varepsilon^{1/2})  - \cos(\tau (\mathcal{L}^0)^{1/2}) \bigr\|_{H^s(\mathbb{R}^3) \to L_2(\mathbb{R}^3)} \le \mathcal{C}(\tau) \varepsilon
	\end{equation}
	holds for all sufficiently small $\varepsilon > 0$.

\noindent $2^\circ$. Let $0 \ne \tau \in \mathbb{R}$ and $0 \le r < 1$. Then there does not exist a constant  $\mathcal{C} (\tau) > 0$ such that the estimate
	\begin{equation}
	\label{sharp2}
	\bigl\| \mathcal{L}_\varepsilon^{-1/2} \sin(\tau \mathcal{L}_\varepsilon^{1/2})  - 
	(\mathcal{L}^0)^{-1/2} \sin(\tau (\mathcal{L}^0)^{1/2}) \bigr\|_{H^r(\mathbb{R}^3) \to L_2(\mathbb{R}^3)} \le \mathcal{C}(\tau) \varepsilon
	\end{equation}
	holds for all sufficiently small $\varepsilon > 0$.

\noindent $3^\circ$. Let $0 \ne \tau \in \mathbb{R}$ and $0 \le s < 2$. Then there does not exist a constant  $\mathcal{C} (\tau) > 0$ such that the estimate 
	\begin{equation}
	\label{sharp3}
	\bigl\| \mathcal{L}_\varepsilon^{-1/2} \sin(\tau \mathcal{L}_\varepsilon^{1/2})  - 
	\bigl( I + \varepsilon \Lambda^\varepsilon \Pi_\varepsilon  b({\mathbf D})\bigr) (\mathcal{L}^0)^{-1/2} \sin(\tau (\mathcal{L}^0)^{1/2}) \bigr\|_{H^s(\mathbb{R}^3) \to L_2(\mathbb{R}^3)} \le \mathcal{C}(\tau) \varepsilon
	\end{equation}
	holds for all sufficiently small $\varepsilon > 0$.
\end{theorem}

By Remark \ref{rem2_5}, the condition $N_0 (\boldsymbol{\theta}_0) \ne 0$ is equivalent to the relations
$\gamma_1(\boldsymbol{\theta}_0) = \gamma_2(\boldsymbol{\theta}_0)$ and
$f(\boldsymbol{\theta}_0) \ne 0$, where $f(\boldsymbol{\theta})$ is defined by  \eqref{Mjk}.

Now,  from Theorem \ref{s<2_thrm} we deduce a similar result for the operator $\mathcal{L}_{J,\varepsilon}$ confirming the sharpness of Theorems 
\ref{cos_thrm1_J}($1^\circ$) and \ref{cos_thrm2_J}$(1^\circ)$.

\begin{theorem}
	\label{s<2_thrm_J}
Let $\mathcal{L}_{J,\varepsilon}$ be the part of the operator~\eqref{L_eps} in the subspace  $J(\mu_0)$, and let $\mathcal{L}_{J}^0$ be the part of the effective operator~\eqref{L0} in the subspace $J(\mu_0)$. Let  $N_0 (\boldsymbol{\theta})$ be the operator defined by~\emph{(\ref{N0})}. Suppose that   $N_0 (\boldsymbol{\theta}_0) \ne 0$ at least for one point $\boldsymbol{\theta}_0 \in \mathbb{S}^{2}$. 
	
\noindent $1^\circ$. Let $0 \ne \tau \in \mathbb{R}$ and $0 \le s < 2$. Then there does not exist a constant  $\widetilde{\mathcal{C}} (\tau) > 0$ such that the estimate 
	\begin{equation}
	\label{s<2_est_imp_J}
	\bigl\| \cos(\tau \mathcal{L}_{J,\varepsilon}^{1/2})  - \cos(\tau (\mathcal{L}^0_J)^{1/2}) \bigr\|_{J^s \to J} \le
\widetilde{\mathcal C}(\tau) \varepsilon
	\end{equation}
	holds for all sufficiently small $\varepsilon > 0$.

\noindent $2^\circ$. Let $0 \ne \tau \in \mathbb{R}$ and $0 \le r < 1$. Then there does not exist a constant  $\widetilde{\mathcal{C}} (\tau) > 0$ such that the estimate 
	\begin{equation}
	\label{sharp4}
	\bigl\| \mathcal{L}_{J,\varepsilon}^{-1/2} \sin(\tau \mathcal{L}_{J,\varepsilon}^{1/2})  - 
	(\mathcal{L}^0_J)^{-1/2} \sin(\tau (\mathcal{L}^0_J)^{1/2}) \bigr\|_{J^r \to J} \le
\widetilde{\mathcal C}(\tau) \varepsilon
	\end{equation}
	holds for all sufficiently small $\varepsilon > 0$.

\noindent $3^\circ$. Let $0 \ne \tau \in \mathbb{R}$ and $0 \le s < 2$. Then there does not exist a constant  $\widetilde{\mathcal{C}} (\tau) > 0$ such that the estimate 
	\begin{equation}
	\label{sharp5}
	\bigl\|\mathcal{L}_{J,\varepsilon}^{-1/2} \sin( \tau \mathcal{L}_{J,\varepsilon}^{1/2})
 - \bigl(I + \varepsilon \mu_0^{-1/2} \Psi^\varepsilon \operatorname{curl} \mu_0^{-1/2} \bigr)(\mathcal{L}^0_J)^{-1/2} \sin( \tau (\mathcal{L}^0_J)^{1/2})
\bigr\|_{J^s  \to H^1} 
 \le \widetilde{\mathcal C}(\tau)  \varepsilon
 \end{equation}
holds for all sufficiently small $\varepsilon > 0$.
\end{theorem}

\begin{proof} Let us check statement $1^\circ$.
 It suffices to assume that  $3/2 \le s <2$.
We prove by contradiction. Suppose that for some  $3/2 \le s <2$ and \hbox{$\tau \ne 0$} estimate 
\eqref{s<2_est_imp_J} holds. By Corollary \ref{corollary}, estimate 
	\begin{equation}
	\label{s<2_est_imp_G}
	\bigl\| \cos(\tau \mathcal{L}_{G,\varepsilon}^{1/2})  - \cos(\tau (\mathcal{L}^0_G)^{1/2}) \bigr\|_{G^s \to G} \le
\check{\mathcal C}(\tau) \varepsilon
	\end{equation}
	is also valid. According to Lemma \ref{lemma}, relations \eqref{s<2_est_imp_J} and \eqref{s<2_est_imp_G} imply 
 \eqref{s<2_est_imp} with the constant $\mathcal{C}(\tau) = \max\{\widetilde{\mathcal C}(\tau),\check{\mathcal C}(\tau)\}$. But this contradicts statement $1^\circ$ of Theorem \ref{s<2_thrm}.
 
 Statement $2^\circ$ is proved similarly.
 
 Let us check statement $3^\circ$. It suffices to assume that $3/2 \le s <2$.
Suppose that for some $3/2 \le s <2$ and \hbox{$\tau \ne 0$} estimate 
\eqref{sharp5} is satisfied for sufficiently small $\varepsilon$.  
By Corollary \ref{corollary25}, estimate \eqref{sin_est_corrector3} holds.
Then from Lemma \ref{lemma2} it follows that the estimate
\begin{equation}
\label{sharp6}
	\bigl\|\mathcal{L}_{\varepsilon}^{-1/2} \sin( \tau \mathcal{L}_{\varepsilon}^{1/2})
 - \bigl(I + \varepsilon \Lambda^\varepsilon b({\mathbf D}) \bigr)(\mathcal{L}^0)^{-1/2} \sin( \tau (\mathcal{L}^0)^{1/2})
\bigr\|_{H^s  \to H^1} 
 \le {\mathcal C}(\tau)  \varepsilon
\end{equation}
holds for sufficiently small $\varepsilon$. 
It remains to take into account the following estimate proved in  \cite[Section 14.7]{DSu4}:
\begin{equation}
\label{sharp7}
	\bigl\| \varepsilon \Lambda^\varepsilon (I - \Pi_\varepsilon) b({\mathbf D}) (\mathcal{L}^0)^{-1/2} \sin( \tau (\mathcal{L}^0)^{1/2})
\bigr\|_{H^{3/2}  \to H^1}  \le C  \varepsilon,\quad 0< \varepsilon  \le 1.
\end{equation}
By \eqref{sharp6} and \eqref{sharp7}, we conclude that estimate \eqref{sharp3} is valid for sufficiently small 
$\varepsilon$. But this contradicts statement  $3^\circ$ of Theorem \ref{s<2_thrm}.
 \end{proof}

Applying Theorem 15.17 from \cite{DSu4}, we obtain the following result confirming the sharpness of Theorems  \ref{cos_thrm1} and \ref{th1_corrector} regarding the dependence of the estimates on  $\tau$.

\begin{theorem}
	\label{th_time_sharp}
	Let $N_0 (\boldsymbol{\theta})$ be the operator defined by~\emph{(\ref{N0})}. Suppose that $N_0 (\boldsymbol{\theta}_0) \ne 0$ at least for one point $\boldsymbol{\theta}_0 \in \mathbb{S}^{2}$. 
	Then the following is true. 
	
\noindent $1^\circ$. Let $s \ge 2$. Then there does not exist a positive function $\mathcal{C} (\tau)$ such that 
$\lim_{\tau \to \infty} \mathcal{C}(\tau)/ |\tau| =0$ and estimate  \eqref{s<2_est_imp} holds for $\tau \in {\mathbb R}$ and sufficiently small $\varepsilon>0$.

\noindent $2^\circ$. Let $r \ge 1$. Then there does not exist a positive function $\mathcal{C} (\tau)$ such that $\lim_{\tau \to \infty} \mathcal{C}(\tau)/ |\tau| =0$ and estimate \eqref{sharp2} holds for $\tau \in {\mathbb R}$ and sufficiently small $\varepsilon>0$.

\noindent $3^\circ$. Let $s \ge 2$. Then there does not exist a positive function $\mathcal{C} (\tau)$ such that 
$\lim_{\tau \to \infty} \mathcal{C}(\tau)/ |\tau| =0$   and estimate \eqref{sharp3} holds for $\tau \in {\mathbb R}$ 
and sufficiently small $\varepsilon>0$.
\end{theorem}

Theorem \ref{th_time_sharp} implies a similar result for the operator $\mathcal{L}_{J,\varepsilon}$ confirming that Theorems 
\ref{cos_thrm1_J}($1^\circ$) and \ref{cos_thrm2_J}($1^\circ$) are sharp regarding the dependence of the estimates on $\tau$.

\begin{theorem}
	\label{th_time_sharp2}
Let $\mathcal{L}_{J,\varepsilon}$ be the part of the operator~\eqref{L_eps} in the  subspace $J(\mu_0)$, and let $\mathcal{L}_{J}^0$ be the part of the effective operator~\eqref{L0} in the subspace  $J(\mu_0)$. Let  $N_0 (\boldsymbol{\theta})$ be the operator defined by~\emph{(\ref{N0})}. Suppose that 
 $N_0 (\boldsymbol{\theta}_0) \ne 0$ at least for one point $\boldsymbol{\theta}_0 \in \mathbb{S}^{2}$. 

\noindent $1^\circ$. Let $s \ge 2$. Then there does not exist a positive function $\widetilde{\mathcal{C}} (\tau)$ such that $\lim_{\tau \to \infty} \widetilde{\mathcal{C}}(\tau)/ |\tau| =0$ and estimate  \eqref{s<2_est_imp_J} holds for $\tau \in {\mathbb R}$ and sufficiently small $\varepsilon>0$.

\noindent $2^\circ$. Let $r \ge 1$. Then there does not exist a positive function $\widetilde{\mathcal{C}} (\tau)$ such that 
$\lim_{\tau \to \infty} \widetilde{\mathcal{C}}(\tau)/ |\tau| =0$ and estimate \eqref{sharp4} holds for $\tau \in {\mathbb R}$ and sufficiently small $\varepsilon >0$.

\noindent $3^\circ$. Let $s \ge 2$. Then there does not exist a positive function $\widetilde{\mathcal{C}} (\tau)$ such that 
$\lim_{\tau \to \infty} \widetilde{\mathcal{C}}(\tau)/ |\tau| =0$  and estimate \eqref{sharp5} holds for $\tau \in {\mathbb R}$ and sufficiently small $\varepsilon >0$.	
\end{theorem}

\begin{proof}
Let us check statement $1^\circ$. We prove by contradiction. Suppose that for some 
$s \ge 2$ there exists a positive function 
$\widetilde{\mathcal{C}} (\tau)$ such that 
$\lim_{\tau \to \infty} \widetilde{\mathcal{C}}(\tau)/ |\tau| =0$  and estimate \eqref{s<2_est_imp_J} holds for $\tau \in {\mathbb R}$ and sufficiently small $\varepsilon>0$.
	By Corollary \ref{corollary}, the estimate 
	\begin{equation}
	\label{s<2_est_imp_G2}
	\bigl\| \cos(\tau \mathcal{L}_{G,\varepsilon}^{1/2})  - \cos(\tau (\mathcal{L}^0_G)^{1/2}) \bigr\|_{G^s \to G} \le
\check{C}_3(1+ |\tau|)^{1/2} \varepsilon
	\end{equation}
	is also satisfied.
 By Lemma \ref{lemma}, relations \eqref{s<2_est_imp_J} and \eqref{s<2_est_imp_G2} imply 
 \eqref{s<2_est_imp} with  $\mathcal{C}(\tau) = \max\{\widetilde{\mathcal C}(\tau),\check{C}_3(1+|\tau|)^{1/2}\}$. We have $\lim_{\tau \to \infty} {\mathcal{C}}(\tau)/ |\tau| =0$.
 But this contradicts  statement  $1^\circ$ of Theorem \ref{th_time_sharp}.
 
 Statement $2^\circ$ is proved similarly.
 
 Let us check statement $3^\circ$. Suppose that for some $s \ge 2$ there exists a positive function
$\widetilde{\mathcal{C}} (\tau)$ such that 
$\lim_{\tau \to \infty} \widetilde{\mathcal{C}}(\tau)/ |\tau| =0$ and  estimate \eqref{sharp5} holds for $\tau \in {\mathbb R}$ and sufficiently small $\varepsilon >0$.	
By Corollary \ref{corollary25}, estimate  \eqref{sin_est_corrector3} is satisfied.
Combining this with Lemma \ref{lemma2}, we conclude that  \eqref{sharp6} holds with
 $\mathcal{C}(\tau) = \max\{\widetilde{\mathcal C}(\tau),\check{C}_9(1+|\tau|)^{1/2}\}$. We have  
$\lim_{\tau \to \infty} {\mathcal{C}}(\tau)/ |\tau| =0$. It remains to take  \eqref{sharp7} into account.
Relations  \eqref{sharp6} and \eqref{sharp7} imply estimate \eqref{sharp3}. But this contradicts statement $3^\circ$ of Theorem \ref{th_time_sharp}.
\end{proof}

\subsection{Examples}

Concrete examples of both situations were given in  \cite[\S 4]{DSu3}. 

1) Let $\Gamma = (2 \pi \mathbb{Z})^3$. Assume that  $\mu_0 = \mathbf{1}$.
Suppose that the matrix $\eta(\mathbf{x})$ depends only on $x_1$ and is given by 
	\begin{equation*}
	\eta(\mathbf{x}) = 	\begin{pmatrix}
	\eta_1(x_1) & \eta_2(x_1) & 0 \\
	\eta_2(x_1) & \eta_3(x_1) & 0 \\
	0 & 0 & \eta_4(x_1)	
	\end{pmatrix},
	\end{equation*}
	where $\eta_j(x_1)$, $j = 1,2,3,4$,~are $(2 \pi)$-periodic real-valued functions. It is assumed that the 
	matrix-valued function $\eta(\mathbf{x})$ is bounded and uniformly positive definite.
In \cite[Section 4.1]{DSu3}, it was shown that the functions $\eta_j(x_1)$, $j=1,2,3,4,$ can be chosen so that 
 $\gamma_1(\boldsymbol{\theta}_0) = \gamma_2(\boldsymbol{\theta}_0)$
and $\mu_1(\boldsymbol{\theta}_0) = -\mu_2(\boldsymbol{\theta}_0) \ne 0$
for some  $\boldsymbol{\theta}_0 \in \mathbb{S}^2$.
Then  $N_0(\boldsymbol{\theta}_0) \ne 0$. We can apply general results (Theorems \ref{cos_thrm1_J}($1^\circ$) and \ref{cos_thrm2_J}$(1^\circ)$), and they are sharp both regarding the norm type and regarding the dependence of the estimates on  $\tau$.

2) Recall that some cases where $N(\boldsymbol{\theta}) \equiv 0$ were distinguished in Remark \ref{N=0}.
One more example borrowed from~\cite{Zh1} was discussed in \cite[Section 4.2]{DSu3}. 
Suppose that   $\mu_0 = \mathbf{1}$.
	Let $\Gamma = (2\pi \mathbb{Z})^3$, and choose the cell centred at zero: $\Omega = (-\pi, \pi)^3$. 
Let $B_1 = \{|\mathbf{x}| \le 1\}$~be the unit ball, $B_\vartheta$~be the ball concentric with $B_1$ and such that
 $|B_\vartheta| = \vartheta |B_1|$, $0 < \vartheta < 1$. Let $\eta(\mathbf{x})$ be the $\Gamma$-periodic matrix-valued function, on the cell given by 
 	\begin{equation*}
	\eta(\mathbf{x}) = a(\mathbf{x})I, \qquad  a(\mathbf{x}) = \left\lbrace \begin{aligned}
	&\kappa , & &\text{for} \; \mathbf{x} \in B_\vartheta,\\
	&1, & &\text{for} \; \mathbf{x} \in B_1 \setminus B_\vartheta,\\
	&1 + \tfrac{3 \vartheta (\kappa -1)}{3 + (1 - \vartheta)(\kappa -1)}, &  &\text{for} \; \mathbf{x} \in \Omega \setminus B_1,
	\end{aligned} \right.
	\end{equation*}
where $\kappa >0$.
As was shown in  \cite[Section 4.2]{DSu3}, in this example  $N (\boldsymbol{\theta}) = 0$ for any $\boldsymbol{\theta}\in \mathbb{S}^2$.
 
 In the case where $N (\boldsymbol{\theta}) \equiv 0$,  the general results can be improved: 
 we can apply Theorems~ \ref{cos_thrm1_J}($2^\circ$) and \ref{cos_thrm2_J}$(2^\circ)$.

\section{Homogenization of the nonstationary Maxwell system}

\subsection{Statement of the problem\label{sec4.1}}

Suppose that the dielectric permittivity is given by the rapidly oscillating matrix $\eta^\varepsilon ({\mathbf x})$,
and the magnetic permeability is equal to the constant matrix $\mu_0$.
Suppose that  $\eta({\mathbf x})$ and $\mu_0$ satisfy the assumptions of Subsection \ref{Subsection Operator L}.
 We use the following notation for the physical fields:

${\mathbf u}_\varepsilon ({\mathbf x},\tau)$ is the intensity of the electric field;

${\mathbf w}_\varepsilon ({\mathbf x},\tau) = \eta^\varepsilon({\mathbf x}) {\mathbf u}_\varepsilon({\mathbf x},\tau)$ is the electric displacement vector;

${\mathbf v}_\varepsilon ({\mathbf x},\tau)$ is the intensity of the magnetic field;

${\mathbf z}_\varepsilon({\mathbf x},\tau) = \mu_0 {\mathbf v}_\varepsilon ({\mathbf x},\tau)$ is the magnetic displacement vector.

Consider the following Cauchy problem for the nonstationary Maxwell system:
\begin{equation}
\label{63}
\left\{
\begin{aligned}
&\partial_\tau {\mathbf u}_\varepsilon({\mathbf x},\tau) = (\eta^\varepsilon ({\mathbf x}))^{-1}  \operatorname{curl} 
 {\mathbf v}_\varepsilon ({\mathbf x},\tau),
\quad \operatorname{div} \, \eta^\varepsilon( {\mathbf x}) {\mathbf u}_\varepsilon ({\mathbf x},\tau) =0,
\quad {\mathbf x} \in {\mathbb R}^3,\ \tau \in {\mathbb R};
\\
&\partial_\tau {\mathbf v}_\varepsilon({\mathbf x},\tau) = - \mu_0^{-1} \operatorname{curl}  {\mathbf u}_\varepsilon ({\mathbf x},\tau),
\quad  \operatorname{div} \, \mu_0 {\mathbf v}_\varepsilon ({\mathbf x},\tau) =0,
\quad {\mathbf x} \in {\mathbb R}^3,\ \tau \in {\mathbb R};
\\
&{\mathbf u}_\varepsilon({\mathbf x},0) = (P_\varepsilon {\mathbf f})({\mathbf x}),\quad  {\mathbf v}_\varepsilon({\mathbf x},0) = {\boldsymbol \phi}({\mathbf x}), \quad {\mathbf x} \in {\mathbb R}^3.
\end{aligned}
\right.
\end{equation}
Here ${\boldsymbol \phi} \in L_2({\mathbb R}^3;{\mathbb C}^3)$ and $ \operatorname{div} \mu_0 {\boldsymbol \phi}({\mathbf x}) =0$ 
(this relation is understood in the sense of distributions). Next, ${\mathbf f} \in L_2({\mathbb R}^3;{\mathbb C}^3)$ and $P_\varepsilon$ is the orthogonal projection of the weighted space $L_2({\mathbb R}^3;{\mathbb C}^3; \eta^\varepsilon )$ onto the subspace  
$$
\{ {\mathbf u} \in L_2({\mathbb R}^3;{\mathbb C}^3): \ \operatorname{div} \, \eta^\varepsilon ({\mathbf x}) {\mathbf u}({\mathbf x}) =0\}.
$$
The projection  $P_\varepsilon$ acts as follows: 
$(P_\varepsilon {\mathbf f})({\mathbf x}) = {\mathbf f}({\mathbf x}) - \nabla \omega_\varepsilon ({\mathbf x})$, where $\omega_\varepsilon$ is the solution of the equation
$\operatorname{div} \, \eta^\varepsilon  \nabla \omega_\varepsilon = \operatorname{div} \, \eta^\varepsilon {\mathbf f}$ (understood in the generalized sense):  $\omega_\varepsilon \in L_{2,\,\text{loc}}({\mathbb R}^3)$, $\nabla \omega_\varepsilon \in L_2({\mathbb R}^3;{\mathbb C}^3)$, and  
$$
\intop_{{\mathbb R}^3} \langle \eta^\varepsilon ({\mathbf x}) ({\mathbf f}({\mathbf x}) - \nabla \omega_\varepsilon({\mathbf x})), \nabla \chi( {\mathbf x})\rangle \, d {\mathbf x}=0,  
\quad \chi \in L_{2,\,\text{loc}} ({\mathbb R}^3),\ \nabla \chi \in L_2({\mathbb R}^3;{\mathbb C}^3).
$$

\subsection{The homogenized Maxwell system\label{sec4.2}}

We use the following notation for the homogenized physical fields:

${\mathbf u}_0({\mathbf x},\tau)$ is the intensity of the electric field;

${\mathbf w}_0({\mathbf x},\tau) = \eta^0 {\mathbf u}_0({\mathbf x},\tau)$ is the electric displacement vector;

${\mathbf v}_0({\mathbf x},\tau)$ is the intensity of the magnetic field;

${\mathbf z}_0({\mathbf x},\tau) = \mu_0 {\mathbf v}_0({\mathbf x},\tau)$ is the magnetic displacement vector.

Here $\eta^0$ is the effective matrix defined in Subsection \ref{sec_effective}.

The homogenized problem is given by 
\begin{equation}
\label{64}
\left\{
\begin{aligned}
&\partial_\tau {\mathbf u}_0({\mathbf x},\tau) = (\eta^0)^{-1}  \operatorname{curl}  {\mathbf v}_0({\mathbf x},\tau),
\quad \operatorname{div} \, \eta^0 {\mathbf u}_0({\mathbf x},\tau) =0, \quad {\mathbf x} \in {\mathbb R}^3,\ \tau \in {\mathbb R};
\\
&\partial_\tau {\mathbf v}_0({\mathbf x},\tau) = - \mu_0^{-1} \operatorname{curl}  {\mathbf u}_0({\mathbf x},\tau),\quad
 \operatorname{div} \, \mu_0 {\mathbf v}_0({\mathbf x},\tau) =0, \quad
{\mathbf x} \in {\mathbb R}^3,\ \tau \in {\mathbb R};
\\
&{\mathbf u}_0({\mathbf x},0) = (P_0 {\mathbf f})({\mathbf x}),\quad  {\mathbf v}_0({\mathbf x},0) = {\boldsymbol \phi}({\mathbf x}), \quad {\mathbf x} \in {\mathbb R}^3.
\end{aligned}
\right.
\end{equation}
Here  $P_0$ is the orthogonal projection of the weighted space $L_2({\mathbb R}^3;{\mathbb C}^3; \eta^0)$ onto the subspace 
$$
\{ {\mathbf u} \in L_2({\mathbb R}^3;{\mathbb C}^3): \ \operatorname{div} \, \eta^0 {\mathbf u}({\mathbf x}) =0\}.
$$
The projection $P_0$ acts as follows: 
$(P_0 {\mathbf f})({\mathbf x}) = {\mathbf f} ({\mathbf x}) - \nabla \omega_0({\mathbf x})$, where $\omega_0$ is the solution of the equation
$\operatorname{div} \, \eta^0 \nabla \omega_0 = \operatorname{div} \, \eta^0 {\mathbf f}$ (understood in the weak sense):  
$\omega_0 \in L_{2,\,\text{loc}}({\mathbb R}^3)$, $\nabla \omega_0 \in L_2({\mathbb R}^3;{\mathbb C}^3)$, and  
$$
\intop_{{\mathbb R}^3} \langle \eta^0 ({\mathbf f}({\mathbf x}) - \nabla \omega_0({\mathbf x})), \nabla \chi({\mathbf x})\rangle \, d{\mathbf x} =0,  
\quad \chi \in L_{2,\,\text{loc}}({\mathbb R}^3),\ \nabla \chi \in L_2({\mathbb R}^3;{\mathbb C}^3).
$$

\begin{remark}
Since $P_\varepsilon {\mathbf f} = {\mathbf f} - \nabla \omega_\varepsilon$ and $P_0 {\mathbf f} = {\mathbf f} - \nabla \omega_0$, then  
$$
\operatorname{curl} P_\varepsilon {\mathbf f} = \operatorname{curl} P_0 {\mathbf f} = \operatorname{curl} {\mathbf f},
$$
which is understood in the sense of distributions.
\end{remark}

\subsection{Reduction to the second order equation}

From \eqref{63} we obtain the second order equation for ${\mathbf v}_\varepsilon$:
$$
\partial^2_\tau {\mathbf v}_\varepsilon ({\mathbf x},\tau) = - \mu_0^{-1} \operatorname{curl} \partial_\tau {\mathbf u}_\varepsilon ({\mathbf x},\tau)=
- \mu_0^{-1}  \operatorname{curl} (\eta^\varepsilon ({\mathbf x}))^{-1}\operatorname{curl} {\mathbf v}_\varepsilon ({\mathbf x},\tau),
$$
with the initial conditions 
$$
{\mathbf v}_\varepsilon ({\mathbf x},0) = {\boldsymbol \phi}({\mathbf x}), \quad 
\partial_\tau {\mathbf v}_\varepsilon ({\mathbf x},0) = - \mu_0^{-1} \operatorname{curl} {\mathbf u}_\varepsilon ({\mathbf x},0) =
- \mu_0^{-1} \operatorname{curl} (P_\varepsilon {\mathbf f})({\mathbf x}) = - \mu_0^{-1} \operatorname{curl} {\mathbf f}({\mathbf x}).
$$

Thus, the magnetic intensity   ${\mathbf v}_\varepsilon ({\mathbf x},\tau)$ is the generalized solution of the following Cauchy problem 
\begin{equation}
\label{51}
\left\{
\begin{aligned}
&\mu_0 \,\partial_\tau^2 {\mathbf v}_\varepsilon ({\mathbf x},\tau) = -  \operatorname{curl} (\eta^\varepsilon(\mathbf{x}))^{-1} \operatorname{curl}  {\mathbf v}_\varepsilon ({\mathbf x},\tau),
\quad \operatorname{div} \, \mu_0 {\mathbf v}_\varepsilon ({\mathbf x},\tau) =0, \quad {\mathbf x} \in {\mathbb R}^3,\ \tau \in {\mathbb R};
\\
& {\mathbf v}_\varepsilon ({\mathbf x},0) = {\boldsymbol \phi}({\mathbf x}),\quad  \mu_0 \,\partial_\tau {\mathbf v}_\varepsilon ({\mathbf x},0) = {\boldsymbol \psi}({\mathbf x}),
\quad {\mathbf x} \in {\mathbb R}^3,
\end{aligned}
\right.
\end{equation}
where $\boldsymbol{\psi}:= - \operatorname{curl} {\mathbf f}$. 
Other fields are expressed in terms of  ${\mathbf v}_\varepsilon$ as follows: 
\begin{equation}
\label{51a}
\begin{aligned}
{\mathbf z}_\varepsilon ( {\mathbf x},\tau) &= \mu_0 {\mathbf v}_\varepsilon ({\mathbf x},\tau),
\\
 {\mathbf u}_\varepsilon ({\mathbf x},\tau) - {\mathbf u}_\varepsilon({\mathbf x},0) &= \int_0^\tau (\eta^\varepsilon ({\mathbf x}))^{-1} \operatorname{curl} {\mathbf v}_\varepsilon ({\mathbf x},\widetilde{\tau}) \, d\widetilde{\tau},
\\
 {\mathbf w}_\varepsilon ({\mathbf x},\tau) - {\mathbf w}_\varepsilon({\mathbf x},0) &= \int_0^\tau \operatorname{curl} {\mathbf v}_\varepsilon ({\mathbf x},\widetilde{\tau}) \, d\widetilde{\tau}.
\end{aligned}
\end{equation}

We substitute $\mu_0^{1/2} {\mathbf v}_\varepsilon  =  \boldsymbol{\varphi}_\varepsilon$.
Then $ \boldsymbol{\varphi}_\varepsilon$ is the solution of the problem  
\begin{equation*}
\left\{
\begin{aligned}
&\partial_\tau^2  \boldsymbol{\varphi}_\varepsilon ({\mathbf x},\tau) = - \mu_0^{-1/2} \!\operatorname{curl} (\eta^\varepsilon(\mathbf{x}))^{-1}\! \operatorname{curl} \mu_0^{-1/2}  \boldsymbol{\varphi}_\varepsilon({\mathbf x},\tau),
\ \;
 \operatorname{div}  \mu_0^{1/2}  \!\boldsymbol{\varphi}_\varepsilon({\mathbf x},\tau) =0, \ \; {\mathbf x} \in {\mathbb R}^3,\ \tau \in {\mathbb R};
\\
& \boldsymbol{\varphi}_\varepsilon({\mathbf x},0) = \mu_0^{1/2}{\boldsymbol \phi}({\mathbf x}),\quad  
 \partial_\tau  \boldsymbol{\varphi}_\varepsilon ({\mathbf x},0) = \mu_0^{-1/2}{\boldsymbol \psi}({\mathbf x}), \quad {\mathbf x} \in {\mathbb R}^3.
\end{aligned}
\right.
\end{equation*}
 The solution is represented as 
 $$
 \boldsymbol{\varphi}_\varepsilon = \cos (\tau {\mathcal L}_{J,\varepsilon}^{1/2}) \mu_0^{1/2}{\boldsymbol \phi}
 + {\mathcal L}_{J,\varepsilon}^{-1/2} \sin (\tau {\mathcal L}_{J,\varepsilon}^{1/2}) \mu_0^{-1/2}{\boldsymbol \psi}.
 $$
Hence,
\begin{equation}
\label{51b}
 {\mathbf v}_\varepsilon (\cdot,\tau) = \mu_0^{-1/2} \cos (\tau {\mathcal L}_{J,\varepsilon}^{1/2}) \mu_0^{1/2}{\boldsymbol \phi}
 + \mu_0^{-1/2}{\mathcal L}_{J,\varepsilon}^{-1/2} \sin (\tau {\mathcal L}_{J,\varepsilon}^{1/2}) \mu_0^{-1/2}{\boldsymbol \psi}.
 \end{equation}

Similarly,  the homogenized Maxwell system  \eqref{64} is reduced to the following problem for 
${\mathbf v}_0$: 
\begin{equation*}
\left\{
\begin{aligned}
& \mu_0\, \partial_\tau^2 {\mathbf v}_0({\mathbf x},\tau) = -  \operatorname{curl} (\eta^0)^{-1} \operatorname{curl}  {\mathbf v}_0({\mathbf x},\tau),
\quad
 \operatorname{div} \, \mu_0 {\mathbf v}_0({\mathbf x},\tau) =0, \quad {\mathbf x} \in {\mathbb R}^3,\ \tau \in {\mathbb R};
\\
& {\mathbf v}_0({\mathbf x},0) = {\boldsymbol \phi}({\mathbf x}),\quad \mu_0 \, \partial_\tau {\mathbf v}_0({\mathbf x},0) = {\boldsymbol \psi}({\mathbf x}),
\quad {\mathbf x} \in {\mathbb R}^3.
\end{aligned}
\right.
\end{equation*}
Other homogenized fields are expressed in terms of ${\mathbf v}_0$ as follows:
\begin{equation}
\label{53a}
\begin{aligned}
& {\mathbf z}_0({\mathbf x},\tau) = \mu_0 {\mathbf v}_0({\mathbf x},\tau),
\\
& {\mathbf u}_0({\mathbf x},\tau) - {\mathbf u}_0({\mathbf x},0) = \int_0^\tau (\eta^0)^{-1} \operatorname{curl} {\mathbf v}_0({\mathbf x},\widetilde{\tau}) \, d\widetilde{\tau},
\\
& {\mathbf w}_0({\mathbf x},\tau) - {\mathbf w}_0({\mathbf x},0) = \int_0^\tau \operatorname{curl} {\mathbf v}_0({\mathbf x},\widetilde{\tau}) \, d\widetilde{\tau}.
\end{aligned}
\end{equation}

Similarly to \eqref{51b}, we have 
\begin{equation}
\label{54}
{\mathbf v}_0(\cdot, \tau)= \mu_0^{-1/2} \cos(\tau (\mathcal{L}_{J}^0)^{1/2}) \mu_0^{1/2} {\boldsymbol \phi} + \mu_0^{-1/2}
(\mathcal{L}_{J}^0)^{-1/2} \sin( \tau (\mathcal{L}_{J}^0)^{1/2}) \mu_0^{-1/2} {\boldsymbol \psi}.
\end{equation}

\subsection{The results on homogenization of the Maxwell system}
From Theorem~\ref{cos_thrm1_J}($1^\circ$) we deduce the following result.

\begin{theorem}
	\label{thrm_51}
Under the assumptions of Subsections \emph{\ref{sec4.1}, \ref{sec4.2}}, the magnetic intensity 
${\mathbf v}_\varepsilon({\mathbf x},\tau)$ and the magnetic displacement vector ${\mathbf z}_\varepsilon({\mathbf x},\tau)$ satisfy the following statements. 

\noindent $1^\circ$. Let ${\boldsymbol \phi}, {\mathbf f} \in H^2({\mathbb R}^3;{\mathbb C}^3)$, and 
$\operatorname{div} \mu_0 {\boldsymbol \phi} =0$. 
Then for $\tau \in {\mathbb R}$ and $\varepsilon>0$ we have 
\begin{align}
\label{55}
\| {\mathbf v}_\varepsilon(\cdot, \tau) - {\mathbf v}_0(\cdot, \tau) \|_{L_2({\mathbb R}^3)} \le \mathfrak{C}_1(1+|\tau|)\varepsilon
\left( \| {\boldsymbol \phi} \|_{H^2({\mathbb R}^3)} +
 \| {\mathbf f}  \|_{H^2({\mathbb R}^3)}\right),
 \\
\label{55aaa}
\| {\mathbf z}_\varepsilon(\cdot, \tau) - {\mathbf z}_0(\cdot, \tau) \|_{L_2({\mathbb R}^3)} \le \mathfrak{C}_2 (1+|\tau|)\varepsilon
\left( \| {\boldsymbol \phi} \|_{H^2({\mathbb R}^3)} + \| {\mathbf f}  \|_{H^2({\mathbb R}^3)}\right). 
\end{align}
The constants $\mathfrak{C}_1$ and $\mathfrak{C}_2$ depend on  $|\mu_0|$, $|\mu_0^{-1}|$, $\|\eta\|_{L_\infty}$, $\|\eta^{-1}\|_{L_\infty}$, and the parameters of the lattice $\Gamma$.

\noindent $2^\circ$. Let ${\boldsymbol \phi}, {\mathbf f} \in H^s({\mathbb R}^3;{\mathbb C}^3)$, where $0\le s \le 2$, and $\operatorname{div} \mu_0 {\boldsymbol \phi} =0$. 
 Then for  $\tau \in {\mathbb R}$ and $\varepsilon>0$ we have 
\begin{align}
\label{55a}
\| {\mathbf v}_\varepsilon(\cdot, \tau) - {\mathbf v}_0(\cdot, \tau) \|_{L_2({\mathbb R}^3)} \le  \mathfrak{C}_3(s) (1+|\tau|)^{s/2}\varepsilon^{s/2}
\left( \| {\boldsymbol \phi} \|_{H^s({\mathbb R}^3)} + \| {\mathbf f}  \|_{H^s({\mathbb R}^3)}\right),
\\
\nonumber
\| {\mathbf z}_\varepsilon(\cdot, \tau) - {\mathbf z}_0(\cdot, \tau) \|_{L_2({\mathbb R}^3)} \le  \mathfrak{C}_4(s) (1+|\tau|)^{s/2}\varepsilon^{s/2}
\left( \| {\boldsymbol \phi} \|_{H^s({\mathbb R}^3)} + \| {\mathbf f}  \|_{H^s({\mathbb R}^3)}\right).
\end{align}
The constants $\mathfrak{C}_3(s)$ and $\mathfrak{C}_4(s)$ depend on  
$|\mu_0|$, $|\mu_0^{-1}|$, $\|\eta\|_{L_\infty}$, $\|\eta^{-1}\|_{L_\infty}$,
the parameters of the lattice $\Gamma$, and on  $s$.

\noindent $3^\circ$. If ${\boldsymbol \phi}, {\mathbf f} \in L_2({\mathbb R}^3;{\mathbb C}^3)$, and 
$\operatorname{div} \mu_0 {\boldsymbol \phi} =0$, then
\begin{align}
\label{57}
\lim_{\varepsilon \to 0} \| {\mathbf v}_\varepsilon(\cdot, \tau) - {\mathbf v}_0(\cdot, \tau) \|_{L_2({\mathbb R}^3)} =0,\quad \tau \in {\mathbb R},
\\
\label{57b}
\lim_{\varepsilon \to 0} \| {\mathbf z}_\varepsilon(\cdot, \tau) - {\mathbf z}_0(\cdot, \tau) \|_{L_2({\mathbb R}^3)} =0,\quad \tau \in {\mathbb R}.
\end{align}
\end{theorem}

\begin{proof}
Inequality \eqref{55} follows directly from    \eqref{cos_thrm1_H^2_L2_est},
\eqref{sin_thrm1_H^1_L2_est} and representations \eqref{51b}, \eqref{54}. Similarly, estimate \eqref{55a} is deduced from  \eqref{cos_thrm1_H^s_L2_est} and \eqref{sin_thrm1_H^s_L2_est}.
The results for ${\mathbf z}_\varepsilon$ directly follow from the results for ${\mathbf v}_\varepsilon$, since
 ${\mathbf z}_\varepsilon = \mu_0 {\mathbf v}_\varepsilon$ and ${\mathbf z}_0 = \mu_0 {\mathbf v}_0$.

 Estimate \eqref{55a} with $s=0$ shows that the norm on the left is uniformly bounded provided that  
${\boldsymbol \phi}, {\mathbf f} \in L_2({\mathbb R}^3;{\mathbb C}^3)$ (and $\operatorname{div} \mu_0 {\boldsymbol \phi} =0$).
Applying  \eqref{55a} with $s=0$ and \eqref{55} and using that   $H^2$ is dense in  $L_2$, and the set $\{ {\mathbf u} \in H^2: \operatorname{div} \mu_0 {\mathbf u} =0\}$ is dense in the space  $\{ {\mathbf u} \in L_2: 
\operatorname{div} \mu_0 {\mathbf u} =0\}$,
by the Banach--Steinhaus theorem, we obtain \eqref{57}.
Relation \eqref{57b}  follows from  \eqref{57}.
\end{proof}

Theorems \ref{s<2_thrm_J}($1^\circ,2^\circ$) and \ref{th_time_sharp2}($1^\circ, 2^\circ$) show that, in the general case, estimates \eqref{55} and \eqref{55aaa} are sharp regarding the norm type and regarding the dependence on  
$\tau$.  

However, under some additional assumptions, statements $1^\circ, 2^\circ$ of Theorem \ref{thrm_51} can be improved. This follows from Theorem~\ref{cos_thrm1_J}($2^\circ$).

\begin{theorem}
	\label{thrm_52}
Suppose that the assumptions of Theorem \emph{\ref{thrm_51}} are satisfied. 
Suppose that Condition \emph{\ref{cond1}} or Condition \emph{\ref{cond2}} is satisfied.

\noindent $1^\circ$. 
Let ${\boldsymbol \phi}, {\mathbf f} \in H^{3/2}({\mathbb R}^3;{\mathbb C}^3)$, and $\operatorname{div} \mu_0 {\boldsymbol \phi} =0$. 
Then for  $\tau \in {\mathbb R}$ and $\varepsilon>0$ we have 
\begin{align*}
\| {\mathbf v}_\varepsilon(\cdot, \tau) - {\mathbf v}_0(\cdot, \tau) \|_{L_2({\mathbb R}^3)} \le \mathfrak{C}_5 (1+|\tau|)^{1/2}\varepsilon
\left( \| {\boldsymbol \phi} \|_{H^{3/2}({\mathbb R}^3)} +
 \| {\mathbf f}  \|_{H^{3/2}({\mathbb R}^3)}\right),
 \\
\| {\mathbf z}_\varepsilon(\cdot, \tau) - {\mathbf z}_0(\cdot, \tau) \|_{L_2({\mathbb R}^3)} \le \mathfrak{C}_6 (1+|\tau|)^{1/2}\varepsilon
\left( \| {\boldsymbol \phi} \|_{H^{3/2}({\mathbb R}^3)} + \| {\mathbf f}  \|_{H^{3/2}({\mathbb R}^3)}\right). 
\end{align*}
Under Condition \emph{\ref{cond1}} the constants $\mathfrak{C}_5$ and $\mathfrak{C}_6$ depend on $|\mu_0|$, $|\mu_0^{-1}|$, $\|\eta\|_{L_\infty}$, $\|\eta^{-1}\|_{L_\infty}$, and the parameters of the lattice $\Gamma$.
Under Condition \emph{\ref{cond2}} these constants depend also on  $c^\circ$.

\noindent $2^\circ$. Let ${\boldsymbol \phi}, {\mathbf f} \in H^s({\mathbb R}^3;{\mathbb C}^3)$, where $0\le s \le 3/2$, and $\operatorname{div} \mu_0 {\boldsymbol \phi} =0$. 
 Then for $\tau \in {\mathbb R}$ and $\varepsilon>0$ we have
\begin{align*}
\| {\mathbf v}_\varepsilon(\cdot, \tau) - {\mathbf v}_0(\cdot, \tau) \|_{L_2({\mathbb R}^3)} \le  \mathfrak{C}_7(s) (1+|\tau|)^{s/3}\varepsilon^{2s/3}
\left( \| {\boldsymbol \phi} \|_{H^s({\mathbb R}^3)} + \| {\mathbf f}  \|_{H^s({\mathbb R}^3)}\right),
\\
\| {\mathbf z}_\varepsilon(\cdot, \tau) - {\mathbf z}_0(\cdot, \tau) \|_{L_2({\mathbb R}^3)} \le  \mathfrak{C}_8(s) (1+|\tau|)^{s/3}\varepsilon^{2s/3}
\left( \| {\boldsymbol \phi} \|_{H^s({\mathbb R}^3)} + \| {\mathbf f}  \|_{H^s({\mathbb R}^3)}\right).
\end{align*}
Under Condition \emph{\ref{cond1}} the constants $\mathfrak{C}_7(s)$ and $\mathfrak{C}_8(s)$ depend on
 $|\mu_0|$, $|\mu_0^{-1}|$, $\|\eta\|_{L_\infty}$, $\|\eta^{-1}\|_{L_\infty}$,
the parameters of the lattice $\Gamma$, and $s$.
Under Condition \emph{\ref{cond2}} these constants depend also on $c^\circ$.
\end{theorem}

\begin{remark}
In the case where ${\boldsymbol \phi} \ne 0$, we are not able to derive 
 approximation for the fields ${\mathbf u}_\varepsilon$ and ${\mathbf w}_\varepsilon$ from the known results for the operator ${\mathcal L}_{J,\varepsilon}$, because ${\mathbf u}_\varepsilon$ and ${\mathbf w}_\varepsilon$ are expressed in terms of the derivatives of  ${\mathbf v}_\varepsilon$ \emph{(}see \eqref{51a}\emph{)}, but we do not have approximation for the operator  
$\cos (\tau {\mathcal L}_{J,\varepsilon}^{1/2})$ in the energy norm.
\end{remark}

In the case where ${\boldsymbol \phi}=0$, we obtain approximation for all four fields applying 
Theorem \ref{cos_thrm2_J}($1^\circ$).

\begin{theorem}
	\label{thrm_53}
Under the assumptions of Subsections \emph{\ref{sec4.1} and \ref{sec4.2}}, suppose in addition that  
${\boldsymbol \phi}=0$.

\noindent $1^\circ$.
If ${\mathbf f}  \in H^3({\mathbb R}^3;{\mathbb C}^3)$, then for $\tau \in {\mathbb R}$ and $0< \varepsilon\le 1$ we have the following approximations for the fields  ${\mathbf v}_\varepsilon$ and ${\mathbf z}_\varepsilon$ in the energy norm\emph{:}
\begin{align}
\label{4.20}
\bigl\| {\mathbf v}_\varepsilon(\cdot, \tau) - {\mathbf v}_0(\cdot, \tau) - \varepsilon \mu_0^{-1} \Psi^\varepsilon \operatorname{curl} {\mathbf v}_0(\cdot, \tau) \bigr\|_{H^1({\mathbb R}^3)}
\le {\mathfrak C}_9 (1+ |\tau|) \varepsilon \| {\mathbf f} \|_{H^3({\mathbb R}^3)},
\\
\label{4.21}
\bigl\| {\mathbf z}_\varepsilon(\cdot, \tau) - {\mathbf z}_0(\cdot, \tau) - \varepsilon \Psi^\varepsilon  
\operatorname{curl} {\mathbf v}_0(\cdot, \tau) \bigr\|_{H^1({\mathbb R}^3)}
\le {\mathfrak C}_{10} (1+ |\tau|) \varepsilon \| {\mathbf f} \|_{H^3({\mathbb R}^3)}.
\\
\label{4.22}
\bigl\| (\eta^\varepsilon)^{-1} \operatorname{curl} {\mathbf v}_\varepsilon(\cdot, \tau) - ((\eta^0)^{-1} + \Sigma^\varepsilon) \operatorname{curl} {\mathbf v}_0(\cdot, \tau) \bigr\|_{L_2({\mathbb R}^3)}
\le {\mathfrak C}_{11} (1+ |\tau|) \varepsilon \| {\mathbf f} \|_{H^3({\mathbb R}^3)}.
\end{align}
If ${\mathbf f}  \in H^3({\mathbb R}^3;{\mathbb C}^3)$, then for $\tau \in {\mathbb R}$ and $0< \varepsilon \le 1$ we have the following approximations for the fields ${\mathbf u}_\varepsilon$ and ${\mathbf w}_\varepsilon$ in $L_2$\emph{:}
\begin{align}
\label{4.23}
\bigl\| \bigl({\mathbf u}_\varepsilon(\cdot, \tau) - {\mathbf u}_\varepsilon(\cdot,0)\bigr) - 
\bigl({\mathbf 1}+ \Sigma_\circ^\varepsilon \bigr) \bigl( {\mathbf u}_0(\cdot, \tau) - {\mathbf u}_0(\cdot, 0) \bigr) \bigr\|_{L_2({\mathbb R}^3)}
\le {\mathfrak C}_{11} |\tau|(1+ |\tau|) \varepsilon \| {\mathbf f} \|_{H^3({\mathbb R}^3)},
\\
\label{4.24}
\bigl\| \bigl( {\mathbf w}_\varepsilon(\cdot, \tau) - {\mathbf w}_\varepsilon(\cdot,0)\bigr) - 
\widetilde{\eta}^\varepsilon (\eta^0)^{-1} \bigl( {\mathbf w}_0(\cdot, \tau) - {\mathbf w}_0(\cdot, 0) \bigr) \bigr\|_{L_2({\mathbb R}^3)}
\le {\mathfrak C}_{12} |\tau|(1+ |\tau|) \varepsilon \| {\mathbf f} \|_{H^3({\mathbb R}^3)}.
\end{align}
The constants $\mathfrak{C}_9, \mathfrak{C}_{10}, \mathfrak{C}_{11}, \mathfrak{C}_{12}$ depend on $|\mu_0|$, $|\mu_0^{-1}|$, $\|\eta\|_{L_\infty}$, $\|\eta^{-1}\|_{L_\infty}$, and the parameters of the lattice $\Gamma$.

\noindent $2^\circ$.
Let ${\mathbf f} \in H^{1+s}({\mathbb R}^3;{\mathbb C}^3)$, where $0\le s \le 2$. Then for $\tau \in {\mathbb R}$ and $\varepsilon>0$ we have 
\begin{align*}
 \bigl\| {\mathbf D} \bigl({\mathbf v}_\varepsilon(\cdot, \tau) - {\mathbf v}_0(\cdot, \tau) - \varepsilon \mu_0^{-1} \Psi^\varepsilon  \Pi_\varepsilon \operatorname{curl} {\mathbf v}_0(\cdot, \tau) \bigr)
 \bigr\|_{L_2({\mathbb R}^3)}
\le {\mathfrak C}_{13}(s) (1+ |\tau|)^{s/2} \varepsilon^{s/2} \| {\mathbf f} \|_{H^{1+s}({\mathbb R}^3)},
\\
\bigl\| {\mathbf D} \bigl( {\mathbf z}_\varepsilon(\cdot, \tau) - {\mathbf z}_0(\cdot, \tau) - \varepsilon \Psi^\varepsilon \Pi_\varepsilon \operatorname{curl} {\mathbf v}_0(\cdot, \tau) \bigr)\bigr\|_{L_2({\mathbb R}^3)}
\le {\mathfrak C}_{14}(s) (1+ |\tau|)^{s/2} \varepsilon^{s/2} \| {\mathbf f} \|_{H^{1+s}({\mathbb R}^3)}.
\\
\bigl\| (\eta^\varepsilon )^{-1} \operatorname{curl} {\mathbf v}_\varepsilon (\cdot, \tau) - ( (\eta^0)^{-1} + \Sigma^\varepsilon \Pi_\varepsilon) \operatorname{curl} {\mathbf v}_0(\cdot, \tau) \bigr\|_{L_2({\mathbb R}^3)}
\le {\mathfrak C}_{15}(s) (1+ |\tau|)^{s/2} \varepsilon^{s/2} \| {\mathbf f} \|_{H^{1+s}({\mathbb R}^3)}.
\end{align*}
If ${\mathbf f} \in H^{1+s}({\mathbb R}^3;{\mathbb C}^3)$, where $0\le s \le 2$, then for $\tau \in {\mathbb R}$ and $\varepsilon>0$ we have 
\begin{align*}
\begin{split}
\bigl\| \bigl({\mathbf u}_\varepsilon (\cdot, \tau) - {\mathbf u}_\varepsilon(\cdot,0)\bigr) - 
\bigl(I+ \Sigma_\circ^\varepsilon  \Pi_\varepsilon \bigr) \bigl({\mathbf u}_0(\cdot, \tau) - {\mathbf u}_0(\cdot, 0) \bigr) \bigr\|_{L_2({\mathbb R}^3)}
\\
\le {\mathfrak C}_{15}(s) |\tau|(1+ |\tau|)^{s/2} \varepsilon^{s/2} \| {\mathbf f} \|_{H^{1+s}({\mathbb R}^3)},
\end{split}
\\
\begin{split}
\bigl\| \bigl( {\mathbf w}_\varepsilon(\cdot, \tau) - {\mathbf w}_\varepsilon(\cdot,0)\bigr) - 
\bigl( I +(\widetilde{\eta}^\varepsilon (\eta^0)^{-1} - {\mathbf 1})\Pi_\varepsilon  \bigr) \bigl( {\mathbf w}_0(\cdot, \tau) -
 {\mathbf w}_0(\cdot, 0) \bigr)
 \bigr\|_{L_2({\mathbb R}^3)}
\\
\le {\mathfrak C}_{16}(s) |\tau|(1+ |\tau|)^{s/2} \varepsilon^{s/2} \| {\mathbf f} \|_{H^{1+s}({\mathbb R}^3)}.
\end{split}
\end{align*}
The constants $\mathfrak{C}_{13}(s), \mathfrak{C}_{14}(s), \mathfrak{C}_{15}(s), \mathfrak{C}_{16}(s)$ depend on  
$|\mu_0|$, $|\mu_0^{-1}|$, $\|\eta\|_{L_\infty}$, $\|\eta^{-1}\|_{L_\infty}$, the parameters of the lattice  $\Gamma$, and $s$.

\noindent $3^\circ$.
If ${\mathbf f} \in H^{1}({\mathbb R}^3;{\mathbb C}^3)$, then for $\tau \in {\mathbb R}$ we have
$$
\begin{aligned}
&\lim_{\varepsilon \to 0} \bigl\| {\mathbf D} \bigl( {\mathbf v}_\varepsilon(\cdot, \tau) - {\mathbf v}_0(\cdot, \tau) - \varepsilon \mu_0^{-1} \Psi^\varepsilon \Pi_\varepsilon \operatorname{curl} {\mathbf v}_0(\cdot, \tau) \bigr)\bigr\|_{L_2({\mathbb R}^3)} =0,
\\
&\lim_{\varepsilon \to 0} 
\bigl\| {\mathbf D} \bigl( {\mathbf z}_\varepsilon(\cdot, \tau) - {\mathbf z}_0(\cdot, \tau) - \varepsilon  \Psi^\varepsilon  \Pi_\varepsilon \operatorname{curl} {\mathbf v}_0(\cdot, \tau) 
\bigr) \bigr\|_{L_2({\mathbb R}^3)}
=0,
\\
&\lim_{\varepsilon \to 0} 
\bigl\| (\eta^\varepsilon)^{-1} \operatorname{curl} {\mathbf v}_\varepsilon(\cdot, \tau) - ( (\eta^0)^{-1} + \Sigma^\varepsilon  \Pi_\varepsilon) \operatorname{curl} {\mathbf v}_0(\cdot, \tau)
\bigr\|_{L_2({\mathbb R}^3)}
=0,
\\
&\lim_{\varepsilon \to 0} 
\bigl\| \bigl({\mathbf u}_\varepsilon(\cdot, \tau) - {\mathbf u}_\varepsilon(\cdot,0)\bigr) - 
\bigl(I+ \Sigma_\circ^\varepsilon  \Pi_\varepsilon \bigr) \bigl({\mathbf u}_0(\cdot, \tau) - {\mathbf u}_0(\cdot, 0) \bigr) \bigr\|_{L_2({\mathbb R}^3)}
=0,
\\
&\lim_{\varepsilon \to 0} 
\bigl\| \bigl( {\mathbf w}_\varepsilon (\cdot, \tau) - {\mathbf w}_\varepsilon (\cdot,0)\bigr) - 
\bigl(I +(\widetilde{\eta}^\varepsilon  (\eta^0)^{-1} - {\mathbf 1})\Pi_\varepsilon  \bigr) \bigl( {\mathbf w}_0(\cdot, \tau) - {\mathbf w}_0(\cdot, 0) \bigr) \bigr\|_{L_2({\mathbb R}^3)} =0.
\end{aligned}
$$
\end{theorem}

\begin{proof}
Estimates \eqref{4.20} and \eqref{4.22} follow directly from 
\eqref{sin_thrm1_corr1}, \eqref{sin_thrm1_corr2},  representations
\eqref{51b}, \eqref{54}, and the relation $\boldsymbol{\psi} = - \operatorname{curl} {\mathbf f}$.
Inequality \eqref{4.21} follows from  \eqref{4.20} and the relations  ${\mathbf z}_\varepsilon = \mu_0 {\mathbf v}_\varepsilon$, ${\mathbf z}_0 = \mu_0 {\mathbf v}_0$.

 Next, integrating \eqref{4.22} in time and taking  \eqref{51a} and \eqref{53a} into account, we obtain \eqref{4.23}.
We have used that $\Sigma({\mathbf x}) \eta^0 = \Sigma_\circ({\mathbf x})$. Estimate \eqref{4.24} follows from  \eqref{4.23} and the relations ${\mathbf w}_\varepsilon = \eta^\varepsilon  {\mathbf u}_\varepsilon$, 
${\mathbf w}_0 = \eta^0 {\mathbf u}_0$.

Statement  $2^\circ$ is proved similarly with the help of \eqref{sin_thrm1_corr3} and \eqref{sin_thrm1_corr4}.

Statement $3^\circ$ follows from statement $2^\circ$, by the Banach--Steinhaus theorem.
\end{proof}

In \cite[Lemma 8.6]{BSu2}, it was shown that the weak 
 $(L_2 \to L_2)$-limit of the operator $[Y^\varepsilon] \Pi_{\varepsilon}$ is equal to zero if $Y({\mathbf x})$ is a  
$\Gamma$-periodic matrix-valued function with zero mean value. 
Using this property, we deduce the following corollary from statement $3^\circ$ of Theorem \ref{thrm_53}.

\begin{corollary}
If ${\boldsymbol \phi}=0$ and ${\mathbf f} \in H^{1}({\mathbb R}^3;{\mathbb C}^3)$, then for  $\tau \in {\mathbb R}$ and $\varepsilon \to 0$ we have
$$
\begin{aligned}
& {\mathbf v}_\varepsilon(\cdot, \tau) \to  {\mathbf v}_0(\cdot, \tau)\ \text{weakly in}\ H^1({\mathbb R}^3;{\mathbb C}^3);
\\
&{\mathbf z}_\varepsilon(\cdot, \tau) \to  {\mathbf z}_0(\cdot, \tau)\  \text{weakly in}\ H^1({\mathbb R}^3;{\mathbb C}^3);
\\
&(\eta^\varepsilon)^{-1} \operatorname{curl} {\mathbf v}_\varepsilon (\cdot, \tau) \to  (\eta^0)^{-1}
 \operatorname{curl} {\mathbf v}_0(\cdot, \tau)\  \text{weakly in}\ L_2({\mathbb R}^3;{\mathbb C}^3);
\\
&{\mathbf u}_\varepsilon(\cdot, \tau) - {\mathbf u}_\varepsilon(\cdot,0) \to  {\mathbf u}_0(\cdot, \tau) - {\mathbf u}_0(\cdot,0)\  \text{weakly in}\ L_2({\mathbb R}^3;{\mathbb C}^3);
\\
& {\mathbf w}_\varepsilon(\cdot, \tau) - {\mathbf w}_\varepsilon(\cdot,0) \to  {\mathbf w}_0(\cdot, \tau) - {\mathbf w}_0(\cdot,0)\ \text{weakly in}\ L_2({\mathbb R}^3;{\mathbb C}^3).
\end{aligned}
$$
\end{corollary}

Theorems \ref{s<2_thrm_J}($3^\circ$) and \ref{th_time_sharp2}($3^\circ$) show that, in the general case, estimates \eqref{4.20} and \eqref{4.21} are sharp regarding the norm type and regarding the dependence on  $\tau$.  However,
statements $1^\circ$ and $2^\circ$ of Theorem \ref{thrm_53} can be improved under some additional assumptions. The following result is deduced from Theorem \ref{cos_thrm2_J}($2^\circ$).

\begin{theorem}
	\label{thrm_54}
Under the assumptions of Subsections \emph{\ref{sec4.1}, \ref{sec4.2}}, suppose in addition that 
${\boldsymbol \phi}=0$. 
Suppose that Condition \emph{\ref{cond1}} or Condition \emph{\ref{cond2}} is satisfied.

\noindent $1^\circ$.
If ${\mathbf f}  \in H^{5/2}({\mathbb R}^3;{\mathbb C}^3)$, then for $\tau \in {\mathbb R}$ and $0< \varepsilon\le 1$ we have the following approximations for the fields  ${\mathbf v}_\varepsilon$ and ${\mathbf z}_\varepsilon$ in the energy norm\emph{:}
\begin{align*}
\bigl\| {\mathbf v}_\varepsilon (\cdot, \tau) - {\mathbf v}_0(\cdot, \tau) - \varepsilon \mu_0^{-1} \Psi^\varepsilon
 \operatorname{curl} {\mathbf v}_0(\cdot, \tau) \bigr\|_{H^1({\mathbb R}^3)}
\le {\mathfrak C}_{17} (1+ |\tau|)^{1/2} \varepsilon \| {\mathbf f} \|_{H^{5/2}({\mathbb R}^3)},
\\
\bigl\| {\mathbf z}_\varepsilon (\cdot, \tau) - {\mathbf z}_0(\cdot, \tau) - \varepsilon \Psi^\varepsilon
 \operatorname{curl} {\mathbf v}_0(\cdot, \tau) \bigr\|_{H^1({\mathbb R}^3)}
\le {\mathfrak C}_{18} (1+ |\tau|)^{1/2} \varepsilon \| {\mathbf f} \|_{H^{5/2}({\mathbb R}^3)}.
\\
\bigl\| (\eta^\varepsilon)^{-1} \operatorname{curl} {\mathbf v}_\varepsilon (\cdot, \tau) - ((\eta^0)^{-1} + \Sigma^\varepsilon) \operatorname{curl} {\mathbf v}_0(\cdot, \tau) \bigr\|_{L_2({\mathbb R}^3)}
\le {\mathfrak C}_{19} (1+ |\tau|)^{1/2} \varepsilon \| {\mathbf f} \|_{H^{5/2}({\mathbb R}^3)}.
\end{align*}
If ${\mathbf f}  \in H^{5/2}({\mathbb R}^3;{\mathbb C}^3)$, then for $\tau \in {\mathbb R}$ and $0< \varepsilon\le 1$ we have the following approximations for the fields ${\mathbf u}_\varepsilon$ and ${\mathbf w}_\varepsilon$ in $L_2$\emph{:}
\begin{align*}
\bigl\| \bigl({\mathbf u}_\varepsilon(\cdot, \tau) - {\mathbf u}_\varepsilon(\cdot,0)\bigr) - 
\bigl({\mathbf 1} + \Sigma_\circ^\varepsilon \bigr) \bigl({\mathbf u}_0(\cdot, \tau) - {\mathbf u}_0(\cdot, 0) \bigr) \bigr\|_{L_2({\mathbb R}^3)}
\le {\mathfrak C}_{19} |\tau|(1+ |\tau|)^{1/2} \varepsilon \| {\mathbf f} \|_{H^{5/2}({\mathbb R}^3)},
\\
\bigl\| \bigl( {\mathbf w}_\varepsilon(\cdot, \tau) - {\mathbf w}_\varepsilon(\cdot,0)\bigr) - 
\widetilde{\eta}^\varepsilon (\eta^0)^{-1} \bigl({\mathbf w}_0(\cdot, \tau) - {\mathbf w}_0(\cdot, 0) \bigr) \bigr\|_{L_2({\mathbb R}^3)}
\le {\mathfrak C}_{20} |\tau|(1+ |\tau|)^{1/2} \varepsilon \| {\mathbf f} \|_{H^{5/2}({\mathbb R}^3)}.
\end{align*}
Under Condition \emph{\ref{cond1}}, the constants $\mathfrak{C}_{17}, \mathfrak{C}_{18}, \mathfrak{C}_{19}, \mathfrak{C}_{20}$ depend on $|\mu_0|$, $|\mu_0^{-1}|$, $\|\eta\|_{L_\infty}$, $\|\eta^{-1}\|_{L_\infty}$, and the parameters of the lattice $\Gamma$. Under Condition \emph{\ref{cond2}}, these constants depend also on $c^\circ$.

\noindent $2^\circ$.
Let ${\mathbf f} \in H^{1+s}({\mathbb R}^3;{\mathbb C}^3)$, where $0\le s \le 3/2$.
Then for $\tau \in {\mathbb R}$ and $\varepsilon>0$ we have
\begin{align*}
\bigl\| {\mathbf D} \bigl( {\mathbf v}_\varepsilon(\cdot, \tau) - {\mathbf v}_0(\cdot, \tau) - \varepsilon \mu_0^{-1} \Psi^\varepsilon \Pi_\varepsilon \operatorname{curl} {\mathbf v}_0(\cdot, \tau) \bigr)
\bigr\|_{L_2({\mathbb R}^3)}
\le {\mathfrak C}_{21}(s) (1+ |\tau|)^{s/3} \varepsilon^{2s/3} \| {\mathbf f} \|_{H^{1+s}({\mathbb R}^3)},
\\
\bigl\| {\mathbf D} \bigl( {\mathbf z}_\varepsilon(\cdot, \tau) - {\mathbf z}_0(\cdot, \tau) - \varepsilon \Psi^\varepsilon \Pi_\varepsilon \operatorname{curl} {\mathbf v}_0(\cdot, \tau) \bigr)
\bigr\|_{L_2({\mathbb R}^3)}
\le {\mathfrak C}_{22}(s) (1+ |\tau|)^{s/3} \varepsilon^{2s/3} \| {\mathbf f} \|_{H^{1+s}({\mathbb R}^3)}.
\\
\bigl\| (\eta^\varepsilon)^{-1} \operatorname{curl} {\mathbf v}_\varepsilon(\cdot, \tau) - ( (\eta^0)^{-1} + \Sigma^\varepsilon \Pi_\varepsilon) \operatorname{curl} {\mathbf v}_0(\cdot, \tau) \bigr\|_{L_2({\mathbb R}^3)}
\le {\mathfrak C}_{23}(s) (1+ |\tau|)^{s/3} \varepsilon^{2s/3} \| {\mathbf f} \|_{H^{1+s}({\mathbb R}^3)}.
\end{align*}
If ${\mathbf f} \in H^{1+s}({\mathbb R}^3;{\mathbb C}^3)$, where $0\le s \le 3/2$, then for $\tau \in {\mathbb R}$ and $\varepsilon>0$ we have 
\begin{align*}
\begin{split}
\bigl\| \bigl({\mathbf u}_\varepsilon(\cdot, \tau) - {\mathbf u}_\varepsilon(\cdot,0)\bigr) - 
\bigl(I+ \Sigma_\circ^\varepsilon \Pi_\varepsilon \bigr) \bigl( {\mathbf u}_0(\cdot, \tau) - {\mathbf u}_0(\cdot, 0) \bigr) \bigr\|_{L_2({\mathbb R}^3)}
\\
\le {\mathfrak C}_{23}(s) |\tau|(1+ |\tau|)^{s/3} \varepsilon^{2s/3} \| {\mathbf f} \|_{H^{1+s}({\mathbb R}^3)},
\end{split}
\\
\begin{split}
\bigl\| \bigl({\mathbf w}_\varepsilon(\cdot, \tau) - {\mathbf w}_\varepsilon(\cdot,0)\bigr) - 
\bigl(I +(\widetilde{\eta}^\varepsilon  (\eta^0)^{-1} - {\mathbf 1})\Pi_\varepsilon  \bigr) \bigl( {\mathbf w}_0(\cdot, \tau) - {\mathbf w}_0(\cdot, 0) \bigr) \bigr\|_{L_2({\mathbb R}^3)}
\\
\le {\mathfrak C}_{24}(s) |\tau|(1+ |\tau|)^{s/3} \varepsilon^{2s/3} \| {\mathbf f} \|_{H^{1+s}({\mathbb R}^3)}.
\end{split}
\end{align*}
Under Condition \emph{\ref{cond1}}, the constants $\mathfrak{C}_{21}(s), \mathfrak{C}_{22}(s), \mathfrak{C}_{23}(s), \mathfrak{C}_{24}(s)$ depend on  
$|\mu_0|$, $|\mu_0^{-1}|$, $\|\eta\|_{L_\infty}$, $\|\eta^{-1}\|_{L_\infty}$, the parameters of the  lattice $\Gamma$, and $s$. Under Condition \emph{\ref{cond2}}, these constants depend also on  $c^\circ$.
\end{theorem}

\begin{remark}
$1^\circ$. In the estimates from Theorems \emph{\ref{thrm_51}}, \emph{\ref{thrm_52}}, \emph{\ref{thrm_53}}, \emph{\ref{thrm_54}},  the norm $\| {\mathbf f} \|_{H^s}$ can be replaced by $\| \operatorname{curl} {\mathbf f} \|_{H^{s-1}}$,
because these theorems are deduced from the results for problem \eqref{51} with the initial data  
 $\boldsymbol{\psi} =- \operatorname{curl} {\mathbf f}$.
 
  $2^\circ$. Tracking the dependence of estimates on  $\tau$ allows us to get qualified estimates for small  $\varepsilon$ and large $\tau$. Under the assumptions of Theorem \emph{\ref{thrm_51}($1^\circ$)} we have
 $$
\begin{aligned}
 &\| {\mathbf v}_\varepsilon(\cdot, \tau) - {\mathbf v}_0(\cdot, \tau) \|_{L_2} = O(\varepsilon^{1-\alpha}), 
 \\
& \| {\mathbf z}_\varepsilon(\cdot, \tau) - {\mathbf z}_0(\cdot, \tau) \|_{L_2} = O(\varepsilon^{1-\alpha}),
 \\
&\qquad \text{for}\  \tau = O(\varepsilon^{-\alpha}),\ 0< \alpha <1.
 \end{aligned}
 $$
 Under the assumptions of Theorem \emph{\ref{thrm_51}($2^\circ$)} we have 
 $$
 \begin{aligned}
 &\| {\mathbf v}_\varepsilon(\cdot, \tau) - {\mathbf v}_0(\cdot, \tau) \|_{L_2} = O(\varepsilon^{(1-\alpha)s/2}),
 \\ 
&\| {\mathbf z}_\varepsilon(\cdot, \tau) - {\mathbf z}_0(\cdot, \tau) \|_{L_2} = O(\varepsilon^{(1-\alpha)s/2}),
\\
&\qquad \text{for}\  \tau = O(\varepsilon^{-\alpha}),\ 0< \alpha <1.
\end{aligned}
 $$
 Under the assumptions of Theorem \emph{\ref{thrm_52}($1^\circ$)} we have 
 $$
\begin{aligned}
 &\| {\mathbf v}_\varepsilon(\cdot, \tau) - {\mathbf v}_0(\cdot, \tau) \|_{L_2} = O(\varepsilon^{1-\alpha/2}), 
 \\
 &\| {\mathbf z}_\varepsilon(\cdot, \tau) - {\mathbf z}_0(\cdot, \tau) \|_{L_2} = O(\varepsilon^{1-\alpha/2}),
 \\
&\qquad \text{for}\  \tau = O(\varepsilon^{-\alpha}),\ 0< \alpha <2.
 \end{aligned}
 $$
 Under the assumptions of Theorem \emph{\ref{thrm_52}($2^\circ$)} we have 
 $$
 \begin{aligned}
 &\| {\mathbf v}_\varepsilon(\cdot, \tau) - {\mathbf v}_0(\cdot, \tau) \|_{L_2} = O(\varepsilon^{(2-\alpha) s/3}), 
 \\ 
&\| {\mathbf z}_\varepsilon(\cdot, \tau) - {\mathbf z}_0(\cdot, \tau) \|_{L_2} = O(\varepsilon^{(2-\alpha)s/3}),
\\
&\qquad \text{for}\  \tau = O(\varepsilon^{-\alpha}),\ 0< \alpha <2.
\end{aligned}
 $$
 Under the assumptions of Theorem \emph{\ref{thrm_53}($1^\circ$)} for ${\mathbf v}_\varepsilon$ and ${\mathbf z}_\varepsilon$ we have $$
\begin{aligned}
&\bigl\| {\mathbf v}_\varepsilon(\cdot, \tau) - {\mathbf v}_0(\cdot, \tau) - \varepsilon \mu_0^{-1} \Psi^\varepsilon
 \operatorname{curl}  {\mathbf v}_0(\cdot, \tau)
\bigr\|_{H^1({\mathbb R}^3)} 
= O(\varepsilon^{1- \alpha}),
\\
& \bigl\| {\mathbf z}_\varepsilon(\cdot, \tau) - {\mathbf z}_0(\cdot, \tau) - \varepsilon \Psi^\varepsilon 
\operatorname{curl} {\mathbf v}_0(\cdot, \tau) \bigr\|_{H^1({\mathbb R}^3)}
= O(\varepsilon^{1- \alpha}), 
\\
&\| (\eta^\varepsilon)^{-1} \operatorname{curl} {\mathbf v}_\varepsilon(\cdot, \tau) - ((\eta^0)^{-1} + \Sigma^\varepsilon) \operatorname{curl} {\mathbf v}_0(\cdot, \tau)\|_{L_2({\mathbb R}^3)}=
O(\varepsilon^{1- \alpha}),
\\
&\qquad \text{for}\  \tau = O(\varepsilon^{-\alpha}),\ 0< \alpha <1.
\end{aligned}
$$
Under the assumptions of Theorem \emph{\ref{thrm_53}($1^\circ$)} for ${\mathbf u}_\varepsilon$ and 
${\mathbf w}_\varepsilon$ we have 
$$
\begin{aligned}
& \bigl\| \bigl({\mathbf u}_\varepsilon(\cdot, \tau) - {\mathbf u}_\varepsilon (\cdot,0)\bigr) - 
\bigl({\mathbf 1} + \Sigma_\circ^\varepsilon  \bigr) \bigl({\mathbf u}_0(\cdot, \tau) - {\mathbf u}_0(\cdot, 0) \bigr) \bigr\|_{L_2({\mathbb R}^3)}
= O(\varepsilon^{1- 2 \alpha}), 
\\
& \bigl\| \bigl( {\mathbf w}_\varepsilon (\cdot, \tau) - {\mathbf w}_\varepsilon (\cdot,0)\bigr) - 
\widetilde{\eta}^\varepsilon  (\eta^0)^{-1} \bigl( {\mathbf w}_0(\cdot, \tau) - {\mathbf w}_0(\cdot, 0) \bigr) \bigr\|_{L_2({\mathbb R}^3)}
= O(\varepsilon^{1- 2 \alpha}), 
\\ 
&\qquad \text{for}\  \tau = O(\varepsilon^{-\alpha}),\ 0< \alpha <1/2.
\end{aligned}
$$
Under the assumptions of Theorem \emph{\ref{thrm_53}($2^\circ$)} for ${\mathbf v}_\varepsilon$ and ${\mathbf z}_\varepsilon$ we have
$$
\begin{aligned}
& \bigl\| {\mathbf D} \bigl( {\mathbf v}_\varepsilon(\cdot, \tau) - {\mathbf v}_0(\cdot, \tau) - \varepsilon \mu_0^{-1} \Psi^\varepsilon \Pi_\varepsilon \operatorname{curl} {\mathbf v}_0(\cdot, \tau) \bigr)\bigr\|_{L_2({\mathbb R}^3)}  = O(\varepsilon^{(1- \alpha)s/2}), 
\\
& \bigl\| {\mathbf D} \bigl( {\mathbf z}_\varepsilon(\cdot, \tau) - {\mathbf z}_0(\cdot, \tau) - \varepsilon \Psi^\varepsilon \Pi_\varepsilon  \operatorname{curl} {\mathbf v}_0(\cdot, \tau) \bigr) 
\bigr\|_{L_2({\mathbb R}^3)}
= O(\varepsilon^{(1- \alpha)s/2}),
\\
& \bigl\| (\eta^\varepsilon)^{-1} \operatorname{curl} {\mathbf v}_\varepsilon(\cdot, \tau) - ((\eta^0)^{-1} + \Sigma^\varepsilon  \Pi_\varepsilon) \operatorname{curl} {\mathbf v}_0(\cdot, \tau)
\bigr\|_{L_2({\mathbb R}^3)}=
O(\varepsilon^{(1- \alpha)s/2}), 
\\
&\qquad \text{for}\  \tau = O(\varepsilon^{-\alpha}),\ 0< \alpha <1.
\end{aligned}
$$
Under the assumptions of Theorem \emph{\ref{thrm_53}($2^\circ$)} for ${\mathbf u}_\varepsilon$ and 
${\mathbf w}_\varepsilon$ we have
$$
\begin{aligned}
& \bigl\| \bigl({\mathbf u}_\varepsilon(\cdot, \tau) - {\mathbf u}_\varepsilon(\cdot,0)\bigr) - 
\bigl({\mathbf 1} + \Sigma_\circ^\varepsilon \bigr) \bigl({\mathbf u}_0(\cdot, \tau) - {\mathbf u}_0(\cdot, 0) \bigr) \bigr\|_{L_2({\mathbb R}^3)}
= O(\varepsilon^{(1-  \alpha)s/2 - \alpha}), 
\\
& \bigl\| \bigl( {\mathbf w}_\varepsilon(\cdot, \tau) - {\mathbf w}_\varepsilon(\cdot,0)\bigr) - 
\widetilde{\eta}^\varepsilon (\eta^0)^{-1} \bigl( {\mathbf w}_0(\cdot, \tau) - {\mathbf w}_0(\cdot, 0) \bigr) \bigr\|_{L_2({\mathbb R}^3)}
= O(\varepsilon^{(1- \alpha)s/2 - \alpha}), 
\\ 
&\qquad \text{for}\  \tau = O(\varepsilon^{-\alpha}),\ 0< \alpha <\frac{s}{s+2}.
\end{aligned}
$$
Under the assumptions of Theorem \emph{\ref{thrm_54}}$(1^\circ)$ for ${\mathbf v}_\varepsilon$ and ${\mathbf z}_\varepsilon$ we have
$$
\begin{aligned}
& \bigl\| {\mathbf D} \bigl( {\mathbf v}_\varepsilon (\cdot, \tau) - {\mathbf v}_0(\cdot, \tau) - \varepsilon \mu_0^{-1} \Psi^\varepsilon \Pi_\varepsilon \operatorname{curl} {\mathbf v}_0(\cdot, \tau) \bigr)\bigr\|_{L_2({\mathbb R}^3)}  = O( \varepsilon^{1- \alpha/2}), 
\\
&\bigl\| {\mathbf D} \bigl( {\mathbf z}_\varepsilon(\cdot, \tau) - {\mathbf z}_0(\cdot, \tau) - \varepsilon \Psi^\varepsilon \Pi_\varepsilon \operatorname{curl} {\mathbf v}_0(\cdot, \tau) \bigr) \bigr\|_{L_2({\mathbb R}^3)}
= O(\varepsilon^{1- \alpha/2}),
\\
&\bigl\| (\eta^\varepsilon)^{-1} \operatorname{curl} {\mathbf v}_\varepsilon(\cdot, \tau) - ((\eta^0)^{-1} + \Sigma^\varepsilon \Pi_\varepsilon) \operatorname{curl} {\mathbf v}_0(\cdot, \tau)\bigr\|_{L_2({\mathbb R}^3)}=
O(\varepsilon^{1- \alpha /2}), 
\\
&\qquad \text{for}\  \tau = O(\varepsilon^{-\alpha}),\ 0< \alpha <2.
\end{aligned}
$$
Under the assumptions of Theorem \emph{\ref{thrm_54}($1^\circ$)}  for ${\mathbf u}_\varepsilon$ and 
${\mathbf w}_\varepsilon$ we have
$$
\begin{aligned}
&\bigl\| \bigl({\mathbf u}_\varepsilon(\cdot, \tau) - {\mathbf u}_\varepsilon(\cdot,0)\bigr) - 
\bigl({\mathbf 1} + \Sigma_\circ^\varepsilon \bigr) \bigl({\mathbf u}_0(\cdot, \tau) - {\mathbf u}_0(\cdot, 0) \bigr) \bigr\|_{L_2({\mathbb R}^3)}
= O(\varepsilon^{1- 3 \alpha/2}), 
\\
& \bigl\| \bigl( {\mathbf w}_\varepsilon(\cdot, \tau) - {\mathbf w}_\varepsilon(\cdot,0)\bigr) - 
\widetilde{\eta}^\varepsilon (\eta^0)^{-1} \bigl( {\mathbf w}_0(\cdot, \tau) - {\mathbf w}_0(\cdot, 0) \bigr) \bigr\|_{L_2({\mathbb R}^3)}
= O(\varepsilon^{1- 3 \alpha/2}), 
\\ 
&\qquad \text{for}\  \tau = O(\varepsilon^{-\alpha}),\ 0< \alpha <2/3.
\end{aligned}
$$
Under the assumptions of Theorem \emph{\ref{thrm_54}($2^\circ$)} for  ${\mathbf v}_\varepsilon$ and ${\mathbf z}_\varepsilon$ we have
$$
\begin{aligned}
& \bigl\| {\mathbf D} \bigl( {\mathbf v}_\varepsilon(\cdot, \tau) - {\mathbf v}_0(\cdot, \tau) - \varepsilon \mu_0^{-1} \Psi^\varepsilon \Pi_\varepsilon \operatorname{curl} {\mathbf v}_0(\cdot, \tau) \bigr)\bigr\|_{L_2({\mathbb R}^3)}  = O(\varepsilon^{(2- \alpha)s/3}), 
\\
& \bigl\| {\mathbf D} \bigl( {\mathbf z}_\varepsilon(\cdot, \tau) - {\mathbf z}_0(\cdot, \tau) -\varepsilon \Psi^\varepsilon \Pi_\varepsilon \operatorname{curl} {\mathbf v}_0(\cdot, \tau) \bigr)
\bigr\|_{L_2({\mathbb R}^3)}
= O(\varepsilon^{(2- \alpha)s/3}),
\\
& \bigl\| (\eta^\varepsilon)^{-1} \operatorname{curl} {\mathbf v}_\varepsilon(\cdot, \tau) - ((\eta^0)^{-1} + \Sigma^\varepsilon \Pi_\varepsilon) \operatorname{curl} {\mathbf v}_0(\cdot, \tau) \bigr\|_{L_2({\mathbb R}^3)}=
O(\varepsilon^{(2- \alpha)s/3}), 
\\
&\qquad \text{for}\  \tau = O(\varepsilon^{-\alpha}),\ 0< \alpha <2.
\end{aligned}
$$
Under the assumptions of Theorem \emph{\ref{thrm_54}($2^\circ$)} for ${\mathbf u}_\varepsilon$ and 
${\mathbf w}_\varepsilon$ we have 
$$
\begin{aligned}
&\bigl\| \bigl({\mathbf u}_\varepsilon(\cdot, \tau) - {\mathbf u}_\varepsilon(\cdot,0)\bigr) - 
\bigl( {\mathbf 1} + \Sigma_\circ^\varepsilon \bigr) \bigl({\mathbf u}_0(\cdot, \tau) - {\mathbf u}_0(\cdot, 0) \bigr) \bigr\|_{L_2({\mathbb R}^3)}
= O(\varepsilon^{(2-  \alpha)s/3 - \alpha}), 
\\
& \bigl\| \bigl({\mathbf w}_\varepsilon(\cdot, \tau) - {\mathbf w}_\varepsilon(\cdot,0)\bigr) - 
\widetilde{\eta}^\varepsilon (\eta^0)^{-1} \bigl({\mathbf w}_0(\cdot, \tau) - {\mathbf w}_0(\cdot, 0) \bigr) \bigr\|_{L_2({\mathbb R}^3)}
= O(\varepsilon^{(2- \alpha)s/3 - \alpha}), 
\\ 
&\qquad \text{for}\  \tau = O(\varepsilon^{-\alpha}),\ 0< \alpha <\frac{2s}{s+3}.
\end{aligned}
$$
\end{remark}

\end{document}